\documentclass[12pt]{amsart}
\usepackage{amssymb}
\usepackage{mathrsfs}
\usepackage{graphicx}
\usepackage{tikz-cd}
\usepackage{url}
\usepackage{appendix}
\usepackage[all]{xy}

\allowdisplaybreaks

\newtheorem{theorem}[equation]{Theorem}
\newtheorem{lemma}[equation]{Lemma}
\newtheorem{proposition}[equation]{Proposition}
\newtheorem{corollary}[equation]{Corollary}

\theoremstyle{definition}
\newtheorem{definition}[equation]{Definition}
\newtheorem{example}[equation]{Example}

\theoremstyle{remark}
\newtheorem{remark}[equation]{Remark}

\numberwithin{equation}{section}

\newcommand{\ZZ}{\mathbb{Z}}
\newcommand{\QQ}{\mathbb{Q}}
\newcommand{\RR}{\mathbb{R}}

\newcommand{\TT}{\mathbb{T}}

\newcommand{\CC}{\mathbb{C}}

\newcommand{\TTs}{\mathbb{T}_s}

\DeclareMathOperator{\Li}{Li}

\DeclareMathOperator{\ad}{ad}
\DeclareMathOperator{\GL}{GL}
\DeclareMathOperator{\Mat}{Mat}

\DeclareMathOperator{\dep}{dep}
\DeclareMathOperator{\wght}{wght}
\DeclareMathOperator{\proj}{proj}

\DeclareMathOperator{\ord}{ord}

\DeclareMathOperator{\Id}{Id}

\DeclareMathOperator{\Ext}{Ext}

\newcommand{\dnorm}[1]{\lVert #1 \rVert_{\infty}}

\newcommand{\inorm}[1]{{\lvert #1 \rvert}_{\infty}}

\begin{document}
	
	\title[Deformation of Multiple Zeta Values and Their Logarithmic Interpretation]{Deformation of Multiple Zeta Values and Their Logarithmic Interpretation in Positive Characteristic}
	
	\author{O\u{g}uz Gezm\.{i}\c{s}}

\thanks{This project was  supported by MoST Grant 109-2811-M-007-553.}


	\address{Department of Mathematics, National Tsing Hua University, Hsinchu City 30042, Taiwan R.O.C.}
	\email{gezmis@math.nthu.edu.tw}
	
	\maketitle
	
\begin{abstract}
	Pellarin introduced the deformation of multiple zeta values of Thakur as elements over Tate algebras.
	In this paper, we relate these values to a certain coordinate of the logarithm of a higher dimensional Drinfeld module over the Tate algebra which we will introduce. Moreover, we define multiple polylogarithms in our setting and represent deformation of multiple zeta values as a linear combination of multiple polylogarithms. As an application of our results, we also give Dirichlet-Goss multiple $L$-values as a linear combination of twisted multiple polylogarithms at algebraic points.
\end{abstract}
	
\section{Introduction}
\subsection{Background} Multiple zeta values were introduced by Euler as the infinite sum
\[
\zeta(s_1,\dots,s_r):=\sum_{\substack{n_1>n_2>\dots>n_r> 0\\n_1,\dots,n_r\in \ZZ_{\geq 1}}}\frac{1}{n_1^{s_1}\dots n_r^{s_r}}\in \RR
\]
for positive integers $s_1,\dots,s_r$ such that $s_1>1$. These values can be seen as a generalization of special values $\zeta(n)$ of Riemann zeta function for a positive integer $n>1$. Their motivic interpretation is given by Terasoma \cite{T02} and Goncharov independently.  Moreover for the tuple $\mathfrak{s}=(s_1,\dots,s_r)$, the multiple polylogarithm  $\Li_{\mathfrak{s}}(z_1,\dots,z_r)$ is defined by
\[
\Li_{\mathfrak{s}}(z_1,\dots,z_r):=\sum_{\substack{n_1>n_2>\dots>n_r> 0\\n_1,\dots,n_r\in \ZZ_{\geq 1}}}\frac{z_1^{n_1}\dots z_r^{n_r}}{n_1^{s_1}\dots n_r^{s_r}}\in \QQ[[z_1,\dots,z_r]],
\]
and its specialization at $z_1=\dots=z_r=1$ gives the value of $\zeta(s_1,\dots,s_r)$. We refer the reader to \cite{Waldschmidt} and \cite{Zhaou} for interesting properties of those objects.

In this paper, we are interested in the function field analogue of multiple zeta values and their deformation in positive characteristic. Let $q$ be a power of a prime $p$. We let $\mathbb{F}_q$ be the finite field with $q$ elements. We set $A:=\mathbb{F}_q[\theta]$ as the polynomial ring in the variable $\theta$ with coefficients from $\mathbb{F}_q$, and $A_{+}$ as the set of monic polynomials of $A$. We let $K$ be the function field $\mathbb{F}_q(\theta)$ and $\ord$ be the valuation corresponding to the infinite place normalized so that $\ord(\theta)=-1$. Moreover, we define the norm $\inorm{\cdot}$ corresponding to $\ord$ so that $\inorm{\theta}=q$. We also let $K^{\text{perf}}$ be the perfect closure of $K$ and $\overline{K}$ be the algebraic closure of $K$. The completion of $K$ with respect to $\inorm{\cdot}$ is denoted by $\mathbb{K}_{\infty}$ and the completion of an algebraic closure of $\mathbb{K}_{\infty}$ is denoted by $\CC_{\infty}$. 

We define the Carlitz-Goss zeta value $\zeta_{A}(n)$ at a positive integer $n$ by the infinite series
\[
\zeta_{A}(n)=\sum_{a\in A_{+}}\frac{1}{a^n}\in \mathbb{K}_{\infty}^{\times},
\]
which can be seen as a function field analogue of $\zeta(n)$. The arithmetic of these special values were studied by Carlitz \cite{Carlitz}, Gekeler \cite{Gekeler}, Goss \cite{Goss} and Thakur \cite{Thakur}. Also their transcendental behavior over $K$ was discovered by Chang  and Yu \cite{CY} and Yu \cite{Yu}. 

Let $\mathfrak{s}=(s_1,\dots,s_r)$ be a tuple in $\ZZ_{\geq 1}^r$ for some positive integer $r$  and set $w:=\sum s_i$. Then the multiple zeta value $\zeta_{A}(\mathfrak{s})$ of weight $w$ and depth $r$ is defined by Thakur in \cite[Sec. 5.10]{Thakur2} as the infinite sum
\[
\zeta_{A}(\mathfrak{s}):=\sum_{\substack{\inorm{a_1}>\inorm{a_2}>\dots >\inorm{a_r}\geq 0 \\ a_1,\dots,a_r\in A_{+}}} \frac{1}{a_1^{s_1}\dots a_r^{s_r}}\in \mathbb{K}_{\infty}.
\]
In 2009, Thakur \cite[Thm. 4]{Thakur3} proved that $\zeta_{A}(\mathfrak{s})$ is non-zero. Furthermore, Anderson and Thakur \cite{AT09} give the realization of multiple zeta values as periods of a certain $t$-motive (see \cite[\S 4]{BPrapid} for more details on $t$-motives). 

In 2014, Chang \cite{Chang} defined the multiple polylogarithm $\Li_{\mathfrak{s}}(z_1,\dots,z_r)$ by 
\begin{equation}\label{CMPLL}
\Li_{\mathfrak{s}}(z_1,\dots,z_r):=\sum_{i_1> i_2> \dots>i_r\geq 0} \frac{z_1^{q^{i_1}}\dots z_r^{q^{i_r}}}{\ell_{i_1}^{s_1}\dots \ell_{i_r}^{s_r}}\in K[[z_1,\dots,z_r]],
\end{equation}
where $\ell_{i}:=(\theta-\theta^q)\dots (\theta-\theta^{q^i})$ for $i\geq 0$ and $\ell_{0}:=1$. When $r=1$, it becomes the Carlitz $n$-th polylogarithm 
\[
\log_{n}(z):=\sum_{i=0}^{\infty}\frac{z^{q^{i}}}{\ell_{i}^n}\in K[[z]]
\]
defined by Anderson and Thakur \cite{AndThak90}. Moreover, Anderson and Thakur \cite{AndThak90} represents $\zeta_{A}(n)$ as a $K$-linear combination of $\log_{n}(\theta^{j})$ where $j<nq/(q-1)$. 

Unlike classical case, relating multiple zeta values to multiple polylogarithms is not trivial in function field setting. Using $t$-motivic interpretation of multiple zeta values in \cite{AT09}, Chang \cite{Chang} clarified this phenomenon for higher depths stating that there exist tuples $(a_j,(u_{j1},\dots,u_{jr}))\in A\times A^r$ where $j$ is in a finite index set $J$, and $\Gamma_{\mathfrak{s}}\in A$, which all can be explicitly defined, such that 
\begin{equation}\label{E:MZV2}
\Gamma_{\mathfrak{s}}\zeta_{A}(\mathfrak{s})=\sum_{j\in J}a_j\Li_{\mathfrak{s}}(u_{j1},\dots,u_{jr}).
\end{equation}

Later Chang and Mishiba \cite[Thm. 1.4.1]{ChangMishibaOct} related $\zeta_{A}(\mathfrak{s})$ to a certain  coordinate of the logarithm of a $t$-module (see \cite{And86} for details on $t$-modules) by proving that there exist a uniformizable $t$-module $G_{\mathfrak{s}}$ of dimension $k_{\mathfrak{s}}$ defined over $K$, a special point $v_{\mathfrak{s}}\in G_{\mathfrak{s}}(K)$ and an element $Z_{\mathfrak{s}}\in G_{\mathfrak{s}}(\mathbb{K}_{\infty})$ such that $\Gamma_{\mathfrak{s}}\zeta_{A}(\mathfrak{s})$ occurs as the $w$-th coordinate of $Z_{\mathfrak{s}}$ and $\exp_{G_{\mathfrak{s}}}(Z_{\mathfrak{s}})=v_{\mathfrak{s}}.$ 
Thus, the logarithmic interpretation of multiple zeta values allow them to verify the function field analogue of Furusho's conjecture (see \cite{F06} and  \cite{F07}). 
\subsection{Tate algebras}
Let $U\subset \ZZ_{\geq 1}$ be any finite set and let $\TT_{U}$ be the Tate algebra on the closed unit polydisc over $\CC_{\infty}$ with independent variables $t_i$ for $i\in U$. Let $\Sigma\subset \ZZ_{\geq 1}$ be a finite union of finite sets $U_i\subset \ZZ_{\geq 1}$.  The Frobenius automorphism $\tau :\TT_{\Sigma}\to \TT_{\Sigma}$ is given by raising the coefficients of the given infinite series to their $q$-th power and fixing the independent variables $t_i$ (see \S 2.1 for details). Furthermore, for any $\mathbb{F}_q$-algebra $R$, set $R[\underline{t}_{U}]:=R[t_i : i\in U]$ and  $R(\underline{t}_{U})$ to be the fraction field of the polynomial ring $R[\underline{t}_{U}]$.

For a moment let us concentrate on $\TT_{\Sigma}$ where $\Sigma=\{1,\dots,n\}$. In 2012, Pellarin \cite{Pellarin0} defined the following $L$-series
\begin{equation}\label{E:Pellarin}
\zeta_{C}\binom{\Sigma}{s}:=\sum_{a\in A_{+}}\frac{a(t_1)\dots a(t_n)}{a^s}\in \TT_{\Sigma}^{\times}
\end{equation}
for some positive integer $s$ as a deformation of Carlitz-Goss zeta value $\zeta_{A}(s)$. Similar notion of deformation has been also carried to multiple zeta values of Thakur by Pellarin in \cite{PellarinMZV} and \cite{PellarinMZV2} as follows: For any $a\in A_{+}$ and an independent variable $t$, we set $a(t):=a_{|\theta=t}$. We now define the map $\sigma_{U}:A_{+}\to \mathbb{F}_q[\underline{t}_{\Sigma}]$ by $\sigma_{U}(a):=1$ if $U=\emptyset$ and  $\sigma_{U}(a):=\prod_{i\in U}a(t_i)\in \mathbb{F}_q[\underline{t}_{\Sigma}]$ otherwise. For some $r\in \mathbb{Z}_{\geq 1}$, we call
\begin{equation}\label{E:star00}
\mathcal{C}=\binom{U_1, \dots ,U_r}{s_1, \dots, s_r}
\end{equation}
a composition array of weight $w=:\wght(\mathcal{C})$ and depth $r:=\dep(\mathcal{C})$. Now we define
\begin{equation}\label{E:deformmzv0}
\zeta_{C}(\mathcal{C}):=\sum_{\substack{\inorm{a_1}>\inorm{a_2}>\dots >\inorm{a_r}\geq 0 \\ a_1,\dots,a_r\in A_{+}}}\frac{\sigma_{U_1}(a_1)\sigma_{U_2}(a_2)\dots \sigma_{U_r}(a_r)}{a_1^{s_1}a_2^{s_2}\dots a_{r}^{s_r}}\in \TT_{\Sigma},
\end{equation}
to be the multiple zeta value corresponding to $\mathcal{C}$. Observe that when 
\begin{equation}\label{E:array1}
\mathcal{C}=\binom{\emptyset,\dots,\emptyset}{s_1,\dots,s_r},
\end{equation}
we see that $\zeta_{C}(\mathcal{C})=\zeta_{A}(\mathfrak{s})$. Using the non-vanishing of $\zeta_{A}(\mathfrak{s})$ and a specialization argument, Pellarin \cite[Prop.  3]{PellarinMZV2} proved that the multiple zeta values $\zeta_{C}(\mathcal{C})$ are non-zero as elements of $\TT_{\Sigma}$. 

For any $i,j\geq 1$, we define the element $b_{i}(t_j):=\prod_{k=0}^{i-1}(t_j-\theta^{q^k})\in A[\underline{t}_{\Sigma}]$ and $b_{0}(t_j):=1$. We also let $b_i(U):=1$ if $U= \emptyset$ and  $b_i(U):=\prod_{j\in U}b_{i}(t_j)$ otherwise. For some tuple $(u_1,\dots,u_r)\in \TT_{\Sigma}^{r}$ living in a certain subset of $\TT_{\Sigma}$ (see \S2.3 for details), we define the multiple polylogarithm $\Li_{\mathcal{C}}(u_1,\dots,u_r)$ by the infinite series
\[
\Li_{\mathcal{C}}(u_1,\dots,u_r):=\sum_{i_1>i_2>\dots>i_r\geq 0}\frac{b_{i_1}(U_1)\dots b_{i_r}(U_r)\tau^{i_1}(u_1)\dots \tau^{i_r}(u_r)}{\ell_{i_1}^{s_1}\dots \ell_{i_r}^{s_r}}\in \TT_{\Sigma}.
\]
Our first result (stated as Theorem \ref{T:polylogarithm} later) is as follows.
\begin{theorem}\label{T:11}
	For any composition array $\mathcal{C}$ of depth $r$ defined as in \eqref{E:star00}, there exist tuples $(a_j,(u_{j1},\dots,u_{jr}))\in A\times K^{\text{perf}}(\underline{t}_{\Sigma})^r$ where $j$ is in a finite index set $\mathfrak{I}$, and $\Gamma_{\mathcal{C}}\in A[\underline{t}_{\Sigma}]$, which all can be explicitly defined, such that
	\[
	\Gamma_{\mathcal{C}} \zeta_{C}(\mathcal{C})=\sum_{j\in \mathfrak{I}}a_j\Li_{\mathcal{C}}(u_{j1},\dots,u_{jr}).
	\]
\end{theorem}
Note that Theorem \ref{T:11} can be seen as a generalization of Chang's identity \eqref{E:MZV2} and that identity follows from our result by choosing the composition array $\mathcal{C}$ as in \eqref{E:array1} (see \S 2.3 for details).

\subsection{Dirichlet-Goss multiple $L$-values} We denote the algebraic closure of $\mathbb{F}_q$ by $\overline{\mathbb{F}}_q$. For any $1\leq i \leq r$, let  $\chi_i:A\to \overline{\mathbb{F}}_q$ be a Dirichlet character. Furthermore let $(s_1,\dots,s_r)$ be a tuple of positive integers. We define the Dirichlet-Goss multiple $L$-value $L(\chi_1,\dots,\chi_r;s_1,\dots,s_r)$ of weight $w$ and depth $r$ by the infinite series
\[
L(\chi_1,\dots,\chi_r;s_1,\dots,s_r):=\sum_{\substack{\inorm{a_1}>\inorm{a_2}>\dots >\inorm{a_r}\geq 0 \\ a_1,\dots,a_r\in A_{+}}}\frac{\chi_1(a_1)\dots \chi_r(a_r)}{a_1^{s_1}\dots a_r^{s_r}}\in \mathbb{C}_{\infty}.
\]

One of the advantages of studying $\zeta_{C}(\mathcal{C})$ is to be able to deduce some properties of Dirichlet-Goss multiple $L$-values. Consider the composition array $\mathcal{C}$ defined as in \eqref{E:star00} with pairwise disjoint sets $U_i$ such that $\Sigma=\sqcup_{i=1}^r U_i$.
For any $1\leq i \leq r$ and $j\in U_i$, let $\mathfrak{p}_{ij}\in A_{+}$ be the minimal polynomial of $\xi_{i,j}\in \overline{\mathbb{F}}_q$. Since $\zeta_{C}(\mathcal{C})$ converges in $\TT_{\Sigma}$, evaluating \eqref{E:deformmzv0} at $t_j=\xi_{ij}$ produces the Dirichlet-Goss multiple $L$-value $L(\chi_1,\dots,\chi_r;s_1,\dots,s_r)$ such that $\chi_i:A\to \overline{\mathbb{F}}_q$ is the map sending $a\in A$ to $\prod_{j\in U_i}a(\xi_{ij})$ which is actually the Dirichlet character modulo the ideal generated by $\prod_{j\in U_i}\mathfrak{p}_{ij}$ in $A$.

For any $j\in \ZZ_{\geq 0}$, we set elements $B_{\chi_i,j}\in \CC_{\infty}^{\times}$ corresponding to the character $\chi_i$ for all $1\leq i \leq r$ (see \S2.4 for details). We define
\begin{multline}
\Li_{\big(\substack{\chi_1,\dots,\chi_r\\s_1,\dots,s_r}\big)}(z_{1},\dots,z_{r})\\
\ \ \ \  \ :=\sum_{i_1>i_2>\dots>i_r\geq 0}\frac{B_{\chi_1,i_1}\dots B_{\chi_r,i_r}}{\ell_{i_1}^{s_1}\dots \ell_{i_r}^{s_r}}z_1^{q^{i_1}}\dots z_{r}^{q^{i_r}}\in \CC_{\infty}[[z_1,\dots,z_r]].
\end{multline}

Let $\mathfrak{p}$ be irreducible in $A_{+}$ of degree $d$ and $\lambda_{\mathfrak{p}}\in \CC_{\infty}^{\times}$ be a $\mathfrak{p}$-torsion point. Let $K_{\mathfrak{p}}:=K(\lambda_{\mathfrak{p}})$ be the $\mathfrak{p}$-th cyclotomic field extension of $K$ and $\Delta_{\mathfrak{p}}$ be its Galois group (see  \cite[Chap. 12]{Ros} for the details of cyclotomic field extensions over function fields). Consider the unique group isomorphism ${v}_{\mathfrak{p}}:\Delta_{\mathfrak{p}}\to \mathbb{F}_{q^d}^{\times}$ induced by the Teichm\"uller character corresponding to a fixed choice of a root $\xi_{\mathfrak{p}}$ of $\mathfrak{p}$ and let $g(v_{\mathfrak{p}})$ be the Gauss-Thakur sum (see \S2.4 and \cite{AP15} for the details). We obtain the following corollary of Theorem \ref{T:11}.
\begin{corollary}\label{C:DG} Fix a positive integer $r$. For any $1\leq i \leq r$ and  $j_1,\dots,j_r\in \mathbb{Z}_{\geq 1}$, 
	let $\xi_{i1},\dots,\xi_{ij_i}$ be elements in $\overline{\mathbb{F}}_q$ whose minimal polynomials are $\mathfrak{p}_{i1},\dots,\mathfrak{p}_{ij_i}$ respectively and define the Dirichlet character $\chi_i:A\to \overline{\mathbb{F}}_q$  given by $\chi_i(a)=a(\xi_{i1})\dots a(\xi_{ij_i})$. Then there exist elements $b_j\in \overline{K}$ and $\eta_{ji}\in K^{\text{perf}}$ for $1\leq i \leq r$ where $j$ is in a finite index set $\mathfrak{I}$ such that 
	\begin{multline}
	\prod_{1\leq k \leq j_1}g(v_{\mathfrak{p}_{1k}})\dots \prod_{1\leq k \leq j_r}g(v_{\mathfrak{p}_{rk}})L(\chi_1,\dots,\chi_r;s_1,\dots,s_r)\\
	\ \ \ \ \  \ \ \ \ \ \ \  \ \ \ \ \ =\sum_{j\in \mathfrak{I}}b_j\Li_{\big(\substack{\chi_1,\dots,\chi_r\\s_1,\dots,s_r}\big)}(\eta_{j1},\dots,\eta_{jr}).
	\end{multline}
\end{corollary}

For a special class of Dirichlet characters, we can deduce more about Dirichlet-Goss multiple $L$-values and their transcendental properties. The next theorem will be proved in \S2.4.  
\begin{theorem}\label{T:linindep} For a finite set $U\subset \ZZ_{\geq 1}$ and a tuple $\xi=(\xi_i| \ \ i\in U)\in \mathbb{F}_q^{|U|}$ where $|U|$ is the cardinality of $U$, let 
	$
	\chi_{U,\xi}:A\to \mathbb{F}_q
	$
	be the Dirichlet character given by 
	\begin{equation}\label{E:char}
	\chi_{U,\xi}(a)=\prod_{i\in U}a(\xi_i).
	\end{equation}
	\begin{itemize}
		\item[(i)] Let $U_1,\dots,U_r$ be finite subsets of $\ZZ_{\geq 1}$,  $(s_1,\dots,s_r)$ a tuple of positive integers and $\xi_j=(\xi_{ji}| \ \ i\in U_j)\in \mathbb{F}_q^{|U_j|}$ for $1\leq j \leq r$. If $L(\chi_{U_1,\xi_1},\dots,\chi_{U_r,\xi_r};s_1,\dots,s_r)$ is non-zero, then it is transcendental over $\overline{K}$.
		\item[(ii)] Fix positive integers $m,j_1,\dots,j_m$. For any $1\leq i \leq m$ and $1\leq k \leq j_i$, let $U_{ik}$ be a finite subset of $\ZZ_{\geq 1}$, $\xi_{ik}=(\xi_{ikj}| \ \ j\in U_{ik})\in \mathbb{F}_q^{|U_{ik}|}$ and $(s_{i1},\dots,s_{ij_i})$ be a tuple of positive integers. Furthermore set $w_i=\sum_{b=1}^{j_i}s_{ib}$  and assume that $w_i\neq w_j$ if $i\neq j$ for $1\leq i,j\leq m$. If $L(\chi_{U_{i1},\xi_{i1}},\dots,\chi_{U_{ij_i},\xi_{ij_i}};s_{i1},\dots,s_{ij_i})$ is non-zero for each $1\leq i \leq m$, then the set $\{ L(\chi_{U_{i1},\xi_{i1}},\dots,\chi_{U_{ij_i},\xi_{ij_i}};s_{i1},\dots,s_{ij_i})| \  \ 1\leq i \leq m\}$ is $\overline{K}$-linearly independent. 
	\end{itemize}
	
\end{theorem}

\subsection{Anderson $A[\underline{t}_{\Sigma}]$-modules} Let $ \Mat_{n}(\TT_{\Sigma})[\tau]$ be the twisted polynomial ring in $\tau$ with coefficients in $\Mat_{n}(\TT_{\Sigma})$ (see \S2.1 for details). Inspired by \cite[\S 2]{Demeslay14}, we call an Anderson $A[\underline{t}_{\Sigma}]$-module $\phi:A[\underline{t}_{\Sigma}]\to \Mat_{n}(\TT_{\Sigma})[\tau]$ of dimension $n$ defined over $\TT_{\Sigma}$ as an $\mathbb{F}_q[\underline{t}_{\Sigma}]$-linear homomorphism given by
\[
\phi(\theta)=A_0+A_1\tau +\dots +A_s\tau^s
\]
for some $s$ and $(\theta \Id_{n}-A_0)^n=0$. We should highlight the fact that Angl\`{e}s, Pellarin and Tavares Ribeiro \cite{APTR} and Angl\`{e}s and Tavares Ribeiro \cite{AnglesTavaresRibeiro} have already studied $n=1$ case, called Drinfeld $A[\underline{t}_{\Sigma}]$-modules, due to their relation with log-algebraic identites, Taelman's class modules and Pellarin $L$-series.

In this paper, we focus on a special class of Anderson $A[\underline{t}_{\Sigma}]$-modules, given by
\begin{equation}\label{E:class}
\phi(\theta)=\theta \Id_{n}+N+E\tau
\end{equation}
where $N\in \Mat_{n}(\mathbb{F}_q)$ is a nilpotent matrix and $E\in \Mat_{n}(\TT_{\Sigma})$. For such $\phi$, similar to Anderson $t$-modules, we can assign an exponential function, which is a vector valued function denoted by
\[
\exp_{\phi}:\Mat_{n \times 1}(\TT_{\Sigma})\to \Mat_{n\times 1}(\TT_{\Sigma}),
\]
and we show that it has an infinite radius of convergence (see \S 3.1). We also call $\phi$ uniformizable if $\exp_{\phi}$ is a surjective function. Our next result (stated as Theorem \ref{T:result2} later) is as follows.
\begin{theorem}\label{T:22}
	For any composition array $\mathcal{C}$ of weight $w$, there exist a uniformizable Anderson $A[\underline{t}_{\Sigma}]$-module $G_{\mathcal{C}}$ of dimension $k_{\mathcal{C}}$ defined over $\TT_{\Sigma}$, a special point $v_{\mathcal{C}}\in K^{\text{perf}}(\underline{t}_{\Sigma})^{k_{\mathcal{C}}}$ and an element $Z_{\mathcal{C}}\in \TT_{\Sigma}^{k_{\mathcal{C}}}$, which all can be explicitly defined, such that
	\begin{itemize}
		\item [(i)] $\Gamma_{\mathcal{C}}\zeta_{C}(\mathcal{C})$ occurs as the $w$-th coordinate of $Z_{\mathcal{C}}$.
		\item[(ii)] $\exp_{G_{\mathcal{C}}}(Z_{\mathcal{C}})=v_{\mathcal{C}}$.
	\end{itemize}
	
\end{theorem}

%\begin{remark}It should be noted that relating special values to the coordinates of logarithms is important in transcendental number theory. In 1990, Anderson and Thakur \cite{AndThak90} showed the depth one case of Chang and Mishiba's result \cite[Thm. 1.4.1]{ChangMishibaOct} and thus related $A$-multiple of Carlitz-Goss zeta values to the coordinate of the logarithm of a certain $t$-module. Using their result, Yu \cite[Thm. 3.1]{Yu} was able to prove that $\zeta_{A}(n)$ is transcendental over $K$ for $n\geq 1$. Recently, a new interpretation of transcendency of elements in Tate algebras over the function field $\mathbb{F}_q(\theta,t_{\Sigma})$, where $\underline{t}_{\Sigma}$ represents the set of all $t_i$ with $i\in \Sigma$, has been introduced by Pellarin \cite{Pellarin3}. As a corollary of our results, we are also able to give the realization of Pellarin $L$-series as a coordinate of the logarithm of an Anderson $A[\underline{t}_{\Sigma}]$-module (see Example \ref{Ex:Pellarin}). It is an interesting question to analyze how Theorem \ref{T:22} can be used to deduce results about transcendental properties of Pellarin $L$-series as elements of Tate algebras.
%\end{remark}
\begin{remark}Let $\omega_i$ be the Anderson-Thakur element corresponding to $t_i-\theta$ and $\tilde{\pi}$ be the Carlitz period (see \S 2.1 for details). Assume that $\Sigma=\{1,\dots, n\}$ and $s$ is a positive integer. Finally we set $\alpha:=\prod_{i=1}^n (t_i-\theta)$. Using Anderson's ideas (see \cite[\S 4.5]{GP} for Drinfeld modules over Tate algebras), we can show that the generator of the kernel of $\exp_{C^{\otimes s}_{\alpha}}$ (see
	\S 3.2 for the definition of $C^{\otimes s}_{\alpha}$) is a vector whose last coordinate is given by $\frac{\tilde{\pi}^s}{\omega_1\dots \omega_n}$. Together with the use of Theorem \ref{T:22}, we are able to prove that if the point $Z_{\mathcal{C}}$ is an $A[\underline{t}_{\Sigma}]$-torsion point for $C^{\otimes s}_{\alpha}$, then $L(\chi_{t_1}\dots \chi_{t_n},s)$ is Eulerian in the sense that
	\[
	\omega_1\dots \omega_n\frac{L(\chi_{t_1}\dots \chi_{t_n},s)}{\tilde{\pi}^s}\in \mathbb{F}_q(\underline{t}_{\Sigma},\theta)
	\]
	where $\mathbb{F}_q(\underline{t}_{\Sigma},\theta)$ is the fraction field of the polynomial ring $\mathbb{F}_q[\underline{t}_{\Sigma}][\theta]$. The opposite direction is expected to hold but due to lack of an analogue of Yu's transcendence theory \cite{Yu} in our setting, it is still an open problem. We can also ask about the Eulerian criterion for $\zeta_{C}(\mathcal{C})$ in higher depth case similar to the criterion given in \cite{CPY}. In this case, calculations show that not all the generators of the kernel of $\exp_{G_{\mathcal{C}}}$ take $\tilde{\pi}^w$ times the inverse of Anderson-Thakur elements in their $w$-th coordinate and that causes difficulties even proving the direction we show in depth 1 case. One should also understand the generalized version \cite[Lem. 4.1.7]{CGM} of Yu's theorem \cite[Thm. 2.3]{Yu} in this setting. The author hopes to tackle this problem in the near future.
\end{remark}
\begin{remark} We remark that although Chang and Mishiba use dual $t$-motives and their fiber coproducts to prove uniformizability of $G_{\mathfrak{s}}$ in \cite[Thm. 1.4.1]{ChangMishibaOct}, we use a different method as it is still not clear how we should define the dual $t$-motives for Anderson $A[\underline{t}_{\Sigma}]$-modules. In our method, we first prove the uniformizability of Anderson $A[\underline{t}_{\Sigma}]$-modules given of the form \eqref{E:moduleG} by following ideas modified from \cite[\S 4]{GP}, and we construct $G_{\mathcal{C}}$ in \eqref{E:sec6m1} from those modules. Then using the map $\lambda$ defined in \eqref{E:lambdamap}, we prove that $G_{\mathcal{C}}$ is also uniformizable (Proposition \ref{P:sec6unif}).  
\end{remark}

\subsection{Outline of the Paper} The outline of the paper can be given as follows: In \S 2 we cover some necessary notation and background for the rest of the paper and recall recent developments on power sums. We prove Corollary \ref{C:DG} and continue to \S 2 by proving Theorem \ref{T:11}. Basically our method is to modify Chang's ideas in the proof of \cite[Thm. 5.5.2]{Chang}. The main difficulty in our case is to determine the radius of convergence of an infinite series defined in \eqref{E:defljl}. This was overcome in \cite{Chang} by using a property \cite[Prop. 3.1.1]{ABP04} of matrices satisfying a certain functional equation. In our setting, we are able to prove the same property (Theorem \ref{T:entireness}) after some analysis on the norm of a solution of a functional equation (Lemma \ref{L:solution}) and determining the solution in terms of Anderson-Thakur elements (Proposition \ref{P:matrixU}). We finish \S 2 by introducing multiple star polylogarithms and expressing multiple polylogarithms in terms of multiple star polylogarithms (Theorem \ref{T:star0}).

In \S 3, we discuss Anderson $A[\underline{t}_{\Sigma}]$-modules and introduce some properties of a special class of such modules defined as in \eqref{E:class}. Furthermore, we define the notion of uniformizability and give an example. Finally, we finish the section by introducing Frobenius modules corresponding to Anderson $A[\underline{t}_{\Sigma}]$-modules.

In \S 4, we give the definition of $A[\underline{t}_{\Sigma}]$-module $G$ and make some analysis on the coefficients of the logarithm function of $G$ (see \S 3.1 for the details on logarithm function) using Chang and Mishiba's methods in \cite{ChangMishibaApr} and \cite{ChangMishibaOct}. We introduce the Anderson $A[\underline{t}_{\Sigma}]$-module $G_{\mathcal{C}}$ and show that $G_{\mathcal{C}}$ is uniformizable (Proposition \ref{P:sec6unif}). Moreover we give the proof of Theorem \ref{T:22} and discuss Example \ref{Ex:Pellarin}. 

We conclude our paper with an Appendix to give the proof of Theorem \ref{T:sec40} which relates rigid analytic triviality to uniformizability. We note that similar result was proved by the author and Papanikolas in \cite[Thm. 4.5.5]{GP} for Anderson $A[\underline{t}_{\Sigma}]$-modules of dimension 1 over Tate algebras using Anderson's ideas in his unpublished work. Here we modify those techniques for Anderson $A[\underline{t}_{\Sigma}]$-module $G$ of higher dimension.
\section*{Acknowledgement} The author would like to express his gratitude to Chieh-Yu Chang, Nathan Green, Matthew Papanikolas and Federico Pellarin for fruitful discussions and useful suggestions. The author also thanks the referee for valuable suggestions to improve the quality of the manuscript.
\section{Multiple Polylogarithms}

\subsection{Preliminaries}

Let $U\subset \ZZ_{\geq 1}$ be a finite set. We denote the cardinality of $U$ by $|U|$. Assume that $\mu=(\mu_1,\dots,\mu_{|U|})\in \ZZ_{\geq 0}^{|U|}$ and set $\underline{t}^{\mu}_{U}:=\prod_{i\in U}t_i^{\mu_i}$. Recall from \S1 that $\Sigma\subset \ZZ_{\geq 0}$ is a finite union of finite sets. We can write any element $f\in \TT_{\Sigma}$ as $f=\sum_{\mu\in \mathbb{Z}_{\geq 0}^{|\Sigma|}}f_{\mu}\underline{t}^{\mu}_{\Sigma}$ where $f_{\mu}\in \CC_{\infty}$ such that $\inorm{f_{\mu}}\to 0$ as $\sum_{i\in \Sigma}\mu_i\to \infty$. Furthermore we let $\TT_{\Sigma}(\mathbb{K}_{\infty})\subset \TT_{\Sigma}$ to be the set of elements $f=\sum_{\mu\in \mathbb{Z}_{\geq 0}^{|\Sigma|}}f_{\mu}\underline{t}^{\mu}_{\Sigma}\in \TT_{\Sigma}$ such that $f_{\mu}\in \mathbb{K}_{\infty}$. We define the Gauss norm $\dnorm{\cdot}$ in $\TT_{\Sigma}$ by setting
\[
\dnorm{f}:=\sup \{\inorm{f_{\mu}} \ \ | \ \ \mu \in \mathbb{Z}_{\geq 0}^{|\Sigma|} \},
\]
and its corresponding valuation $\ord_{\infty}$ given by
\[
\ord_{\infty}(f):=\inf\{\ord(f_{\mu}) \ \ | \ \ \mu \in \mathbb{Z}_{\geq 0}^{|\Sigma|}\}.
\]
Note that $\TT_{\Sigma}$ is complete with respect to $\dnorm{\cdot}$. Moreover the Tate algebras $\TT_{t}$ and $\TT_{\Sigma,t}$ in variables $t$ and $t$ and $t_i$ for $i\in \Sigma$ respectively can be defined similarly. For more details on Tate algebras, we refer the reader to \cite{BGR} and  \cite{FresnelvdPut}. 

We consider the Frobenius automorphism $\tau :\TT_{\Sigma}\to \TT_{\Sigma}$ by $\tau(f)=\sum_{\mu\in \mathbb{Z}_{\geq 0}^{|\Sigma|}}f_\mu^{q}\underline{t}^{\mu}_{\Sigma}$ and we set $f^{(n)}:=\tau^{n}(f)$ for any integer $n\in \ZZ$. The extension of the homomorphism $\tau$ to $\TT_{\Sigma,t}$ can be defined similarly. 

For $k,d\in \ZZ_{\geq 1}$ and any matrix $M=(M_{ij})\in \Mat_{k\times d}(\TT_{\Sigma})$, we define $M^{(n)}$ by applying the automorphism $\tau^n$ to each entry of $M$. We set 
\[
\dnorm{M}:=\sup_{i,j}\{\dnorm{M_{ij}}\},
\] 
and consider the set $\Mat_{k\times d}(\TT_{\Sigma})[[\tau]]$  of power series of $\tau$ with coefficients in $\Mat_{k\times d}(\TT_{\Sigma})$.

Finally, when $k=d$, we form the non-commutative ring $\Mat_{k}(\TT_{\Sigma})[[\tau]]:=\Mat_{k\times k}(\TT_{\Sigma})[[\tau]]$ subject to the condition
\[
\tau M=M^{(1)}\tau
\]
and let $\Mat_{k}(\TT_{\Sigma})[\tau] \subset \Mat_{k}(\TT_{\Sigma})[[\tau]]$ to be the subring of polynomials of $\tau$ with coefficients in $\Mat_{k\times k}(\TT_{\Sigma})$.

Now we start to define some special elements which will be in use throughout the paper. For any $j\in \Sigma$ and $i\in \ZZ$ we define $b_{i}(t_j)\in K^{\text{perf}}(\underline{t}_{\Sigma})^{\times}$ by
\[
b_{i}(t_j):= \left\{\begin{array}{lr}
\prod_{k=0}^{i-1}(t_j-\theta^{q^{k}}), & \text{if } i \geq 1\\
1, & \text{if } i=0\\
\prod_{k=0}^{-i-1}(t_j-\theta^{q^{-k-1}})^{-1}, & \text{if } i\leq -1
\end{array}\right\}.
\]
Note that one can also define $b_{i}(t)$ for any $i\in \ZZ$ similarly. 
\begin{lemma} \cite[Lem. 3.3.2]{DemeslayTh}\label{L:Demeslay} 
	\begin{enumerate}
		\item[(i)] For any integer $i,d\in \ZZ$, we have
		\[
		\tau^{d}(b_{i}(t_j))=\frac{b_{d+i}(t_j)}{b_{d}(t_j)}=\frac{b_i(t_j)}{b_{d}(t_j)}\tau^i(b_d(t_j)).
		\]
		\item[(ii)] 
		\[
		\frac{1}{\tau(b_{i}(t_j))} \bigg|_{t_j=\theta}= \left\{\begin{array}{lr}
		\ell_i^{-1}, & \text{if } i \geq 0\\
		0, & \text{if } i\leq -1
		\end{array}\right\}.
		\] 
	\end{enumerate}	
\end{lemma}

After fixing a $(q-1)$-st root of $-\theta$, we define the function $\Omega(t)$ as the following infinite product
\[
\Omega(t):=(-\theta)^{\frac{-q}{q-1}}\prod_{i=1}^\infty \bigg(1-\frac{t}{\theta^{q^i}}\bigg)\in \TT_{t}^{\times}.
\]
One can observe that 
$\Omega(t)$ has infinite radius of convergence as a function of $t$ and satisfies 
\begin{equation}\label{E:omegaeq}
\Omega^{(-1)}(t)=(t-\theta)\Omega(t).
\end{equation} 
Moreover for any $n\in \ZZ_{\geq 1}$, we have
\begin{equation}\label{E:Omega1}
\Omega^{(n)}(t)=\frac{\Omega(t)}{(t-\theta^q)(t-\theta^{q^2})\dots(t-\theta^{q^n})}.
\end{equation}
Note also that
\begin{equation}\label{E:Omega2}
\Omega(\theta)=\tilde{\pi}^{-1}
\end{equation}	
where we define 
\[
\tilde{\pi}:=\theta(-\theta)^{1/(q-1)}\prod_{i=1}^{\infty}\Big( 1-\theta^{1-q^i}\Big)^{-1} \in \CC_{\infty}^{\times},
\]
the Carlitz period. Let $\beta$ be a unit in $\TT_{\Sigma}$. Then one can find an element $y\in \CC_{\infty}$ such that $\dnorm{\beta-y}<\dnorm{\beta}$. Choose an element $\gamma\in \CC_{\infty}^{\times}$ such that $\gamma^{q-1}=y$. We now define the infinite product
\[
\omega_{\beta}:=\gamma\prod_{i\geq 0}\frac{y^{q^i}}{\tau^i(\beta)},
\]
which converges in $\TT_{\Sigma}^{\times}$  by the choice of the element $y$
(see \cite[Sec. 6]{APTR} for more details). The element $\omega_{\beta}\in \TT_{\Sigma}^{\times}$ is called the Anderson-Thakur element corresponding to $\beta$ and defines up to the multiplication by an element in $\mathbb{F}_q^{\times}$.  One also notes that  \begin{equation}\label{E:taueq}
\tau(\omega_{\beta})=\beta \omega_{\beta}.
\end{equation}
Furthermore, by \cite[\S 6.1]{APTR}, we have $\dnorm{\omega_{\beta}}=q^{\frac{-\ord(\beta)}{q-1}}$ and if $\beta_1,\beta_2\in \TT_{\Sigma}^{\times}$, then we obtain $\omega_{\beta_1\beta_2}=c\omega_{\beta_1}\omega_{\beta_2}$ for some $c\in \mathbb{F}_q^{\times}$. We set $\omega_{i}:=\omega_{t_i-\theta}\in \TT_{\Sigma}^{\times}$ where $i\in \Sigma$. For any integer $n$, it satisfies that
\begin{equation}\label{E:omega1}
\omega_i^{(n)}=b_{n}(t_i)\omega_i.
\end{equation}
Now for any $U\subset \ZZ_{\geq 1}$, let us set $\omega_{U}:=\prod_{i\in U}\omega_{i}$. 
Recall from \S1 that $\Sigma=\cup_{i=1}^r U_i$. For any $1\leq i \leq r$, we define $\alpha_{i}:=\prod_{j\in U_i} (t_j-\theta)$. Thus we see that $\dnorm{\alpha_{i}}=q^{|U_i|}$. We also set $\alpha_k:=1$ and $\omega_{U_k}:=1$ if $U_k=\emptyset$. Using \eqref{E:omega1} one can obtain that 
\begin{equation}\label{E:omega2}
\omega_{U_i}^{(-1)}=\frac{\omega_{U_i}}{\alpha_{i}^{(-1)}}.
\end{equation}

\subsection{Power Sums} In most of this section, we summarize the work of Angl\`{e}s, Pellarin and Tavares Riberio  \cite[\S 6]{AnglesPellarinTavaresRibeiro} and Demeslay \cite[\S 3.3.1]{DemeslayTh} on power sums.  

We denote the set of degree $d$ polynomials in $A_{+}$ by $A_{+,d}$. Let $z$ be an indeterminate over $\CC_{\infty}$. Following the notation in \cite{AnglesPellarinTavaresRibeiro}, for any $N\in\ZZ$ and $s\in \ZZ_{\geq 1}$, we define $L(N,s,z)$ by
\[
L(N,s,z):=\sum_{d\geq 0}z^d\sum_{a\in A_{+,d}}\frac{a(t_1)\dots a(t_s)}{a^N}\in K[t_1,\dots,t_s][[z]].
\]
Let us set $\exp_z:=\sum_{i\geq 0}\frac{b_{i}(t_1)\dots b_{i}(t_s)}{D_i}z^i\tau^i\in K[t_1,\dots,t_s,z][[\tau]]$, where $D_0:=1$ and $D_i:=(\theta^{q^i}-\theta)D_{i-1}^q$. By \cite[Thm. 4.6]{AnglesPellarinTavaresRibeiro}, we know that $\exp_z(L(1,s,z))\in A[t_1,\dots,t_s][z]$. Therefore for some $m\in \ZZ_{\geq 0}$ we can let 
$
\exp_z(L(1,s,z))=\sum_{i=0}^{m}\sigma_{s,i}(\mathbf{t})z^i
$
so that $\sigma_{s,i}(\mathbf{t})\in A[t_1,\dots,t_s]$ for any $i\in \{0,\dots,m\}$.
\begin{proposition}\cite[Prop. 5.6]{AnglesPellarinTavaresRibeiro}\label{P:bound} We have 
	$
	\deg_{z}(\exp_z(L(1,s,z)))\leq \frac{s-1}{q-1}.
	$
\end{proposition}
We further define $\log_{N,z}$ by
\[
\log_{N,z}=\sum_{i\geq 0}\frac{b_{i}(t_1)\dots b_{i}(t_s)}{\ell_i^N}z^i\tau^i.
\]
Let us fix $N\in \ZZ$, $n\in \ZZ_{\geq 1}$ and set $r\in \ZZ_{\geq 1}$ as a positive integer so that $N\leq q^{r}$. Moreover we set $s:=q^r-N+n$. Thus by Proposition \ref{P:bound} we see that the $z$-degree of $\exp_z(L(1,s,z))$ only depends on integers $N$ and $n$.  For any $0\leq i \leq m$ and $0\leq j \leq B:=(q^r-N)(r-1+m)$, we have the elements $g_{i,j}:=\sum_{i_{n+1}+\dots+i_s=j}f_{i_{n+1},\dots,i_s}\in A[t_{1},\dots,t_n]$ where $f_{i_{n+1},\dots,i_s}\in A[t_{1},\dots,t_n]$ so that
\begin{multline}
b_{r}(t_1)\dots b_{r}(t_n)\tau(b_{r-1}(t_{n+1}))\dots \tau(b_{r-1}(t_s))\tau^{r}(\sigma_{s,i}(\mathbf{t}))\\
=\sum_{i_{n+1},\dots,i_s}f_{i_{n+1},\dots,i_s}t_{n+1}^{i_{n+1}}\dots t_s^{i_s}.
\end{multline}

By \cite[Thm. 6.2]{AnglesPellarinTavaresRibeiro} we obtain 
\begin{equation}\label{E:Qexp}
\sum_{d\geq 0}z^d\sum_{a\in A_{+,d}}\frac{a(t_1)\dots a(t_n)}{a^N}=\frac{1}{\ell_{r-1}^{q^r-N}b_{r}(t_1)\dots b_{r}(t_n)}\sum_{j\geq 0}\theta^j\log_{N,z}\bigg(\sum_{i=0}^{m}z^ig_{i,j}\bigg).
\end{equation}
If we analyze the coefficients of $z^d$ on both sides of \eqref{E:Qexp} we get for any $d\geq 0$,

\begin{multline}\label{E:Qexp2}
\sum_{a\in A_{+,d}}\frac{a(t_1)\dots a(t_n)}{a^N}\\
=\frac{1}{\ell_{r-1}^{q^r-N}b_{r}(t_1)\dots b_{r}(t_n)}\sum_{i=0}^{\min\{m,d\}}\frac{b_{d-i}(t_1)\dots b_{d-i}(t_n)}{\ell_{d-i}^N}\sum_{k=0}^Bg_{i,k}^{(d-i)}\theta^k.
\end{multline}

Let $K^{\text{perf}}(t_1,\dots,t_n)$ be the fraction field of the polynomial ring $K^{\text{perf}}[t_1,\dots,t_n]$. We now define the polynomial $Q_{n,N}(t)\in K^{\text{perf}}(t_1,\dots,t_n)[t]$ by
\begin{equation}\label{E:Qdef}
Q_{n,N}(t):=\sum_{k=0}^B\sum_{i=0}^{m}\frac{b_{-i}(t_1)\dots b_{-i}(t_n)}{\tau(b_{-i}(t))^N}\tau^{-i}(g_{i,k})t^k.
\end{equation}
\begin{remark} It is important to notice that it follows from the definition \eqref{E:Qdef} of $Q_{n,N}(t)$ that the $t$-coefficients of $Q_{n,N}(t)$ lie also in the Tate algebra $\TT_{\Sigma}$ where $\Sigma=\{1,\dots,n\}$.
\end{remark}
Note that 
\begin{equation}\label{E:Qexp3}
\tau^d(Q_{n,N}(t))=\sum_{k=0}^B\sum_{i=0}^{m}\frac{\tau^d(b_{-i}(t_1)\dots b_{-i}(t_n))}{\tau^d(\tau(b_{-i}(t))^N)}\tau^{d-i}(g_{i,k})t^k.
\end{equation}
By Lemma \ref{L:Demeslay}(i), we have that $\tau^d(b_{-i}(t_1)\dots b_{-i}(t_n))=\frac{b_{d-i}(t_1)\dots b_{d-i}(t_n)}{b_{d}(t_1)\dots b_{d}(t_n)}$
and 
\[
\tau^d(\tau(b_{-i}(t))^N)=\tau(\tau^d(b_{-i}(t))^N)=\frac{\tau(b_{d-i}(t)^N)}{\tau(b_{d}(t))^N)}.
\]
Thus by \eqref{E:Qexp3} we have
\begin{equation}\label{E:Qexp4}
\tau^d(Q_{n,N}(t))=\frac{\tau(b_{d}(t))^N)}{b_{d}(t_1)\dots b_{d}(t_n)}\sum_{k=0}^B\sum_{i=0}^{m}\frac{b_{d-i}(t_1)\dots b_{d-i}(t_n)}{\tau(b_{d-i}(t)^N)}\tau^{d-i}(g_{i,k})t^k.
\end{equation}
Using Lemma \ref{L:Demeslay}(ii) and \eqref{E:Qexp4} we see that 
\begin{multline}\label{E:Qexp5}
\tau^d(Q_{n,N}(t))_{|t=\theta}=\\
\frac{\ell_d^N}{b_{d}(t_1)\dots b_{d}(t_n)}\sum_{k=0}^B\sum_{i=0}^{\min\{m,d\}}\frac{b_{d-i}(t_1)\dots b_{d-i}(t_n)}{\ell_{d-i}^N}\tau^{d-i}(g_{i,k})\theta^k.
\end{multline}
Combining \eqref{E:Qexp2} with \eqref{E:Qexp5}, we see that
\begin{equation}\label{E:Qexp6}
\sum_{a\in A_{+,d}}\frac{a(t_1)\dots a(t_n)}{a^N}=\frac{1}{\ell_{r-1}^{q^r-N}b_{r}(t_1)\dots b_{r}(t_n)}\frac{b_{d}(t_1)\dots b_{d}(t_n)}{\ell_d^N}\tau^d(Q_{n,N}(t))_{t=\theta}.
\end{equation}
Observe that $L(N,s,z)$ converges for $z=1$. Thus we have by \eqref{E:Qexp6} that
\begin{equation}\label{E:Qnorm}
\frac{b_{d}(t_1)\dots b_{d}(t_n)\tau^d(Q_{n,N}(t))_{|t=\theta}}{\ell_d^N} \to 0
\end{equation}
as $d \to \infty$. Note that as $d$ gets arbitrarily large, we have  $\dnorm{Q_{n,N}^{(d)}(t)_{|t=\theta}}=\dnorm{Q_{n,N}(t)}^{q^{d}}+\epsilon_d$ where $\epsilon_d\in \RR$ is a constant depending on $d$ such that $\lim_{d\to \infty}\epsilon_d/q^d=0$. Thus, after calculating the norm of the terms in the left hand side of \eqref{E:Qnorm},  for some constant $C\in \mathbb{R}^{\times}$, we obtain
\[
Cq^{q^{d}\big(\frac{n-Nq}{q-1}+\log_q(\dnorm{Q_{n,N}(t)}) + \frac{\epsilon_d}{q^{d}}\big)} \to 0
\]
as $d$ goes to infinity. This can only happen if $\dnorm{Q_{n,N}(t)}<q^{\frac{Nq-n}{q-1}}$.

Now for any $d, N\in \ZZ_{\geq 0}$ and $U\subset \ZZ_{\geq 1}$, we denote the power sum $S_{d}(U,N)$ by
\[
S_{d}(U,N):=\sum_{a\in A_{+,d}}\frac{\sigma_{U}(a)}{a^N}\in K[\underline{t}_{U}],
\]
and state the following theorem.
\begin{theorem} \cite[Thm. 3.3.6, Lem. 3.3.9]{DemeslayTh}\label{T:Demeslay} 
	Let $N\in \ZZ_{\geq 0}$ and $r\in \ZZ_{\geq 1}$ be such that $N\leq q^{r}$. Let $U\subset \ZZ_{\geq 1}$ be a non-empty finite set. Then there exists a unique polynomial $Q_{U,N}(t)\in K^{\text{perf}}(\underline{t}_{U})[t]$ of the form
	\begin{equation}\label{E:polynQ}
	Q_{U,N}(t)=\sum_{k=0}^{B_N}\sum_{i=0}^{m_N}\frac{\prod_{j\in U}b_{-i}(t_j)}{\tau(b_{-i}(t))^N}\tau^{-i}(g_{N,i,k})t^k
	\end{equation}
	for some $B_N,m_N\in \ZZ_{\geq 0}$ and $g_{N,i,k}\in A[\underline{t}_{U}]$ such that 
	\begin{equation}\label{E:powersum}
	S_d(U,N)=\sum_{a\in A_{+,d}}\frac{\sigma_{U}(a)}{a^N}=\frac{b_d(U)}{\ell_d^{N}\ell_{r-1}^{q^r-N}b_r(U)}(\tau^d(Q_{U,N}(t)))_{|t=\theta}
	\end{equation}
	for all $d\geq 0$. In particular,
	\[
	\frac{\ell_{r-1}^{q^r-N}b_{r}(U)\omega_{U}}{\tilde{\pi}^N}S_d(U,N)=(\omega_{U}Q_{U,N}(t)\Omega^N(t))^{(d)}_{|t=\theta}.
	\]
	Moreover, $\dnorm{Q_{U,N}(t)}<q^{\frac{Nq-|U|}{q-1}}$.
\end{theorem}
\begin{proof}
	Due to above discussion, only remaining part is to prove uniqueness. Suppose there exist two polynomials $Q_{U,N}(t)$ and $Q^{\prime}_{U,N}(t)$ in $K^{\text{perf}}(\underline{t}_{U})[t]$ satisfying \eqref{E:powersum}. Then we have $(\tau^d(Q_{U,N}(t)-Q^{\prime}_{U,N}(t)))_{|t=\theta}=0$. Since $t-\theta^{q^{-d}}$ is monic in $K^{\text{perf}}(\underline{t}_{U})[t]$, it divides  $Q_{U,N}(t)-Q^{\prime}_{U,N}(t)$ for any non-negative integer $d$ in $K^{\text{perf}}(\underline{t}_{U})[t]$. But since $Q_{U,N}(t)$ and $Q^{\prime}_{U,N}(t)$ are polynomials, it is possible only if $Q_{U,N}(t)=Q^{\prime}_{U,N}(t)$.
\end{proof}

For the completeness of this section, we also state Anderson and Thakur's result on power sums using our notation as follows.
\begin{theorem}\cite[Eq. 3.7.3-3.7.4]{AndThak90}\label{T:Demeslay2}
	Let $N \in \ZZ_{\geq 0}$ be such that $N=\sum n_iq^i$ where $0\leq n_i \leq q-1$ and set $\Gamma_{N}:=\prod D_i^{n_i}$. Then there exists a unique polynomial $Q_{\emptyset, N}(t)\in A[t]$ such that 
	\[
	S_d(\emptyset,N)=\sum_{a\in A_{+,d}}\frac{1}{a^N}=\frac{1}{\ell_d^{N}\Gamma_{N}}(\tau^d(Q_{\emptyset,N}(t)))_{|t=\theta}
	\]
	for all $d\geq 0$. In particular,
	\[
	\frac{\Gamma_{N}}{\tilde{\pi^N}}S_d(\emptyset, N)=(\Omega^{N}(t)Q_{\emptyset, N}(t))^{(d)}_{|t=\theta}.	
	\]
	Moreover, $\dnorm{Q_{\emptyset,N}(t)}<q^{\frac{Nq}{q-1}}$.
\end{theorem}

\subsection{Multiple Zeta Values over Tate algebras}
Throughout this section we assume that $(s_1,s_2,\dots,s_r)\in \ZZ_{\geq 1}^r$ and $\Sigma=\cup_{i=1}^rU_i$ unless otherwise stated.

\begin{definition} Let $\mathcal{C}$ be a composition array defined as in \eqref{E:star00}. We define the multiple zeta values $\zeta_{C}(\mathcal{C})$ over Tate algebras in the sense of the definition of Pellarin \cite{PellarinMZV}, \cite{PellarinMZV2} as the following object which converges in $\TT_{\Sigma}$:
	\[
	\begin{split}
	\zeta_{C}(\mathcal{C})&=\sum_{\substack{\inorm{a_1}>\dots >\inorm{a_r}\geq 0 \\ a_1,\dots,a_r\in A_{+}}}\frac{\sigma_{U_1}(a_1)\sigma_{U_2}(a_2)\dots \sigma_{U_r}(a_r)}{a_1^{s_1}a_2^{s_2}\dots a_{r}^{s_r}}\\
	&=\sum_{i_1>i_2>\dots>i_r\geq 0}S_{i_1}(U_1,s_1)\dots S_{i_r}(U_r,s_r).
	\end{split}
	\]
\end{definition}
We define $n_l:=|U_l|$ and consider the set $D^{\prime}_{\mathcal{C}}\subset \TT_{\Sigma}^r$ given by
\[
D^{\prime}_{\mathcal{C}}:=\{(f_1,\dots,f_r)\in \TT_{\Sigma}^r \ \ | \ \ \dnorm{f_i}< q^{\frac{s_iq-n_i}{q-1}} \ \ \text{for } i=1,\dots,r \}.
\]
For an $r$-tuple $(f_1,\dots,f_r)\in D^{\prime}_{\mathcal{C}}$, we set
\[
\Li_{\mathcal{C}}(f_1,\dots,f_r):=\sum_{i_1>i_2>\dots>i_r\geq 0}\frac{b_{i_1}(U_1)\dots b_{i_r}(U_r)\tau^{i_1}(f_1)\dots \tau^{i_r}(f_r)}{ \ell_{i_1}^{s_1}\dots \ell_{i_r}^{s_r}}.
\]
Using the definition of elements $b_i(U)$ and $\ell_i$, we note that $\Li_{\mathcal{C}}(f_1,\dots,f_r)$ converges in $\TT_{\Sigma}$ if $(f_1,\dots,f_r)\in D^{\prime}_{\mathcal{C}}$.

Let us fix a composition array $\mathcal{C}$ of depth $r$ as in \eqref{E:star00} and let $C_i$ be the $t$-degree of the polynomial $Q_{U_i,s_i}(t)=\sum_{j\geq 0} u_{ij}t^j\in K^{\text{perf}}(\underline{t}_{U_i})[t]$ for $1\leq i \leq r$. Then consider the index set
\[
\mathfrak{I}:=\{0,\dots,C_1\}\times \dots \times \{0,\dots,C_r\}.
\]
For each $i=(j_1,\dots,j_r)\in \mathfrak{I}$, we set 
\begin{equation}\label{E:udefi}
u_i:=(u_{1j_1},\dots,u_{rj_r})\in K^{\text{perf}}(\underline{t}_{U_1})\times \dots \times K^{\text{perf}}(\underline{t}_{U_r}) .
\end{equation}
Furthermore, we set $a_{i}:=\theta^{j_{1}+\dots+j_r}$.

\begin{theorem}\label{T:polylogarithm}
	Let $\mathfrak{I}_1$ be the set of indices $i$ such that $U_i\neq \emptyset$ and $\mathfrak{I}_2$ be the set of $i$ such that $U_i=\emptyset$. Then we have 
	\begin{equation}\label{E:deformmzv}
	\zeta_{C}(\mathcal{C})=\frac{1}{ \prod_{i\in \mathfrak{I}_1} \ell_{r_{s_i}-1}^{q^{r_{s_i}}-s_i}b_{r_{s_i}}(U_i)\prod_{i\in \mathfrak{I}_2}\Gamma_{s_i} }\sum_{i \in \mathfrak{I}}a_{i}\Li_{\mathcal{C}}(u_i)
	\end{equation}
	where $r_{s_i}\geq 1$ is an integer such that $s_i\leq q^{r_{s_i}}$ for $i\in \mathfrak{I}_1$.
\end{theorem}
In \S2.5, we give a proof for Theorem \ref{T:polylogarithm}. The main idea is to modify the idea of the proof of \cite[Prop. 3.1.1]{ABP04} and  combine it with the idea of the proof of \cite[Thm. 5.5.2]{Chang}. Before stating the proof, we discuss some applications of our result.

\subsection{Applications of Theorem \ref{T:polylogarithm} to Dirichlet-Goss multiple $L$-values}
First we briefly discuss Gauss-Thakur sums.  We recall the notation from \S1.3 and for any $j=0,\dots,d-1$ define the Gauss-Thakur sum 
\[
g(v_{\mathfrak{p}}^{q^j})=-\sum_{\delta \in \Delta_{\mathfrak{p}}}v_{\mathfrak{p}}(\delta^{-1})^{q^j}\delta(\lambda_{\mathfrak{p}})\in \mathbb{F}_{q^d}[\theta][\lambda_{\mathfrak{p}}].
\]
By \cite[Thm. I]{Th88}, we know that $g(v_{\mathfrak{p}})$ is non-zero. Moreover, by \cite[Thm. 2.9]{AP15}, we have
\begin{equation}\label{E:GT1}
g(v_{\mathfrak{p}})=\mathfrak{p}^{\prime}(\xi_{p})^{-1}\omega(\xi_{p})
\end{equation}
where $\mathfrak{p}^{\prime}$ is the first derivative of $\mathfrak{p}$ with respect to $\theta$, $\omega\in \TT_{t}$ is the Anderson-Thakur element corresponding to $t-\theta$ defined by 
\[
\omega:=(-\theta)^{\frac{1}{q-1}}\prod_{i=0}^\infty \bigg(1-\frac{t}{\theta^{q^i}}\bigg)^{-1}\in \TT_{t}^{\times},
\]
and $\omega(\xi_{p})=\omega_{|t=\xi_{p}}$. For further details about Gauss-Thakur sums, we refer the reader to \cite{AnglesPellarin}, \cite{AP15}, \cite{Goss} and \cite{Th88}.

\begin{proof}[{Proof of Corollary \ref{C:DG}}]
	Consider the composition array $\mathcal{C}$ as in \eqref{E:star00} with pairwise disjoint subsets  $U_i$ such that $|U_i|=j_i$ for all $1\leq i \leq r$. By the definition of the polynomials $Q_{U_i,s_i}(t)$, there exists $m_i\in \mathbb{Z}_{\geq 0}$ such that 
	\[
	Q_{U_i,s_i}(t)=\prod_{k\in U_i} b_{-m_i}(t_k)\sum_{l\geq 0}c_{i,l}t^l\in K^{\text{perf}}(\underline{t}_{U_i})[t]
	\]
	where $c_{i,l}=\sum_{\mu\in \ZZ_{\geq 0}^{j_i}}c_{\mu,i,l}\underline{t}_{U_i}^{\mu}\in K^{\text{perf}}[\underline{t}_{U_i}]$ with $c_{\mu,i,l}\in K^{\text{perf}}$ and $c_{\mu,i,l}=0$ for $l\gg 0$ and all but finitely many tuple $\mu\in \ZZ_{\geq 0}^{j_i}$. Up to permutation of elements $\xi_{ik}\in \overline{\mathbb{F}}_q$ for $1\leq i \leq r$ and $1\leq k \leq j_i$, assume that the Dirichlet character $\chi_i:A\to \overline{\mathbb{F}}_q$ given by $\chi_i(a)=  a(\xi_{i1})\dots a(\xi_{ij_i})$ is also given by $\chi_i(a)=\prod_{k\in U_i}a(t_k)_{|t_k=\xi_{ik}}$. For any $l\in \ZZ_{\geq 0}$, set $B_{\chi_{i},l}:=\prod_{k\in U_i}\omega_k^{(l)}(\xi_{ik})b_{-m_i}(t_{k})^{(l)}_{|t_{k}=\xi_{ik}}\in \CC_{\infty}^{\times}$. Finally for $1\leq i \leq r$, set $\xi_i:=(\xi_{i1},\dots,\xi_{ij_i})$ and for any tuple $\mu=(\mu_1,\dots,\mu_{j_i})\in \ZZ_{\geq 0}^{j_i}$, define $\xi_i^{\mu}:=\prod_{1\leq k\leq j_i}\xi_k^{\mu_k}$. Thus by using Lemma \ref{L:Demeslay} and \eqref{E:omega1}, for any tuple $(u_{1j_1},\dots,u_{rj_r})\in D^{\prime}_\mathcal{C}$ as in \eqref{E:udefi}, we have
	\begin{align*}
	&\Li_{\mathcal{C}}(u_{1j_1},\dots,u_{rj_r})_{\big|\substack{t_{k}=\xi_{ik}\\k\in U_i}}\\
	&= \frac{1}{\prod_{k\in U_1}\omega_{k}(\xi_{1k})\dots \prod_{k\in U_r}\omega_k(\xi_{rk})}\times \\
	&\ \ \ \ \ \ \ \sum_{\substack{\mu_i\in \ZZ_{\geq 0}^{j_i}\\l_1,\dots,l_r\geq 0}}\xi_1^{\mu_1}\dots \xi_r^{\mu_r}\sum_{i_1>i_2>\dots>i_r\geq 0}\frac{B_{\chi_1,i_1}\dots B_{\chi_r,i_r}c_{\mu_1,1,l_1}^{q^{i_1}}\dots c_{\mu_r,r,l_r}^{q^{i_r}}}{ \ell_{i_1}^{s_1}\dots \ell_{i_r}^{s_r}}\\
	&=\frac{1}{\prod_{k\in U_1}\omega_{k}(\xi_{1k})\dots \prod_{k\in U_r}\omega_k(\xi_{rk})}\times \\ 
	&\ \ \ \ \ \ \ \sum_{\substack{\mu_i\in \ZZ_{\geq 0}^{j_i}\\l_1,\dots,l_r\geq 0}}\xi_1^{\mu_1}\dots \xi_r^{\mu_r}\Li_{\big(\substack{\chi_1,\dots,\chi_r\\s_1,\dots,s_r}\big)}(c_{\mu_1,1,l_1},\dots,c_{\mu_r,r,l_r}).
	\end{align*}	
	Since the coefficients $c_{\mu,i,j}=0$ for sufficiently large $j$ and all but finitely many  $\mu$ when $1\leq i \leq r$, the sum in right hand side of the second equality above has finitely many terms. Thus evaluating both sides of \eqref{E:deformmzv} at $t_{k}=\xi_{ik}$ for $1\leq i \leq r$, $k\in U_i$ and using \eqref{E:GT1} finish the proof.
\end{proof}
Next we prove Theorem \ref{T:linindep} by using  Theorem \ref{T:polylogarithm} and the transcendence theory developed in \cite{Chang}.
\begin{proof}[{Proof of Theorem \ref{T:linindep}}] For some $1\leq i \leq m$ and $j_i\leq i$, consider the Dirichlet-Goss multiple $L$-value $L(\chi_{U_{i1},\xi_{i1}},\dots,\chi_{U_{ij_i},\xi_{ij_i}};s_{i1},\dots,s_{ij_i})$. We continue with the notation of the proof of Corollary \ref{C:DG}. For $1\leq k \leq j_i$, set $P_{ik,l}(t_{U_{ik}}):=\sum_{\mu\in \ZZ_{\geq 0}^{j_{ik}}}c_{\mu,(ik),l}\underline{t}_{U_{ik}}^{\mu}\in K^{\text{perf}}[\underline{t}_{U_i}]$  so that $Q_{U_{ik},s_{ik}}(t)=\prod_{v\in U_{ik}} b_{-m_{ik}}(t_v)\sum_{l\geq 0}P_{ik,l}(t_{U_{ik}})t^l$. Since the field $\mathbb{F}_q$ is invariant under the automorphism $\tau$, by using \eqref{E:taueq}, we see that $\omega_v(\xi_{ikv})$ is algebraic over $K$ for $1\leq k \leq j_i$ and $v\in U_{ik}$. Moreover, for any $n\in \ZZ_{\geq 0}$, we also see that 
	\begin{equation}\label{Eq:b1}
	B_{\chi_{U_{ik}},n}=\prod_{v\in U_{ik}}\omega_v^{(n)}(\xi_{kv})b_{-m_{ik}}(t_{v})^{(n)}_{|t_{v}=\xi_{ikv}}=\Big(\prod_{v\in U_{ik}}\omega_v(\xi_{ikv})b_{-m_{ik}}(\xi_{ikv})\Big)^{q^n}
	\end{equation}
	where $b_{-m_{ik}}(\xi_{ikv})=b_{-m_i}(t_v)_{|t_v=\xi_{ikv}}$. Furthermore one can also obtain that  
	\begin{equation}\label{Eq:b2}
	P_{ik,l}(t_{U_{ik}})^{(n)}_{\Big|\substack{t_{v}=\xi_{ikv}\\{v\in U_{ik}}}}=\Big(P_{ik,l}(t_{U_{ik}})_{\big|\substack{t_{v}=\xi_{ikv}\\{v\in U_{ik}}}}\Big)^{q^n}.
	\end{equation}
	Now for any tuple $l=(l_1,\dots,l_{j_i})\in \mathfrak{I}':=\{0,\dots,C_{i1}\}\times \{0,\dots,C_{ij_i}\}$ where $C_{ih}$ is the degree of $Q_{U_{ih},s_{ih}}(t)$ as a polynomial of $t$ for $1\leq h \leq j_i$, set $ \mu_{l}$ to be an element of $\overline{K}^{j_i}$ given by
	\begin{align*}
	&\Big(\prod_{v\in U_{i1}}\omega_{v}(\xi_{i1v})b_{-m_{i1}}(\xi_{i1v})P_{i1,l_1}(t_{U_{i1}})_{\big|\substack{t_{v}=\xi_{i1v}\\{v\in U_{i1}}}},\dots,\\
	&\ \ \ \ \  \ \ \ \ \ \ \  \ \ \ \ \ \ \prod_{v\in U_{ij_i}}\omega_{v}(\xi_{ij_iv})b_{-m_{ij_i}}(\xi_{ij_iv})P_{ij_i,l_{j_i}}(t_{U_{ij_i}})_{\big|\substack{t_{v}=\xi_{ij_iv}\\{v\in U_{ij_i}}}}\Big).
	\end{align*}
	Consider the composition array $\mathcal{C}_i=\begin{pmatrix}
	U_{i1}&\dots &U_{ij_i}\\s_{i1}&\dots&s_{ij_{i}}
	\end{pmatrix}$.
	Then, for any $(u_{i1l_{1}},\dots,u_{ij_{i}l_{j_i}})\in D^{\prime}_{\mathcal{C}_i}$ corresponding to the tuple $l$ as in \eqref{E:udefi}, using \eqref{Eq:b1} and \eqref{Eq:b2}, we immediately see that 
	\begin{equation}
	\Li_{\mathcal{C}_i}(u_{i1l_{1}},\dots,u_{ij_{i}l_{j_i}})_{\Big|\substack{t_{v}=\xi_{ikv}\\{v\in U_{ik}}}}= \frac{1}{\prod_{v\in U_{ik}}\omega_{v}(\xi_{i1v})\dots \prod_{v\in U_{ij_i}}\omega_v(\xi_{ij_iv})}\Li_{\mathfrak{s}_i}(\mu_{l})
	\end{equation}
	where $\mathfrak{s}_i=(s_{i1},\dots,s_{ij_i})$ and $\Li_{\mathfrak{s}_i}$ is the Carlitz multiple polylogarithm defined as in \eqref{CMPLL}. Thus we obtain that the Dirichlet-Goss $L$-value $L(\chi_{U_{11},\xi_{11}},\dots,\chi_{U_{i1},\xi_{ij_i}};s_{i1},\dots,s_{ij_i})$ can be written as a $\overline{K}$-linear combination of multiple polylogarithms $\Li_{\mathfrak{s}}$ at algebraic points by Theorem \ref{T:polylogarithm}. Finally the result now follows from  \cite[Thm. 3.4.5]{Chang} and \cite[Prop. 5.4.1]{Chang}.
\end{proof}
\begin{remark} In \cite[Thm. 5.2.5]{ChangMishibaOct}, Chang and Mishiba introduced the multiple zeta value $\zeta_{A}(\mathfrak{s})$ as a $K$-linear combination of multiple star polylogarithms at some algebraic points $v_1,\dots,v_m$ for some $m\in \ZZ_{\geq 1}$. Furthermore, they showed that those points are related to dual $t$-motives of certain Anderson $t$-modules via the isomorphism between $\Ext^1$-modules and Anderson $t$-modules (see \cite[Thm. 5.2.3]{CPY} and \cite[Rem. 3.3.5]{ChangMishibaApr}). In our case, when the Dirichlet characters are of the form as in Theorem \ref{T:linindep},
	one can write a Dirichlet-Goss multiple $L$-value as a $\overline{K}$-linear combination of multiple star polylogartihms at some algebraic points $\tilde{v}_1,\dots,\tilde{v}_m'$ for some $m'\in \ZZ_{\geq 1}$ by using Theorem \ref{T:star0} which will be proved in \S2.6 and form a similar relation between those points and dual $t$-motives. However, since the elements of $\overline{\mathbb{F}}_q\setminus\mathbb{F}_q$ are not invariant under the automorphism $\tau$, we do not know the corresponding Anderson $t$-modules to these points when arbitrary Dricihlet characters are used to construct Dirichlet-Goss multiple $L$-values. This is due to the fact that these points can be only expressed after specializing variables by using our methods. It would be interesting to construct these $t$-modules directly to make the relation between dual $t$-motives and the points $\tilde{v}_1,\dots,\tilde{v}_m$ more transparent.
\end{remark}

\subsection{Proof of Theorem \ref{T:polylogarithm}}

For any $1\leq l<j\leq r+1$, we define the following objects:
\begin{equation}\label{E:defljl}
\begin{split}
&L_{j,l}(t):=\\
&\sum_{i_l>\dots>i_{j-1}\geq 0}(\omega_{U_l}\Omega^{s_l}(t)Q_{U_l,s_l}(t))^{(i_l)} \dots (\omega_{U_{j-1}}\Omega^{s_{j-1}}(t)Q_{U_{j-1},s_{j-1}}(t))^{(i_{j-1})}\\
&=\Omega^{s_{l}+\dots +s_{j-1}}(t)\prod_{i=l}^{j-1}\omega_{U_i}\times \\
&\ \  \ \ \ \ \ \ \sum_{i_l>\dots>i_{j-1}\geq 0}\frac{Q^{(i_l)}_{U_l,s_l}(t)\dots Q_{U_{j-1},s_{j-1}}^{(i_{j-1})}(t) b_{i_l}(U_l)\dots b_{i_{j-1}}(U_{j-1}) }{((t-\theta^q)\dots (t-\theta^{q^{i_l}}))^{s_l}\dots ((t-\theta^q)\dots (t-\theta^{q^{i_{j-1}}}))^{s_{j-1}} }.
\end{split}
\end{equation}
Moreover we set $L_{j,l}(t):=1$ if $j=l$.
Observe that for some $\epsilon_l,\dots,\epsilon_{j-1}\in \mathbb{R}_{>0}$,  we have by Theorem \ref{T:Demeslay} and Theorem \ref{T:Demeslay2} that
\begin{align*}
&\frac{\dnorm{Q_{U_l,s_l}^{(i_l)}(t)\dots Q^{(i_{j-1})}_{U_{j-1},s_{j-1}}(t)}\dnorm{b_{i_l}(U_l)\dots b_{i_{j-1}}(U_{j-1})}}{\dnorm{((t-\theta^q)\dots (t-\theta^{q^{i_l}}))^{s_l}\dots ((t-\theta^q)\dots (t-\theta^{q^{i_{j-1}}}))^{s_{j-1}}}}\\
&=q^{q^{i_l}\big(\frac{s_lq-|U_l|}{q-l}-\epsilon_l\big)}\dots q^{q^{i_{j-1}\big(\frac{s_{j-1}q-|U_{j-1}|}{q-1}-\epsilon_{j-1}\big)}}\times \\
&\ \ \  \ \ \ \ \ \ \ \ \ \ q^{\frac{q^{i_l}|U_l|-|U_l|}{q-1}}\dots q^{\frac{q^{i_{j-1}}|U_{j-1}|-|U_{j-1}|}{q-1}}
q^{-\big(s_lq\big(\frac{q^{i_l}-1}{q-1}\big)+\dots +s_{j-1}q\big(\frac{q^{i_{j-1}}-1}{q-1}\big)\big)} \\
&=q^{q^{i_l}\big(\frac{s_lq-|U_l|}{q-1}+\frac{|U_l|}{q-1}-\frac{s_lq}{q-1}-\epsilon_l\big)}\dots q^{q^{i_{j-1}}\big(\frac{s_{j-1}q-|U_{j-1}|}{q-1}+\frac{|U_{j-1}|}{q-1}-\frac{s_{j-1}q}{q-1}-\epsilon_{j-1}\big)}\times\\ &\ \ \  \ \ \ \ \ \ \ \ \ \ q^{\frac{-|U_l|-\dots-|U_{j-1}|+q(s_l+\dots+s_{j-1})}{q-1}}
\\
&=q^{-q^{i_l}\epsilon_l}\dots q^{-q^{i_{j-1}}\epsilon_{j-1}}q^{\frac{-|U_l|-\dots-|U_{j-1}|+q(s_l+\dots+s_{j-1})}{q-1}}\to 0
\end{align*}
whenever $0\leq i_{j-1}<\dots <i_l \to \infty$. Thus we conclude that $L_{j,l}(t)\in \TT_{\Sigma,t}$.

\begin{proposition}\label{P:ljl}For any $1\leq l<j\leq r+1$, we have
	\begin{equation}\label{E:Lj}
	L_{j,l}^{(-1)}(t)=L_{j,l}(t)+Q_{U_{j-1},s_{j-1}}^{(-1)}(t)\frac{(t-\theta)^{s_{j-1}}}{\alpha_{j-1}^{(-1)}}\Omega^{s_{j-1}}(t)\omega_{U_{j-1}}L_{j-1,l}(t).
	\end{equation}
\end{proposition}
\begin{proof}
	Observe that 	
	\begin{align*}
	&L_{j,l}(t)=\Omega^{s_l+\dots +s_{j-1}}(t)\prod_{i=l}^{j-1}\omega_{U_i}\times\\ 
	& \ \  \ \ \ \ \ \ \  \ \ \bigg[\sum_{i_l>\dots>i_{j-1}\geq 1}\frac{Q_{U_l,s_l}^{(i_l)}(t)\dots Q^{(i_{j-1})}_{U_{j-1},s_{j-1}}(t)b_{i_l}(U_l)\dots b_{i_{j-1}}(U_{j-1})}{((t-\theta^q)\dots (t-\theta^{q^{i_l}}))^{s_l}\dots ((t-\theta^q)\dots (t-\theta^{q^{i_{j-1}}}))^{s_{j-1}} }
	\\
	&\ \  \ \ \ \ \ \ \  \ \ \ \  \ \ \ \ \ \ \  \ \ +Q_{U_{j-1,s_{j-1}}}(t)\times \\
	&\ \  \ \ \ \ \ \ \  \ \ \sum_{i_l>\dots>i_{j-2}\geq 1}\frac{Q_{U_l,s_l}^{(i_l)}(t)\dots Q^{(i_{j-2})}_{U_{j-2},s_{j-2}}(t)b_{i_l}(U_l)\dots b_{i_{j-2}}(U_{j-2})}{((t-\theta^q)\dots (t-\theta^{q^{i_l}}))^{s_l}\dots ((t-\theta^q)\dots (t-\theta^{q^{i_{j-2}}}))^{s_{j-2}} }\bigg].
	\end{align*}
	Therefore using \eqref{E:omegaeq} and \eqref{E:omega2} together with the above equation, we see that the equality in \eqref{E:Lj} holds.
\end{proof}

We recall the polynomials $Q_{U_i,s_i}(t)\in K^{\text{perf}}(\underline{t}_{U_i})[t]$ from Theorem \ref{T:Demeslay} and Theorem \ref{T:Demeslay2}, and consider the matrix $\Phi_{l} \in \Mat_{r+2-l}(\TT_{\Sigma}[t])$ defined by the following matrix
\begin{equation} \label{E:Phi}
\begin{bmatrix}
\frac{(t-\theta)^{s_l+\dots+s_r}}{\prod_{i=l}^r\alpha_i^{(-1)}} & 0 & 0 & \cdots & 0 \\
\frac{Q^{(-1)}_{U_l,s_l}(t)(t-\theta)^{s_l+\dots+s_r}}{\prod_{i=l}^r\alpha_i^{(-1)}} & \frac{(t-\theta)^{s_{l+1}+\dots+s_r}}{\prod_{i=l+1}^r\alpha_i^{(-1)}} & 0 & \cdots & 0 \\
0 & \frac{Q^{(-1)}_{U_{l+1},s_{l+1}}(t)(t-\theta)^{s_{l+1}+\dots+s_r}}{\prod_{i=l+1}^r\alpha_i^{(-1)}} & \ddots & \ddots & \vdots \\
\\
\vdots &  & \ddots &  \frac{(t-\theta)^{s_r}}{\alpha_r^{(-1)}}& 0\\
0 & \cdots & 0 & \frac{Q^{(-1)}_{U_r,s_r}(t)(t-\theta)^{s_r}}{\alpha_r^{(-1)}} & 1   \\
\end{bmatrix}.
\end{equation}

and the matrix $\Psi_{l} $ defined by
\begin{equation} \label{E:Psi}
\Psi_{l}:=\begin{bmatrix}
\Omega^{s_l+\dots +s_r}(t)\prod_{i=l}^r\omega_{U_i}  \\
L_{(l+1),l}(t)\Omega^{s_{l+1}+\dots +s_r}(t)\prod_{i=l+1}^r\omega_{U_i}   \\
\vdots \\
L_{r,l}(t)\Omega^{s_r}(t)\omega_{U_r} \\
L_{r+1,l}(t)  \\
\end{bmatrix}\in \Mat_{(r+2-l) \times 1}(\TT_{\Sigma,t}).
\end{equation}

Using \eqref{E:omegaeq}, \eqref{E:omega2} and  Proposition \ref{P:ljl} one can prove the following lemma.
\begin{lemma}\label{L:funceq} We have 
	\[\Psi^{(-1)}_{l}=\Phi_{l}\Psi_{l}.
	\]
\end{lemma}

In order to prove that the function $L_{j,l}(t)$ has infinite radius of convergence, we need to state a technical lemma.
\begin{lemma}\label{L:solution} Let $\{s_1,\dots,s_k\}\in \ZZ_{\geq 0}^k$. Let $G,Q\in \TT_{\Sigma}$ be such that $\dnorm{G}\leq q^{\frac{q(s_i+\dots+s_j)-(|U_i|+\dots+|U_j|)}{q-1}}$  and $\dnorm{Q}<q^{\frac{qs_{i-1}-|U_{i-1}|}{q-1}}$ for some $2\leq i\leq j\leq k$. Let also $H\in \TT_{\Sigma}^{\times}$ such that $\dnorm{H}=q^{\big(s_{i-1}+\dots+s_j-\frac{|U_{i-1}|+\dots+|U_j|}{q}\big)}$.  Then there exists an element $F\in \TT_{\Sigma}$ such that 
	\[
	H(F^{(-1)}-Q^{(-1)}G^{(-1)})=F.
	\]
	Moreover, $\dnorm{F}\leq q^{\frac{q(s_{i-1}+\dots+s_j)-(|U_{i-1}|+\dots+|U_j|)}{q-1}}$.
\end{lemma}
\begin{proof}
	We define the potential solution $F$ as the following infinite sum
	\[
	F:=QG+Q^{(1)}G^{(1)}(H^{-1})^{(1)} + \sum_{r=2}^{\infty}Q^{(r)}G^{(r)}(H^{-1})^{(r)}(H^{-1})^{(r-1)}\dots(H^{-1})^{(1)}.
	\]
	One can see that $F$ satisfies the desired equality in the lemma. We need to show that $F$ is a well-defined element in $\TT_{\Sigma}$.  By the assumptions on elements $G,H$ and $Q$, for some $\epsilon>0$ and $r\to \infty$, we have the following estimate.
	\begin{align*}
	&\dnorm{Q^{(r)}G^{(r)}(H^{-1})^{(r)}\dots(H^{-1})^{(1)}}\\
	&\leq q^{\frac{q^{r}}{q-1}\big(qs_{i-1}-|U_{i-1}|-\epsilon+q(s_i+\dots+s_j)-(|U_{i}|+\dots+|U_j|)\big)}\times \\
	&\ \ \  \ \ \ \ \ \ \ \ \  \ \ \ \ \ \ q^{\frac{q^{r}-1}{q-1}(-q(s_{i-1}+\dots+s_j)+|U_{i-1}|+\dots+|U_j|)}\\
	&=q^{\frac{-q^r\epsilon}{q-1}}q^{\frac{-1}{q-1}(-q(s_{i-1}+\dots+s_j)+|U_{i-1}|+\dots+|U_j|)}\to 0.
	\end{align*}
	Thus we can conclude that as $r\to \infty$, the norm of the general term of the sum approaches to 0. Therefore the sum is well defined and since $\TT_{\Sigma}$ is a complete normed space, $F$ is in $\TT_{\Sigma}$.
	On the other hand, by the assumptions and the properties of the non-archimedean norm $\dnorm{\cdot}$, we have
	$\dnorm{F}\leq \max\{\dnorm{HQ^{(-1)}G^{(-1)}},\dnorm{F^{(-1)}H}\}.
	$

	\textit{Case 1: }$\dnorm{F}=\dnorm{HQ^{(-1)}G^{(-1)}}$.
	
	By the assumption, we have the following estimate.
	\begin{align*}
	\dnorm{F}&<q^{(s_{i-1}+\dots+s_j-(|U_{i-1}|+\dots+|U_j|)/q)}q^{\frac{qs_{i-1}-|U_{i-1}|}{q(q-1)}}q^{\frac{q(s_i+\dots+s_j)-(|U_{i}|+\dots+|U_j|)}{q(q-1)}}\\
	&=q^{\frac{(q-1)(s_{i-1}+\dots+s_j)}{q-1}-(|U_{i-1}|+\dots+|U_j|)/q}\times \\ 
	& \ \  \ \ \ \ \ \ \ \ \ \  \ \ \ q^{\frac{1}{q(q-1)}\big(q(s_{i-1}+s_i+\dots+s_j)-(|U_{i-1}|+\dots+|U_j|)\big)} \\
	&=q^{\frac{q(s_{i-1}+\dots+s_j)}{q-1}-\frac{(|U_{i-1}|+\dots+|U_j|)}{q}\big(1+\frac{1}{q-1}\big)}=q^{\frac{q(s_{i-1}+\dots+s_j)-(|U_{i-1}|+\dots+|U_j|)}{q-1}}.
	\end{align*}
	
	\textit{Case 2: }$\dnorm{F}=\dnorm{HF^{(-1)}}$.
	
	We note that $\dnorm{F^{(-1)}}=\dnorm{F}^{1/q}$. Thus, in this case we have that 
	\begin{align*}
	\dnorm{F}&=\dnorm{H}^{\frac{q}{q-1}}\\
	&=q^{\frac{q}{q-1}\big((s_{i-1}+\dots+s_j)-\frac{(|U_{i-1}|+\dots+|U_j|)}{q}\big)}\\
	&=q^{\frac{q(s_{i-1}+\dots+s_j)-(|U_{i-1}|+\dots+|U_j|)}{q-1}}.
	\end{align*}
	By the analysis of these two cases, we deduce the last statement of the lemma.
\end{proof}
Now recall the matrix $\Phi_{l}$ from \eqref{E:Phi} and let us define elements $b_i\in \Mat_{r+2-l}(\TT_{\Sigma})$ such that $\Phi_l=\sum_{i=0}^m b_it^i$ for some $m\in \mathbb{Z}_{\geq 0}$. 
\begin{proposition}\label{P:matrixU}
	There exists a matrix $U\in \GL_{r+2-l}(\TT_{\Sigma})$ such that $U^{(-1)}b_0=U$.
\end{proposition}  
\begin{proof}
	Without loss of generality let us take $l=1$. To avoid heavy notation let $Q_{U_i,s_i}\in \TT_{\Sigma}$ be the constant term of the polynomial $Q_{U_i,s_i}(t)$ for any $1\leq i\leq r$. Let $\{s_1,\dots,s_r\}\in \mathbb{Z}_{\geq 1}^r$. Observe that 
	\[
	b_0=\begin{bmatrix}
	\frac{(-\theta)^{s_1+\dots+s_r}}{\prod_{i=1}^r\alpha_i^{(-1)}} & 0 & 0 & \cdots & 0 \\
	\frac{Q^{(-1)}_{U_1,s_1}(-\theta)^{s_1+\dots+s_r}}{\prod_{i=1}^r\alpha_i^{(-1)}} & \frac{(-\theta)^{s_2+\dots+s_r}}{\prod_{i=2}^r\alpha_i^{(-1)}} & 0 & \cdots & 0 \\
	0 & \frac{Q^{(-1)}_{U_2,s_2}(-\theta)^{s_2+\dots+s_r}}{\prod_{i=2}^r\alpha_i^{(-1)}} & \ddots &  & \vdots \\
	\\
	\vdots &  & \ddots &  \frac{(-\theta)^{s_r}}{\alpha_r^{(-1)}}& 0\\
	0 & \cdots & 0 & \frac{Q^{(-1)}_{U_r,s_r}(-\theta)^{s_r}}{\alpha_r^{(-1)}} & 1   \\
	\end{bmatrix}.
	\]
	Let the matrix $U=(a_{ij})$ be defined as a potential solution of the equation in the proposition. Therefore the $i$-th row $R_i$ of the matrix $U^{(-1)}b_0$ appears as
	\begin{multline}\label{E:rowsofU}
	R_i:=\bigg[\frac{(-\theta)^{s_1+\dots+s_r}}{\prod_{i=1}^r\alpha_i^{(-1)}}(a_{i,1}^{(-1)}+Q_{U_1,s_1}^{(-1)}a_{i,2}^{(-1)}),\frac{(-\theta)^{s_2+\dots+s_r}}{\prod_{i=2}^r\alpha_i^{(-1)}}(a_{i,2}^{(-1)}+Q_{U_2,s_2}^{(-1)}a_{i,3}^{(-1)})\\
	,\dots,\frac{(-\theta)^{s_r}}{\alpha_r^{(-1)}}(a_{i,r}^{(-1)}+Q_{U_r,s_r}^{(-1)}a_{i,r+1}^{(-1)}),a_{i,r+1}^{(-1)}\bigg].
	\end{multline}
	Now our aim is to pick elements $a_{i,j}\in \TT_{\Sigma}$ in a way that the desired equality would be satisfied. First for $1\leq i \leq r$, we set $a_{i,j}=0$ when $j>i$. In order to see how we can pick the other elements let us analyze the $k$-th row of $U$ where $1\leq k \leq r$. By the above setting, we know $a_{k,k+1}$=0. Then by \eqref{E:rowsofU}, in order to give the desired equality we want to have 
	\[
	\frac{(-\theta)^{s_k+\dots+s_r}}{\prod_{i=k}^r\alpha_i^{(-1)}}a_{k,k}^{(-1)}=a_{k,k}.
	\]
	Set $\beta_k:=\frac{(-\theta)^{q(s_k+\dots+s_r)}}{\prod_{i=k}^r\alpha_i}$. Since $\beta_k \in \TT_{\Sigma}^{\times}$, by \eqref{E:omega2} we can pick $a_{k,k}=\omega_{\beta_k}\in \TT_{\Sigma}^{\times}$. Note also that $\dnorm{\omega_{\beta_k}}=q^{\frac{q(s_k+\dots+s_r)-(|U_k|+\dots+|U_r|)}{q-1}}$.
	
	Now we need to find $a_{k,k-1}$ such that
	\[
	\frac{(-\theta)^{s_{k-1}+\dots+s_r}}{\prod_{i=k-1}^r\alpha_i^{(-1)}}(a_{k,k-1}^{(-1)}-Q_{U_{k-1,s_{k-1}}}^{(-1)}a_{k,k}^{(-1)})=a_{k,k-1}.
	\]
	Since $\dnorm{Q_{U_{k-1},s_{k-1}}}<q^{\frac{qs_{k-1}-|U_{k-1}|}{q-1}}$, we have that the element $a_{k,k-1}$ is in $\TT_{\Sigma}$ and $\dnorm{a_{k,k-1}}\leq q^{\frac{q(s_{k-1}+\dots+s_r)-(|U_{k-1}|+\dots+|U_r|)}{q-1}}$  by Lemma \ref{L:solution}. The other elements of the $k$-th row of $U$ can be found by using the same idea together with Lemma \ref{L:solution}. Thus, we determine the $k$-th row of $U$ recursively when $1\leq k \leq r$ and conclude that all elements in the $k$-th row is in $\TT_{\Sigma}$.
	
	To determine the last row, we let $s_{r+1}=0$ and $U_{r+1}=\emptyset$. Then if we apply the same idea above we see that we let $\beta_{r+1}=\frac{(-\theta)^{qs_{r+1}}}{\alpha_{r+1}^{(-1)}}=1$ and therefore we can pick $a_{r+1,r+1}=1$. We can now pick the other elements of the last row from $\TT_{\Sigma}$ by again using Lemma \ref{L:solution}.
	
	According to our selection for the elements $a_{i,j}$ we now see that 
	\[
	U=\begin{bmatrix}
	\omega_{\beta_1} &  & & & \\
	a_{2,1} & \omega_{\beta_2} & & & \\
	\vdots & & \ddots & & \\
	\vdots & & &\omega_{\beta_r} & \\
	a_{r+1,1}&\dots & \dots &a_{r+1,r}&1
	\end{bmatrix}.
	\]

	Since $U$ is a lower triangular matrix, one can obtain
	$
	\det(U)=\prod_{i=1}^r\omega_{\beta_i} \in \TT_{\Sigma}^{\times}.
	$
	Thus, we conclude that $U \in \GL_{r+1}(\TT_{\Sigma})$.
\end{proof}
\begin{theorem}\label{T:entireness} The function $\Psi_l(t):=\Psi_l$ has infinite radius of convergence. In particular the function $L_{j,l}(t)$ is well defined for any values of $t\in \TT_{\Sigma}$.
\end{theorem}
\begin{proof} We modify the ideas of the proof of Proposition 3.1.3 of \cite{ABP04}. Recall that $\Phi_l=\sum_{i=0}^m b_it^i$. By Proposition \ref{P:matrixU} there exists a matrix $U\in  \GL_{r+2-l}(\TT_{\Sigma})$ such that $U^{(-1)}b_0U^{-1}=\Id_{r+2-l}$. Now set $\Psi_l^{\prime}:=U\Psi_l$ and $\Phi_l^{\prime}:=U^{(-1)}\Phi_l U^{-1}$ and let $\Phi_l^{\prime}=\sum_{i=0}^m b_i^{\prime}t^i$ so that $b^{\prime}_0=\Id_{r+2-l}$ and $\Psi^{\prime}_l=\sum_{i\geq 0} g^{\prime}_it^i$. By Lemma \ref{L:funceq}, one can see that $\Psi^{\prime(-1)}_l=\Phi_l^{\prime}\Psi_l^{\prime}$. Therefore for all $n\geq 1$, we have that 
	\[
	g^{\prime (-1)}_n-g^{\prime}_n=\sum_{i=1}^{\min\{n,m\}}b^{\prime}_ig{\prime}_{n-i}.
	\]
	Since $\Psi\in \Mat_{(r+2-l)\times 1}(\TT_{\Sigma,t})$ and $U\in \GL_{(r+2-l)}(\TT_{\Sigma})$, $\lim_{n\to \infty}\dnorm{g^{\prime}_n}=0$. Let us set 
	\[
	\tilde{g}_n:=\sum_{\nu=1}^{\infty}\big(\sum_{i=1}^mb^{\prime}_ig^{\prime}_{n-i}\big)^{(\nu)}.
	\]
	Note that $\tilde{g}_n$ also converges for all sufficiently large $n$ and we have that $\lim_{n\to \infty}\dnorm{\tilde{g}_n}=0$. We also have that 
	\[
	\tilde{g}_n^{(-1)}=\sum_{i=1}^mb^{\prime}_ig^{\prime}_{n-i}+\tilde{g}_n=g_n^{\prime(-1)}-g^{\prime}_n+\tilde{g}_n.
	\]
	Thus $(\tilde{g}_n-g^{\prime}_n)^{(-1)}=\tilde{g_n}-g^{\prime}_n$ for $n\gg0$. When $n\gg 0$, we see that $\tilde{g_n}-g^{\prime}_n$ is invariant under twisting, thus by \cite[Lem. 2.5.1]{GP}, we have that $\tilde{g_n}-g_n\in \mathbb{F}_q[\underline{t}_{\Sigma}]$. But for sufficiently large $n$ the norm of $\tilde{g_n}-g^{\prime}_n$ is arbitrarily small. Therefore we conclude that $\tilde{g_n}-g^{\prime}_n=0$. Therefore for $n\gg 0$ and a fixed real number $C>1$ we have that
	\begin{align*}
	C^n\dnorm{g^{\prime}_n}&\leq \max_{i=1}^mC^n\dnorm{b^{\prime}_i}^q\dnorm{g^{\prime}_{n-i}}^q\\\
	&\leq (\max_{i=1}^mC^i\dnorm{b^{\prime}_i}^q)(\max_{i=1}^m\dnorm{g^{\prime}_{n-i}})^{q-1}\max_{i=1}^mC^{n-i}\dnorm{g^{\prime}_{n-i}}\\
	&\leq \max_{i=n-m}^{n-1}C^i\dnorm{g^{\prime}_i}
	\end{align*}
	where the last inequality comes from the fact that  
	\[
	(\max_{i=1}^mC^i\dnorm{b^{\prime}_i}^q)(\max_{i=1}^m\dnorm{g^{\prime}_{n-i}})^{q-1}\leq 1
	\]
	when $n$ is sufficiently large.
	Therefore $\sup_{n=0}^{\infty}C^n\dnorm{g^{\prime}_n}< \infty$ and it implies that all entries of $\Psi_l^{\prime}=U\Psi_l$ has infinite radius of convergence. Multiplying the column matrix  $\Psi_l^{\prime}$  containing functions with infinite radius of convergence with $U^{-1}$ from the left then implies that functions in the entries of $\Psi_l$ have also infinite radius of convergence. 
\end{proof}

\begin{proof}[{Proof of Theorem \ref{T:polylogarithm}}] 
	The proof uses the ideas from the proof of Theorem 5.5.2 of \cite{Chang}. Let $\mathcal{C}$ be the composition array as in \eqref{E:star00} so that $\mathfrak{I}_1$ is the set of indices $i$ such that $U_i\neq \emptyset$ and $\mathfrak{I}_2$ is the set of $i$'s such that $U_i=\emptyset$. We denote
	$L_{r+1}(t):=L_{r+1,1}(t)$.
	By Theorem \ref{T:entireness} we have that the function $L_{r+1}(t)$ has infinite radius of convergence. Recall that
	\begin{multline}
	L_{r+1}(t):=\Omega^{s_1+\dots +s_r}(t)\prod_{i=1}^r\omega_{U_i}\times \\
	\sum_{i_1>i_2>\dots>i_r\geq 0}\frac{Q_{U_1,s_1}^{(i_1)}(t)\dots Q^{(i_r)}_{U_r,s_r}(t)b_{i_1}(U_1)\dots b_{i_r}(U_r)}{((t-\theta^q)\dots (t-\theta^{q^{i_1}}))^{s_1}\dots ((t-\theta^q)\dots (t-\theta^{q^{i_r}}))^{s_r} }.
	\end{multline}
	Since $L_{r+1}(t)$ is well-defined at $t=\theta$, by Theorem \ref{T:Demeslay}, Theorem \ref{T:Demeslay2} and the equalities \eqref{E:Omega1}, \eqref{E:Omega2} and \eqref{E:omega1},  we obtain
	\begin{equation}\label{E:1}
	\begin{split}
	L_{r+1}(\theta)&=\sum_{i_1>i_2>\dots>i_r\geq 0}(\omega_{U_1}\Omega^{s_1}Q_{U_1,s_1})^{(i_1)}(\theta) \dots (\omega_{U_r}\Omega^{s_r}Q_{U_r,s_r})^{(i_r)}(\theta)\\
	&=\frac{\prod_{i=1}^r\omega_{U_i}\prod_{i\in \mathfrak{I}_1} \ell_{r_{s_i}-1}^{q^{r_{s_i}}-s_i}b_{r_{s_i}}(U_i)\prod_{i\in \mathfrak{I}_2}\Gamma_{s_i}}{\tilde{\pi}^{s_1+\dots+s_r}}\times\\ &\ \ \ \ \ \ \ \ \ \ \ \ \ \ \ \ \ \ \ \ \sum_{i_1>i_2>\dots>i_r\geq 0}S_{i_1}(U_1,s_1)\dots S_{i_r}(U_r,s_r) \\
	&=\frac{\prod_{i=1}^r\omega_{U_i}\prod_{i\in \mathfrak{I}_1} \ell_{r_{s_i}-1}^{q^{r_{s_i}}-s_i}b_{r_{s_i}}(U_i)\prod_{i\in \mathfrak{I}_2}\Gamma_{s_i}}{\tilde{\pi}^{s_1+\dots+s_r}}\zeta_{C}(\mathcal{C}).
	\end{split}
	\end{equation}
	Observe that 
	\begin{multline}\label{E:2}
	\frac{L_{r+1}(t)}{\prod_{i=1}^r\omega_{U_i}\Omega^{s_1+\dots +s_r}(t)}\\=\sum_{i_1>i_2>\dots>i_r\geq 0}\frac{Q_{U_1,s_1}^{(i_1)}(t)\dots Q^{(i_r)}_{U_r,s_r}(t)b_{i_1}(U_1)\dots b_{i_r}(U_r)}{((t-\theta^q)\dots (t-\theta^{q^{i_1}}))^{s_1}\dots ((t-\theta^q)\dots (t-\theta^{q^{i_r}}))^{s_r} }.
	\end{multline}
	Applying $t=\theta$ to both sides of \eqref{E:2} and combining it with \eqref{E:1} by using \eqref{E:Omega2}, we get that 
	\begin{align*}
	\zeta_{C}(\mathcal{C})\prod_{i\in \mathfrak{I}_1} \ell_{r_{s_i}-1}^{q^{r_{s_i}}-s_i}b_{r_{s_i}}(U_i)\prod_{i\in \mathfrak{I}_2}\Gamma_{s_i}=\sum_{i \in \mathfrak{I}}a_{i}\Li_{\mathcal{C}}(u_i)
	\end{align*}
	where the right hand side is justified by the fact that $\dnorm{Q_{U_j,s_j}(t)}<q^{\frac{s_jq-|U_j|}{q-1}}$ for any $j\in \{1,\dots,r\}$ and therefore $\Li_{\mathcal{C}}(u)$ converges for any $u\in S$.
\end{proof}	
\subsection{Multiple Star Polylogarithms}
Let $\mathcal{C}$ be a composition array as in \eqref{E:star00} such that $r>1$. Recall that $n_l=|U_l|$ for $1\leq l \leq r$. We define the subset $D^{\prime \prime}_{\mathcal{C}}\subset \TT_{\Sigma}^r$ by 
\[
\{f=(f_1,\dots,f_r)\in \TT_{\Sigma}^r \ \ | \ \ \dnorm{f_1}< q^{\frac{s_1q-n_1}{q-1}} \  ,\dnorm{f_i}\leq  q^{\frac{s_iq-n_i}{q-1} } \ \ \text{for } i=2,\dots,r \}
\] 

Inspired by the work of Chang and Mishiba in \cite[Sec. 2.2]{ChangMishibaOct}, for  $u=(u_1,\dots,u_r)\in D^{\prime \prime}_{\mathcal{C}}$, we define the multiple star polylogarithm $\Li_{\mathcal{C}}^{*}(u)$ corresponding to the composition array $\mathcal{C}$ by the infinite series
\[
\Li_{\mathcal{C}}^{*}(u)=\sum_{i_1\geq i_2\geq \dots>i_r\geq 0}\frac{b_{i_1}(U_1)\dots b_{i_r}(U_r)\tau^{i_1}(u_1)\dots \tau^{i_r}(u_r)}{\ell_{i_1}^{s_1}\dots \ell_{i_r}^{s_r} }.
\]
Observe that $\Li^{*}_{\mathcal{C}}(u)$ converges in $\TT_{\Sigma}$ if $u\in D^{\prime \prime}_{\mathcal{C}}$.

Let $\mathcal{C}_i=\binom{U_i}{s_i}$ be a composition array for all $1\leq i \leq r$. We define the addition `$+$' between composition array $\mathcal{C}_i$ and $\mathcal{C}_j$ by
\[
\mathcal{C}_i+\mathcal{C}_j=\binom{U_i\cup U_j}{s_i+s_j}
\]
and the operation `$,$' by
\[
(\mathcal{C}_i,\mathcal{C}_j)=\binom{U_i, U_j}{s_i,s_j}
\]
Observe that $(\mathcal{C}_1,\dots,\mathcal{C}_r)=\mathcal{C}$.

Now similar to \cite[Sec. 5.2]{ChangMishibaOct}, we define the set $S$ whose elements are symbols `$,$' and `$+$' and the set $S^{\times}$ containing symbols `$,$' and `$\times $'. We define the map 
\[
f:S^{r-1}\to S^{\times r-1}
\]
in a way that $f$ fixes the symbol `$,$' and sending `$+$' to `$\times$'. As an example, if $v=(v_1,v_2)\in S^2$ so that $v_1=$ `$,$' and $v_2=$ `$+$' then $f(v)=v^{\prime}=(v_1^{\prime},v_2^{\prime})$ where $v_1^{\prime}=$ `$,$' and $v_2^{\prime}=$ `$\times$'.

We continue with further definition. For any $v=(v_1,\dots,v_{r-1})\in S^{r-1}$ and any composition array $\mathcal{C}=(\mathcal{C}_1,\dots, \mathcal{C}_r)$, we define $v(\mathcal{C})=(\mathcal{C}_1v_1\mathcal{C}_2v_2\dots v_{r-1}\mathcal{C}_r)$. For any $u=(u_1,\dots,u_r)\in \TT_{\Sigma}^{r}$, we set $v^{\times}(u):=(u_1f(v_1)u_2f(v_2)\dots u_rf(v_r))$.
We also set $\gamma(v)$ to be the number of `$+$' in $v$.
As an example, let $v=(v_1,v_2,v_3)$ be such that $v_1=$`$,$', $v_2=$`$+$' and $v_3=$ `$,$'. Let 
\[
\mathcal{C}=\binom{U_1, U_2, U_3, U_4}{s_1, s_2, s_3, s_4}
\]
and $\mathcal{C}_i=\binom{U_i}{s_i}$ for $i=1,2,3,4$. Then 
\[
v(\mathcal{C})=\binom{U_1, U_2\cup U_3, U_4}{s_1, s_2+s_3, s_4}
\]
and $v^{\times}(u)=(u_1,u_2u_3,u_4)$. Note also that $\gamma(v)=1$.

Observe that by the properties of non-archimedean geometry if $u\in D_{\mathcal{C}}^{\prime}$ ($D_{\mathcal{C}}^{\prime\prime}$ resp.) then $v^{\times}(u)$ is also in  $D_{\mathcal{C}}^{\prime}$ ($D_{\mathcal{C}}^{\prime\prime}$ resp.) for any $v=(v_1,\dots,v_{r-1})\in S^{r-1}$. Finally for $r=1$ we define $S^{r-1}=S^0$ and for any $v\in S^{0}$ we have $v(\mathcal{C}):=\mathcal{C}$, $v^{\times}(u):=u$ and $\gamma(v):=0$.

Now using the inclusion-exclusion principle on the set $\{i_1\geq i_2\geq \dots i_r\geq 0\}$ as in \cite[Prop. 5.2.3]{ChangMishibaOct} we have that
\begin{equation}\label{E:star0}
\Li_{\mathcal{C}}(u)=\sum_{v\in S^{r-1}} (-1)^{\gamma(v)}\Li^{*}_{v(\mathcal{C})}(v^{\times}(u)).
\end{equation}

\begin{theorem}[{cf. \cite[Thm.~5.2.5]{ChangMishibaOct}}]\label{T:star0}
	Let $\mathcal{C}$ be a composition array  of depth $r$ as in \eqref{E:star00}. Let $\mathfrak{I}_1$ be the set of indices $i$ such that $U_i\neq \emptyset$ and $\mathfrak{I}_2$ be the set of $i$ such that $U_i=\emptyset$. Then there exist composition arrays $\mathcal{C}_{l}$ with $\wght(\mathcal{C})=\wght(\mathcal{C}_l)$ and $\dep(\mathcal{C}_l)\leq r$, elements $a_l\in A$ and $u_l\in K^{\text{perf}}(\underline{t}_{\Sigma})^{\dep(\mathcal{C}_l)}$ such that
	\[
	\prod_{i\in \mathfrak{J}_2}\Gamma_{s_i} \prod_{i\in \mathfrak{J}_1} \ell_{r_{s_i}-1}^{q^{r_{s_i}}-s_i}b_{r_{s_i}}(U_i)\zeta_{C}(\mathcal{C})=\sum_{l}a_l(-1)^{\dep(\mathcal{C}_l)-1}\Li^{*}_{\mathcal{C}_l}(u_l)
	\]
	where $r_{s_i}\geq 1$ is an integer such that $s_i\leq q^{r_{s_i}}$ for $i\in \mathfrak{I}_1$.
\end{theorem}
\begin{proof} Using Theorem \ref{T:polylogarithm} and \eqref{E:star0} we observe that 
	\begin{align*}
	&\prod_{i\in \mathfrak{J}_2}\Gamma_{s_i} \prod_{i\in \mathfrak{J}_1} \ell_{r_{s_i}-1}^{q^{r_{s_i}}-s_i}b_{r_{s_i}}(U_i)\zeta_{C}(\mathcal{C})\\
	& \ \ \  \ \ \ \ \ \ \ \ \ =\sum_{i\in \mathfrak{I}}a_i\Li_{\mathcal{C}}(u_i)\\
	& \ \ \  \ \ \ \ \ \ \ \ \ =\sum_{i\in \mathfrak{I}}a_i\sum_{V\in S^{r-1}}(-1)^{\gamma(V)}\Li^{*}_{V(\mathcal{C})}(V^{\times}(u_i))\\
	& \ \ \  \ \ \ \ \ \ \ \ \ =\sum_{i\in \mathfrak{I}}\sum_{V\in S^{r-1}}(-1)^{r-1}a_i(-1)^{\dep(V(\mathcal{C}))-1}\Li^{*}_{V(\mathcal{C})}(V^{\times}(u_i))
	\end{align*}
	where the last line follows from the fact that $\gamma(V)+\dep(V(\mathcal{C}))=r$ for $V\in S^{r-1}$. Thus we can define the tuple $(a_l,\mathcal{C}_l,u_l)$ by
	\[
	(a_l,\mathcal{C}_l,u_l)=((-1)^{r-1}a_i, V(\mathcal{C}),V^{\times}(\mathcal{C}))
	\]
	for $i\in \mathfrak{I}$ and $V\in S^{r-1}$.
\end{proof}

\section{Higher Dimensional Drinfeld modules over Tate algebras}
\subsection{Anderson $A[\underline{t}_{\Sigma}]$-modules} The idea of Drinfeld modules over Tate algebras was first introduced by  Angl\`{e}s, Pellarin and Tavares Ribeiro \cite{APTR}. In this section, we  introduce the concept of Anderson $A[\underline{t}_{\Sigma}]$-modules which can be seen as higher dimensional Drinfeld modules over Tate algebras. 

\begin{definition}\label{D:sec40} An Anderson $A[\underline{t}_{\Sigma}]$-module $\phi:A[\underline{t}_{\Sigma}]\to \Mat_{n}(\TT_{\Sigma})[\tau]$ of dimension $n$ defined over $\TT_{\Sigma}$ is an $\mathbb{F}_q[\underline{t}_{\Sigma}]$-linear homomorphism such that
	\[
	\phi(\theta)=A_0+A_1\tau +\dots +A_s\tau^s
	\]
	for some $s$ and $(\theta \Id_{n}-A_0)^n=0$.
	
	Any Anderson $A[\underline{t}_{\Sigma}]$-module $\phi$ defines an $A[\underline{t}_{\Sigma}]$-module action on $\Mat_{n \times 1}(\TT_{\Sigma})$ given by
	\[
	\phi(\theta)\cdot f=A_0f+A_1\tau(f)+\dots +A_s\tau^s(f),\quad f\in \Mat_{n \times 1}(\TT_{\Sigma}).
	\]
\end{definition} 
We also define the $\mathbb{F}_q[\underline{t}_{\Sigma}]$-linear homomorphism $\partial_{\phi}:A[\underline{t}_{\Sigma}]\to \Mat_{n}(\TT_{\Sigma})$ by 
$\partial_{\phi}(\theta)=A_0$ and its action on $ \Mat_{n\times 1}(\TT_{\Sigma})$  by $\partial_{\phi}(\theta)\cdot f=A_0f$ for any $f\in \Mat_{n \times 1}(\TT_{\Sigma})$.

\begin{remark}We note that any $t$-module in the sense of Anderson \cite{And86} can be also seen as an Anderson $A[\underline{t}_{\Sigma}]$-module over $\CC_{\infty}\subset \TT_{\Sigma}$. 
\end{remark}
Let $\phi_1$ and $\phi_2$ be Anderson $A[\underline{t}_{\Sigma}]$-modules of dimension $n_1$ and $n_2$ respectively. An Anderson $A[\underline{t}_{\Sigma}]$-module homomorphism $\varphi:\phi_1\to \phi_2$ is defined as an element $\varphi\in\Mat_{n_2 \times n_1}(\TT_{\Sigma})[\tau]$ such that 
\[
\phi_2(\theta)\varphi=\varphi\phi_1(\theta).
\]

We now discuss the exponential and logarithm function of some class of Anderson $A[\underline{t}_{\Sigma}]$-modules. It is important to point out that one can define an exponential and logarithm function corresponding to any Anderson $A[\underline{t}_{\Sigma}]$-module using methods of Anderson \cite{And86} but concerning the purpose of the present paper, we only analyze special cases.

Let $\phi$ be an Anderson $A[\underline{t}_{\Sigma}]$-module defined by
\begin{equation}\label{E:sec40}
\phi(\theta)=\theta \Id_{n} + N +E\tau
\end{equation}
for some $N\in \Mat_{n}(\mathbb{F}_q)$ such that $N^n=0$ and $E\in \Mat_{n}(\TT_{\Sigma})$. 

For any square matrices $X_1$ and $X_2$, we first define $[X_1,X_2]:=X_1X_2-X_2X_1$. Then we set $\ad(X_1)^{0}(X_2):=X_2$ and for $j\geq 1$, $\ad(X_1)^{j}(X_2):=[X_1,\ad(X_1)^{j-1}(X_2)].$ 
\begin{lemma}\cite[Lem. 3.2.9]{PapanikolasLog}\label{L:Sec40}
	Let $Y, N\in \Mat_{n}(\TT_{\Sigma})$. Then we have 
	\[
	\ad(N)^j(Y)=\sum_{m=0}^{j}(-1)^m\binom{j}{m}N^mYN^{j-m}.
	\]
	Moreover, if $N$ is a nilpotent matrix so that $N^n=0$, then $\ad(N)^j(Y)=0$ for $j>2n-2$.
\end{lemma}
\begin{proof} Using the definition of $\ad(N)^j(Y)$ and an induction argument imply the above formula. Now assume that $N$ is a nilpotent matrix such that $N^n=0$. Thus for $j>2n-2$ and $0\leq m \leq j$, we have $N^{m}YN^{j-m}$ is $0$ as either $m\geq n$ or $j-m\geq n$ and therefore we have either $N^{m}=0$ or $N^{j-m}=0$.  
\end{proof}
\begin{proposition}\label{P:sec40}
	Let 
	\[
	\exp_{\phi}=\sum_{i\geq 0}\beta_i \tau^i\in \Mat_{n}(\TT_{\Sigma})[[\tau]]
	\]
	be the infinite series such that $\beta_0=\Id_{n}$ and 
	\begin{equation}\label{E:sec41}
	\exp_{\phi}\partial_{\phi}(\theta)=\phi(\theta)\exp_{\phi}
	\end{equation}
	holds in $\Mat_{n}(\TT_{\Sigma})[[\tau]]$.
	Then we have
	\begin{equation}\label{E:sec43}
	\beta_{i+1}=\sum_{j=0}^{2n-2}\frac{\ad(N)^j(E\beta_i^{(1)})}{[i+1]^{j+1}}.
	\end{equation}
	Moreover $\dnorm{\beta_{i}}\leq \dnorm{\theta}^{-iq^i}\dnorm{E}^i$ for $i\geq 0$.
\end{proposition}
\begin{proof}
	By comparing the coefficients of $\tau$ in \eqref{E:sec41} we see that 
	\begin{equation}\label{E:sec42}
	\beta_{i+1}(\theta^{q^{i+1}}\Id_{n}+N)=\theta \Id_{n}\beta_{i+1}+N\beta_{i+1}+E\beta_{i}^{(1)}
	\end{equation}
	After some arrangement we see that \eqref{E:sec42} can be rewritten as 
	\[
	\beta_{i+1}=\frac{[N,\beta_{i+1}]}{\theta^{q^{i+1}}-\theta}+\frac{E\beta_i^{(1)}}{\theta^{q^{i+1}}-\theta}.
	\]
	Thus a similar calculation as in  \cite[Eq. 2.2.3]{AndThak90} implies that the formula for $\beta_{i+1}$ in \eqref{E:sec43} holds. Now we claim that $\dnorm{\beta_{i}}\leq \dnorm{\theta}^{-iq^i}\dnorm{E}^i$ for $i\geq 0$. We do induction on $i$. If $i=0$, then $\beta_0=\Id_{n}$ and the claim holds. Assume it also holds for $i$. Then by the induction hypothesis and \eqref{E:sec43} we have 
	\begin{align*}
	\dnorm{\beta_{i+1}}&\leq \max_{0\leq j \leq 2n-2}\frac{\dnorm{\beta_{i}^{(1)}}\dnorm{E}}{[i+1]^{j+1}}\\
	&=\dnorm{\theta}^{-iq^{i+1}}\dnorm{E}^{i+1}\dnorm{\theta}^{-q^{i+1}}\\
	&=\dnorm{\theta}^{-(i+1)q^{i+1}}\dnorm{E}^{i+1}
	\end{align*}
	as desired.
\end{proof}
We call the infinite series $\exp_{\phi}=\sum \beta_i\tau^i\in \Mat_{n}(\TT_{\Sigma})[[\tau]]$ in Proposition \ref{P:sec40} the exponential series of $\phi$. The exponential series $\exp_{\phi}$ induces an $\mathbb{F}_q[\underline{t}_{\Sigma}]$-linear homomorphism $\exp_{\phi}:\Mat_{n\times 1}(\TT_{\Sigma})\to \Mat_{n\times 1}(\TT_{\Sigma})$ defined by 
\[
\exp_{\phi}(f)=\sum_{i\geq 0}\beta_i \tau^i(f) , \ \ f\in \Mat_{n\times 1}(\TT_{\Sigma}).
\]
Moreover by Proposition \ref{P:sec40} we see that the function $\exp_{\phi}$ converges everywhere on $\Mat_{n\times 1}(\TT_{\Sigma})$.

Using a similar argument as in the proof of \cite[Thm. 3.3.2]{GP} together with Proposition \ref{P:sec40}, we deduce the following lemma.

\begin{lemma}[{cf.~\cite[Lem.~3.3.2]{GP}}]\label{L:sec4iso}
	Let $\phi$ be the Anderson $A[\underline{t}_{\Sigma}]$-module defined as in \eqref{E:sec40}. Then there exists $\varepsilon_{\phi} > 0$ such that the open ball $\{f\in \Mat_{n\times 1}(\TT_{\Sigma}) \ \ | \ \   \dnorm{f} < \varepsilon_{\phi} \}\subset \Mat_{n\times 1}(\TT_{\Sigma})$ can be mapped $\dnorm{\cdot}$-isometrically by $\exp_{\phi}$ to itself.
\end{lemma}

Let $P_0=\Id_{n}$. We define the logarithm series 
\[
\log_{\phi}=\sum_{j \geq 0}P_j\tau^j\in \Mat_{n}(\TT_{\Sigma})[[\tau]]
\]
as the formal inverse of $\exp_{\phi}$ in $\Mat_{n}(\TT_{\Sigma})[[\tau]]$ such that \begin{equation}\label{E:loginv}
\exp_{\phi}\log_{\phi}=\log_{\phi}\exp_{\phi}=\Id_{n}
\end{equation}
and it also satisfies
\begin{equation}\label{E:sec45}
\partial_{\phi}(\theta)\log_{\phi}=\log_{\phi}\phi(\theta).
\end{equation}
The logarithm series $\log_{\phi}$ induces an $\mathbb{F}_q[\underline{t}_{\Sigma}]$-linear homomorphism $\log_{\phi}:\Mat_{n\times 1}(\TT_{\Sigma})\to \Mat_{n\times 1}(\TT_{\Sigma})$ defined by 
\[
\log_{\phi}(f)=\sum_{i\geq 0}P_i \tau^i(f), \ \ f\in \Mat_{n\times 1}(\TT_{\Sigma})
\]
which has a finite radius of convergence by Lemma \ref{L:sec4iso}. It also implies that $\exp_{\phi}$ is an automorphism with its inverse $\log_{\phi}$ on $\{f\in \Mat_{n\times 1}(\TT_{\Sigma}) \ \ | \ \   \dnorm{f} < \varepsilon_{\phi} \}\subset\Mat_{n\times 1}(\TT_{\Sigma})$. 

Using a similar calculation to  \cite[Eq. (3.2.4)]{ChangMishibaApr} we can obtain 
\begin{equation}\label{E:sec47}
P_{i+1}=-\sum_{j=0}^{2d^{\prime}-2}\frac{\ad(N)^j(P_i E^{(i)})}{(\theta^{q^{i+1}}-\theta)^{j+1}}
\end{equation}
where $d^{\prime}$ is a positive integer such that $N^{d^{\prime}}=0$ and $N^k\neq 0$ for $k<d^{\prime}$.

\subsection{Uniformizability}
We now discuss the uniformizability of Anderson $A[\underline{t}_{\Sigma}]$-modules. We can refer the reader to \cite[Sec. 3,6]{APTR} and \cite{GP} for more details about uniformizability of Anderson $A[\underline{t}_{\Sigma}]$-modules of dimension 1.
\begin{definition} We call an Anderson $A[\underline{t}_{\Sigma}]$-module $\phi$ of dimension $n$ uniformizable if $\exp_{\phi}:\Mat_{n \times 1}(\TT_{\Sigma})\to \Mat_{n \times 1}(\TT_{\Sigma})$ is a surjective function.
\end{definition}
\begin{example}\label{Ex:sec40}
	Assume that $\Sigma=\{1,\dots,s\}$ for some $s\in \mathbb{Z}_{\geq 1}$ and let $f=\sum_{\mu \in \mathbb{Z}_{\geq 1}^{|\Sigma|}} f_{\mu}\underline{t}_{\Sigma}^{\mu}$ be in $\TT_{\Sigma}$. For any positive integer $n$, we define the Anderson $A[\underline{t}_{\Sigma}]$-module $C^{\otimes n}:A[\underline{t}_{\Sigma}]\to \Mat_{n}(\TT_{\Sigma})[\tau]$, the $n$-th tensor power of the Carlitz module by 
	\[
	C^{\otimes n}(\theta)=\theta \Id_{n} + \begin{bmatrix}
	0&1&\dots &0 \\
	&\ddots&\ddots&\vdots\\
	& &\ddots&1\\
	& & &0
	\end{bmatrix} + \begin{bmatrix}
	0&\dots&\dots & 0\\
	\vdots & & &\vdots \\
	\vdots & & & \vdots \\
	1&\dots&\dots & 0
	\end{bmatrix}\tau.
	\]
	When $n=1$, we call  $C:A[\underline{t}_{\Sigma}]\to \TT_{\Sigma}[\tau]$ given by $C(\theta)=\theta +\tau$ the Carlitz module (see \cite[\S 3]{Goss} for details). By \cite[Sec. 4.3]{BPrapid} we know that $\exp_{C^{\otimes n}}:\Mat_{n\times 1}(\CC_{\infty}) \to \Mat_{n\times 1}(\CC_{\infty})$ is surjective. Indeed we can show that $\exp_{C^{\otimes n}}:\Mat_{n\times 1}(\TT_{\Sigma}) \to \Mat_{n\times 1}(\TT_{\Sigma})$ is also surjective as follows. Let
	\[
	y=\big[
	\sum a_{1,\mu}\underline{t}_{\Sigma}^{\mu},
	\dots,
	\sum a_{n,\mu}\underline{t}_{\Sigma}^{\mu}
	\big]^{\intercal}\in \Mat_{n\times 1}(\TT_{\Sigma})
	\] 
	for some $a_{j,\mu}\in \CC_{\infty}$ where $1\leq j \leq n$ and $\mu \in \ZZ_{\geq 0}^{|\Sigma|}$. Then for any $\mu$, there exists $x_{\mu}=[x_{1,\mu},\dots,x_{n,\mu}]^{\intercal}\in \CC_{\infty}$ such that 
	$
	\exp_{C^{\otimes n}}(x_{\mu})=[
	a_{1,\mu},
	\dots,
	a_{n,\mu}
	]^{\intercal}.
	$
	Note that the entries of $y$ are elements in $\TT_{\Sigma}$. Thus by Lemma \ref{L:sec4iso}, for any $j$, there exists $N_j\in \mathbb{Z}$ such that $a_{j,\mu}$ is in the radius of convergence of $\log_{C^{\otimes n}}$ for any $s$-tuple $\mu$ whose sum of the entries is bigger than $N_j$. Thus for such tuple $(a_{1,\mu},\dots,a_{n,\mu})^{\intercal}$, we can choose $x_{\mu}=\log_{C^{\otimes n}}((a_{1,\mu},\dots,a_{n,\mu})^{\intercal})$  such that $\dnorm{(a_{1,\mu},\dots,a_{n,\mu})^{\intercal}}=\dnorm{x_{\mu}}$ and $
	\exp_{C^{\otimes n}}(x_{\mu})=[
	a_{1,\mu},
	\dots,
	a_{n,\mu}
	]^{\intercal}
	$. Therefore we guarantee that the element 
	\[
	x:=\begin{pmatrix}
	\sum x_{1,\mu}\underline{t}_{\Sigma}^{\mu} \\
	\vdots \\
	\sum x_{n,\mu}\underline{t}_{\Sigma}^{\mu}
	\end{pmatrix}
	\]
	lives in $\Mat_{n\times 1}(\TT_{\Sigma})$. Furthermore, by the $\mathbb{F}_q[\underline{t}_{\Sigma}]$-linearity of $\exp_{C^{\otimes n}}$, we have $\exp_{C^{\otimes n}}(x)=y$. 
\end{example}

Recall that $\Sigma=\cup_{i=1}^r U_i\subset \ZZ_{\geq 1}$. We define $\alpha:=\prod_{i=1}^{r}\prod_{j\in U_i}(t_j-\theta)$. Following Demeslay \cite[Sec. 4.1]{Demeslay14}, we define the Anderson $A[\underline{t}_{\Sigma}]$-module $C_{\alpha}^{\otimes n}:A[\underline{t}_{\Sigma}]\to \Mat_{n}(\TT_{\Sigma})[\tau]$ by
\begin{equation}\label{E:sec44}
C_{\alpha}^{\otimes n}(\theta)=\theta \Id_{n} + \begin{bmatrix}
0&1&\dots &0 \\
&\ddots&\ddots&\vdots\\
& &\ddots&1\\
& & &0
\end{bmatrix} + \begin{bmatrix}
0&\dots&\dots & 0\\
\vdots & & &\vdots \\
\vdots & & & \vdots \\
\alpha&\dots&\dots & 0
\end{bmatrix}\tau.
\end{equation}

Note that if $\Sigma=\emptyset$, then $C_{\alpha}^{\otimes n}=C^{\otimes n}$ for any positive integer $n$. As an example, when $r=1$, $n=1$ and $\Sigma=U_1=\{1,\dots,s\}$ for some $s\in \mathbb{Z}_{\geq 1}$, we have that 
\begin{equation}\label{E:sec46}
C_{\alpha}(\theta)=\theta +(t_1-\theta)\dots (t_s-\theta)\tau,
\end{equation}
and when $r=1$, $n=2$ and $U_1=\{1,2\}$ we have
\[
C_{\alpha}^{\otimes 2}(\theta)=\begin{bmatrix}
\theta&1\\
0&\theta
\end{bmatrix}+\begin{bmatrix}
0&0\\
(t_1-\theta)(t_2-\theta) &0
\end{bmatrix}\tau.
\]

\subsection{Frobenius Modules}
We now investigate the idea of Frobenius modules in our setting. Alert reader might notice that the terminology was also used in \cite[Sec. 2.2]{CPY} and \cite[Sec. 4]{GP}.

We define the non-commutative polynomial ring $\TT_{\Sigma}[\sigma]$ subject to the relation
\[\sigma f=f^{(-1)}\sigma, \quad f\in \TT_{\Sigma}[\sigma].\]
Let $d,n\in \ZZ_{\geq 1}$. For any $M=(M_{ij})\in \Mat_{d \times n}(\TT_{\Sigma})$, we define the non-archimedean norm $\lVert \cdot \rVert_{\sigma}$ by
\[
\lVert M \rVert_{\sigma}:=\sup \{\dnorm{M_{ij}}\}.
\]
Let $\varphi=A_0+A_1\tau+\dots +A_k\tau^k\in \Mat_{d\times n}(\TT_{\Sigma})[\tau]$. Then we define the map $*:\Mat_{d\times n}(\TT_{\Sigma})[\tau]\to \Mat_{n\times d}(\TT_{\Sigma})[\sigma]$ by
\[
\varphi^{*}=A_0^{\intercal}+A_1^{\intercal (-1)}\sigma +\dots +A_k^{\intercal (-k)}\sigma^k.
\]
We also define the ring $\TT_{\Sigma}[t,\sigma]:=\TT_{\Sigma}[t][\sigma]$ with respect to the condition
\[
ct=tc, \ \ c^{(-1)}\sigma=\sigma c, \ \ t\sigma=\sigma t, \ \  c\in \TT_{\Sigma}.
\]
\begin{definition}
	Let $\phi$ be an Anderson $A[\underline{t}_{\Sigma}]$-module of dimension $n$ defined as in \eqref{E:sec40}. We call a $\TT_{\Sigma}[t,\sigma]$-module $H(\phi)$ the Frobenius module corresponding to $\phi$ if it is free of rank $n$ over $\TT_{\Sigma}[\sigma]$ and its $\TT_{\Sigma}[t]$-action is defined by
	\[
	ct\cdot h=ch\phi(\theta)^{*}=ch(\theta \Id_{n} + N^{\intercal } +E^{\intercal (-1)}\sigma) 
	\]
	for any $h\in H(\phi)$ and $c\in \TT_{\Sigma}$.
\end{definition}
\begin{example}\label{Ex:sec41} Let $\alpha=\prod_{i\in \Sigma}(t_i-\theta)$. Consider the left  $\TT_{\Sigma}[t,\sigma]$-module $H:=\TT_{\Sigma}[t]$ whose $\TT_{\Sigma}[\sigma]$-action is given by $c\sigma\cdot f=cf^{(-1)}\alpha^{(-1)^{-1}}(t-\theta)^n$ for any $c\in \TT_{\Sigma}$ and $f\in H$. One can see that $H$ is free of rank $n$ over $\TT_{\Sigma}[\sigma]$ with the basis $\{v_n,\dots,v_1\}$ where $v_i=(t-\theta)^{i-1}$ for $1\leq i \leq n$.   Let $H(C_{\alpha}^{\otimes n}):=\Mat_{1\times n}(\TT_{\Sigma}[\sigma])$  be the Frobenius module corresponding to $C_{\alpha}^{\otimes n}$ whose  $\TT_{\Sigma}[t]$-module action is defined by $ct\cdot h=ch(C_{\alpha}^{\otimes n})^{*}$ for any $c\in \TT_{\Sigma}$ and $h\in H(C_{\alpha}^{\otimes n})$. We define $e_i\in \Mat_{1\times n}(\TT_{\Sigma}[\sigma])$ to be the row matrix whose $i$-th coordinate is 1 and the rest is zero. One notes that $\{e_n\}$ forms a $\TT_{\Sigma}[t]$-basis for $H(C_{\alpha}^{\otimes n})$. There exists a $\TT_{\Sigma}[t]$-module isomorphism $g:H\to H(C_{\alpha}^{\otimes n})$ given by $g(1)=e_{n}$. Furthermore $g$ also respects the $\TT_{\Sigma}$-linear action of $\sigma$ and therefore $g$ is a $\TT_{\Sigma}[t,\sigma]$-module isomorphism. 
\end{example}

\section{Anderson $A[\underline{t}_{\Sigma}]$-Module $G_{\mathcal{C}}$}
\subsection{The Construction of Anderson $A[\underline{t}_{\Sigma}]$-Module $G$}
For the rest of the paper, for any matrix $M\in \Mat_{k}(\TT_{\Sigma})$ of the form
\begin{equation}\label{E:block}
M=\begin{bmatrix}
M[11] & \cdots & M[1r]
\\
\vdots & & \vdots
\\
M[r1] & \cdots & M[rr]
\end{bmatrix}
\end{equation}
such that $M[ij]\in \Mat_{d_i \times d_j}(\TT_{\Sigma})$, we call $M[ij]$ the $(i,j)$-th block matrix of $M$.

We fix  a composition array $\mathcal{C}$ defined as in \eqref{E:star00} and consider $(u_1,\dots,u_r)\in (\TT_{\Sigma}\setminus \{0\})^r$. For any $1\leq j \leq r$ we set
$d_{j}:=s_{j}+\dots +s_r$ and 
$k:=d_{1}+\dots +d_r$. For each $j$, we define the matrices $N_j\in \Mat_{d_j}(\mathbb{F}_q)$ and $N\in \Mat_{k}(\mathbb{F}_q)$ by: 
\[
N_{j}=\begin{bmatrix}
0&1&\dots &0 \\
&\ddots&\ddots&\vdots\\
& &\ddots&1\\
& & &0
\end{bmatrix} , \quad  N:=\begin{bmatrix}
N_1 & & & \\
&N_2& & \\
& &\ddots & \\
& & & N_r
\end{bmatrix}.
\]
Recall from \S2.1 that $\Sigma \subset \mathbb{Z}_{\geq 1}$ is a union of finite sets given by $\Sigma=\cup_{i=1}^r U_i$ and $\alpha_{i}=\prod_{j\in U_i}(t_j-\theta)$. Set
\begin{equation}\label{E:matrixe1}
E[jm]:=\begin{bmatrix}
0&\dots & \dots &0\\
\vdots & \ddots & & \vdots \\
0& & \ddots & \vdots \\
\prod_{n=m}^r \alpha_n&0&\dots &0 
\end{bmatrix}\in \Mat_{d_j}(\TT_{\Sigma}) \text{ if $j=m$}
\end{equation}
and  if $j<m$ we define
\begin{equation}\label{E:matrixe2}
E[jm]:=\begin{bmatrix}
0&\dots & \dots &0\\
\vdots & \ddots & & \vdots \\
0& & \ddots & \vdots \\
(-1)^{m-j}\prod_{i=j}^{m-1} u_i\prod_{n=m}^r \alpha_{n}&0&\dots &0 
\end{bmatrix}\in \Mat_{d_j\times d_m}(\TT_{\Sigma}).
\end{equation}
Using \eqref{E:matrixe1} and \eqref{E:matrixe2} we define the block matrix $E$ by
\[
E:=\begin{bmatrix}
E[11] & E[12]& \dots &E[1r] \\
&E[22]&\ddots & \vdots \\
& & \ddots &E[(r-1)r]\\
& & &E[rr]
\end{bmatrix}\in \Mat_{k}(\TT_{\Sigma}).
\]
Finally we define the Anderson $A[\underline{t}_{\Sigma}]$-module $G:A[\underline{t}_{\Sigma}]\to \Mat_{k}(\TT_{\Sigma})[\tau]$ by
\[
G(\theta)=\theta \Id_{k}+N+E\tau.
\]
Set $\mathfrak{a}_j:=\prod_{i=j}^r\alpha_i$. Then we can also write the Anderson $A[\underline{t}_{\Sigma}]$-module $G$ of dimension $k$ corresponding to $\mathcal{C}$ and $(u_1,\dots,u_r)\in (\TT_{\Sigma}\setminus \{0\})^r$ as
\begin{equation}\label{E:moduleG}
G(\theta)=\begin{bmatrix}
C_{\mathfrak{a}_1}^{\otimes d_1}(\theta)&E[12]&\dots & E[1r]\\
&C_{\mathfrak{a}_2}^{\otimes d_2}(\theta)& &E[2r]\\
& & \ddots &\vdots \\
& & &C_{\mathfrak{a}_r}^{\otimes d_r}(\theta)
\end{bmatrix}.
\end{equation}
By definition, the Frobenius module $H(G)$ of $G$ carries a  $\TT_{\Sigma}[t]$-module structure such that for any $x\in H(G)$ the $t$-action is given by
\begin{equation}\label{E:actiont}
t\cdot x=xG(\theta)^{*}
\end{equation}
and is free of rank $k$ over $\TT_{\Sigma}[\sigma]$. We now claim that  $H(G)$  can be given by the direct sum of $\TT_{\Sigma}[t,\sigma]$-modules
\begin{equation}\label{E:Frobmodule}
H(G)=H(C^{\otimes d_1}_{\mathfrak{a}_1})\oplus \dots \oplus H(C^{\otimes d_r}_{\mathfrak{a}_r})
\end{equation} 
and therefore is free of rank $r$ over $\TT_{\Sigma}[t]$.  If $r=1$, then we see that $G=C^{\otimes d_1}_{\mathfrak{a}_1}$ and the Frobenius module $H(G)$ is free of rank 1 over $\TT_{\Sigma}[t]$ by Example \ref{Ex:sec41}. Suppose that $r=2$. We consider the short exact sequence 
\begin{equation}\label{E:ses}
\begin{tikzcd}[column sep=large]
0 \arrow{r} &H(C^{\otimes d_1}_{\mathfrak{a}_1}) \arrow{r}{\rho_1}
&H(G)\arrow{r}{\rho_2}
&H(C^{\otimes d_2}_{\mathfrak{a}_2})\arrow{r}&
0 
\end{tikzcd}
\end{equation}
of free  $\TT_{\Sigma}[\sigma]$-modules such that for any $x=(x_1,\dots,x_{d_1})\in H(C^{\otimes d_1})$ and $y=(y_1,\dots,y_{d_2})\in H(C^{\otimes d_2})$, we have $\rho_1(x)=(x,0,\dots,0)$ and $\rho_2(x,y)=y$. Note that
\begin{align*}
\rho_1(t\cdot x)=\rho_1(x C^{\otimes d_1 }_{\mathfrak{a}_1}(\theta)^{*})&=\rho_1((\theta x_1+ x_2,\dots,x_1\mathfrak{a}_1^{(-1)}\sigma +\theta x_{d_1}))\\
&=(\theta x_1+ x_2,\dots,x_1\mathfrak{a}_1^{(-1)}\sigma +\theta x_{d_1},0,\dots,0)\\
&=(x,0,\dots,0)G(\theta)^{*}\\
&=t\cdot \rho_1(x).
\end{align*}
Similar calculation can be applied to see that 
\begin{align*}
\rho_2(t\cdot (x,y))=t\cdot \rho_2((x,y)).
\end{align*}

Thus the maps $\rho_1$ and $\rho_2$ are compatible with the $t$-action of $\TT_{\Sigma}[t]$-modules $H(C_{\mathfrak{a}_1}^{\otimes d_1})$, $H(C_{\mathfrak{a}_2}^{\otimes d_2})$ and $H(G)$. Therefore the short exact sequence in \eqref{E:ses} is also a short exact sequence of $\TT_{\Sigma}[t]$-modules. Since $H(C_{\mathfrak{a}_1}^{\otimes d_1})$ and $H(C_{\mathfrak{a}_2}^{\otimes d_2})$ are free over $\TT_{\Sigma}[t]$ of rank 1, $H(G)$ is also free of rank 2 over $\TT_{\Sigma}[t]$ with the basis $\{m_1,m_2\}$ such that under the projection map $\proj_{i}:H(G)\to H(C^{\otimes d_i}_{\mathfrak{a}_i})$, $\proj_i(m_i)$ is a $\TT_{\Sigma}[t]$-basis for $H(C^{\otimes d_i}_{\mathfrak{a}_i})$ when $i=1,2$. 

To show the claim for any $r$, we just replace the short exact sequence in \eqref{E:ses} with
\begin{displaymath}
\begin{tikzcd}[column sep=large]
0 \arrow{r} &H(C^{\otimes d_1}_{\mathfrak{a}_1}) \arrow{r}{\rho_1}
&H(G)\arrow{r}{\rho_2}
&H(G_{r-1})\arrow{r}&
0 
\end{tikzcd}
\end{displaymath}
such that 
\[
G_{r-1}(\theta)=\begin{bmatrix}
C_{\mathfrak{a}_2}^{\otimes d_2}(\theta)&E[23]&\dots & E[2r]\\
&C_{\mathfrak{a}_3}^{\otimes d_3}(\theta)& &E[3r]\\
& & \ddots &\vdots \\
& & &C_{\mathfrak{a}_r}^{\otimes d_r}(\theta)
\end{bmatrix}
\]
and apply the same argument above.

\begin{remark}\label{R:sec4}
	We observe from the above discussion and Example \ref{Ex:sec41} that if $\{m_1,\dots,m_r\}$ is a $\TT_{\Sigma}[t]$-basis for $H(G)$, then the set $\{(t-\theta)^{d_1-1}\cdot m_1,\dots,m_1,\dots,(t-\theta)^{d_r-1}\cdot m_r,\dots,m_r\}$ is a $\TT_{\Sigma}[\sigma]$-basis for $H(G)$. Let us choose $p_{ij}=(t-\theta)^{d_i-j}\cdot m_i$ for $1\leq i \leq r$ and $1\leq j \leq d_i$. Since $H(G)$ is isomorphic to $\Mat_{1\times k}(\TT_{\Sigma})[\sigma]$ as a $\TT_{\Sigma}[\sigma]$-module, by the change of basis, we can identify each $p_{ij}$ with $e_{ij}=(*,\dots,0,1,0,\dots,*)$ where 1 appears in the $(d_1+\dots+d_{i-1}+j)$-th place and the other entries are zero.
\end{remark}
We let $\bold{m}:=[m_1,\dots,m_r]^{\intercal}\in \Mat_{r \times 1}(H(G))$ be the column vector containing $\TT_{\Sigma}[t]$-basis elements of $H(G)$ and consider the matrix $\Phi \in \GL_{r}(\TT_{\Sigma}[t])$ defined by
\begin{equation}\label{E:matrixphi}
\Phi=\begin{bmatrix}
\frac{(t-\theta)^{d_1}}{\mathfrak{a}_1^{(-1)}} &  &  &   \\
\frac{u^{(-1)}_{1}(t-\theta)^{d_1}}{\mathfrak{a}_1^{(-1)}} & \frac{(t-\theta)^{d_2}}{\mathfrak{a}_2^{(-1)}} &  & \\
& \frac{u^{(-1)}_{2}(t-\theta)^{d_2}}{\mathfrak{a}_2^{(-1)}} & \ddots &   \\
\\
& & \ddots & 
\\
&  & \frac{u^{(-1)}_{r-1}(t-\theta)^{d_{r-1}}}{\mathfrak{a}_{r-1}^{(-1)}} &  \frac{(t-\theta)^{d_r}}{\mathfrak{a}_r^{(-1)}}
\end{bmatrix}.
\end{equation}
\begin{proposition}\label{P:sigmaact}
	We have $\sigma \bold{m}=\Phi \bold{m}$.
\end{proposition}
\begin{proof}
	Let $p_{ij}$ and $e_{ij}$ be as in Remark \ref{R:sec4}. Set $c\cdot m_{j}:=0\in H(G)$ for any $c\in \TT_{\Sigma}[t]$ and $j\leq 0$. Claim that for any $1\leq i \leq r $ we have 
	\begin{equation}\label{E:claim}
	\frac{u_{i-1}^{(-1)}(t-\theta)^{d_{i-1}}}{\mathfrak{a}_{i-1}^{(-1)}}\cdot m_{i-1}+\frac{(t-\theta)^{d_i}}{\mathfrak{a}_i^{(-1)}}\cdot m_{i}=\sigma m_{i}.
	\end{equation}
	We do induction on $i$. First we see that 
	\begin{equation}\label{E:repr1}
	\begin{split}
	t\cdot p_{11}&=t\cdot e_{11}\\
	&=(1,0,\dots,0)G(\theta)^{*}\\
	&=(\theta,\dots,\mathfrak{a}_1^{(-1)}\sigma,0,\dots,0)\\
	&=\theta p_{11}+\mathfrak{a}_1^{(-1)}\sigma p_{1d_1}\\
	&=\theta p_{11}+\mathfrak{a}_1^{(-1)}\sigma m_1
	\end{split}
	\end{equation}
	where $\mathfrak{a}_1^{(-1)}\sigma$ appears in the $d_1$-th place. Since $p_{11}=(t-\theta)^{d_1-1}\cdot m_1$, it follows from \eqref{E:repr1} that $\sigma m_1=\frac{(t-\theta)^{d_1}}{\mathfrak{a}_1^{(-1)}}\cdot m_1$.
	Assume that the equality in $\eqref{E:claim}$ holds for $i-1$. By using the induction hypothesis and the $t$-action defined in \eqref{E:actiont} we obtain
	\begin{equation}\label{E:claim222}
	\begin{split}
	&t\cdot p_{i1}\\
	&=t\cdot e_{i1}\\
	&=\theta p_{i1}+\mathfrak{a}_i^{(-1)}
	\Big((-1)^{i-1}u_1^{(-1)}\dots u_{i-1}^{(-1)}\sigma \cdot p_{1d_1}+(-1)^{i-2}u_2^{(-1)}\dots u_{i-1}^{(-1)}\sigma \cdot p_{2d_2} \\
	&\ \ \ +\dots +(-1)u_{i-1}^{(-1)}\sigma \cdot p_{(i-1)d_{i-1}}\Big) +\mathfrak{a}_i^{(-1)}\sigma p_{id_i} \\
	&=\theta p_{i1}+\mathfrak{a}_i^{(-1)}
	\Big((-1)^{i-1}u_1^{(-1)}\dots u_{i-1}^{(-1)}\frac{(t-\theta)^{d_1}}{\mathfrak{a}_1^{(-1)}}\cdot m_1\\
	&\ \ +(-1)^{i-2}u_1^{(-1)}\dots u_{i-1}^{(-1)}\frac{(t-\theta)^{d_1}}{\mathfrak{a}_1^{(-1)}}\cdot m_1+(-1)^{i-3}u_2^{(-1)}\dots u_{i-1}^{(-1)}\frac{(t-\theta)^{d_2}}{\mathfrak{a}_2^{(-1)}}\cdot m_2\\
	&\ \ \ +(-1)^{i-2}u_2^{(-1)}\dots u_{i-1}^{(-1)}\frac{(t-\theta)^{d_2}}{\mathfrak{a}_2^{(-1)}}\cdot m_2\\
	&\ \ \ +\dots +\frac{(-1)^{i-(i-2)}u_{i-2}^{(-1)}u_{i-1}^{(-1)}(t-\theta)^{d_{i-2}}}{\mathfrak{a}_{i-2}^{(-1)}}\cdot m_{i-2} \\
	& \ \ \ +\frac{(-1)u_{i-1}^{(-1)}(t-\theta)^{d_{i-1}}}{\mathfrak{a}_{i-1}^{(-1)}}\cdot m_{i-1}\Big)+\mathfrak{a}_i^{(-1)}\sigma p_{id_i}\\
	&=\theta p_{i1}+\mathfrak{a}_i^{(-1)}\sigma m_i -\frac{u_{i-1}^{(-1)}(t-\theta)^{d_{i-1}}}{\alpha_{i-1}^{(-1)}}\cdot m_{i-1}.
	\end{split}
	\end{equation}
	Since $p_{i1}=(t-\theta)^{d_i-1}\cdot p_{id_i}=(t-\theta)^{d_i-1}\cdot m_i$, the claim follows from the calculation in \eqref{E:claim222}. Thus the definition of the matrix $\Phi$ and $\bold{m}$ together with \eqref{E:claim} imply the proposition.
\end{proof}

\subsection{Rigid Analytic Triviality of $G$}
In this section, we introduce the idea of rigid analytic triviality for Anderson $A[\underline{t}_{\Sigma}]$-module $G$ of dimension $k$ defined as in \eqref{E:moduleG}. We start with explaining necessary background and at the end, we relate them to the uniformizability of Anderson $A[\underline{t}_{\Sigma}]$-modules. 

Let $H(G)$ be the Frobenius module corresponding to $G$ which is free of rank $r$ over $\TT_{\Sigma}[t]$
and  $\bold{m}=[m_1,\dots,m_r]^{\intercal}$ be the column vector consisting of $\TT_{\Sigma}[t]$-basis elements of $H(G)$. Then for any $h=[h_1,\dots,h_r]\in \Mat_{1 \times r}(\TT_{\Sigma}[t])$, we define the map 	
\[
\iota :\Mat_{1\times r}(\TT_{\Sigma}[t])\to \Mat_{1\times k}(\TT_{\Sigma}[\sigma])
\]
by $\iota(h)=h\cdot \bold{m}=h_1\cdot m_1+\dots+h_r\cdot m_r$
where the action $\cdot$ is given by the $\TT_{\Sigma}[t]$-action on $H(G)$. 
\begin{lemma}\label{L:sec4unif}
	We have
	\begin{enumerate}
		\item [(i)] For any $h=[h_1,\dots,h_r] \in \Mat_{1 \times r}(\TT_{\Sigma}[t])$, we have $\iota( h^{(-1)}\Phi)=\sigma \iota(h)$.
		\item[(ii)] For all $h=[h_1,\dots,h_r] \in \Mat_{1 \times r}(\TT_{\Sigma}[t])$, we have $\iota(th)=\iota(h)G(\theta)^{*}.$
	\end{enumerate}
\end{lemma}
\begin{proof} 
	To prove the first part we observe by Proposition \ref{P:sigmaact} that 
	\[
	\sigma\iota(h)=\sigma[h_1,\dots,h_r]\cdot \bold{m}=[h_1^{(-1)},\dots,h_r^{(-1)}]\sigma\cdot \bold{m}=h^{(-1)}\Phi\cdot\bold{m}=\iota(h^{(-1)}\Phi).
	\]
	On the other hand, we have by \eqref{E:actiont} that
	\[
	\iota(th)=t[h_1,\dots,h_r]\cdot \bold{m}=t\cdot \iota(h)=\iota(h)G(\theta)^{*}
	\]
	which proves the second part.
\end{proof}
We call the tuple $(\iota, \Phi)$ a $t$-frame of $G$.

Before we state the definition of rigid analytic triviality, for any $f\in \TT_{\Sigma}$, we define the ring $\TT_{\Sigma}\{t/f\}$ by
\[
\TT_{\Sigma}\{t/f\}:=\bigg\{\sum_{i\geq 0}a_it^i\in \TT_{\Sigma}[[t]] \ \ | \ \ \dnorm{f}^i\dnorm{a_i}\to 0\bigg \}.
\]
Moreover, we define the norm $\lVert g \rVert_f=\lVert \sum a_it^i\rVert_f=\sup \{\dnorm{f}^i\dnorm{a_i}\}$ so that the ring $\TT_{\Sigma}\{t/f\}$ is complete with respect to the norm $\lVert \cdot \rVert_f$. Furthermore, for any $M=(M_{ij})\in \Mat_{n\times l}(\TT_{\Sigma}\{t/f\})$ we set $\lVert M \rVert_f:=\max_{i,j}\lVert M_{ij} \rVert_f$.
\begin{definition}Let $(\iota,\Phi)$ be a $t$-frame of $G$ defined as in \eqref{E:moduleG}  and let $\Psi \in \GL_{r}(\TT_{\Sigma}\{t/\theta\})$
	be a matrix such that 
	\[
	\Psi^{(-1)}=\Phi \Psi.
	\]
	Then we call $(\iota,\Phi,\Psi)$ a rigid analytic trivialization of $G$ and say $G$ is rigid analytically trivial.
\end{definition}

\begin{remark}\label{R:sec40} We recall Example \ref{Ex:sec41}. Note that the $\TT_{\Sigma}[\sigma]$-action on the $\TT_{\Sigma}[t]$-basis $v_1$ of $H(C_{\alpha}^{\otimes n})$ is represented by  $\Phi=\frac{(t-\theta)^n}{\alpha^{(-1)}}$. By \eqref{E:omegaeq} and \eqref{E:omega1} we see that $\Psi=\omega_{\alpha}\Omega(t)^n$ satisfies the equality $\Psi^{(-1)}=\Phi \Psi$. Moreover since $\Omega(t)\in \TT_{\Sigma}\{t/\theta\}$ by \cite[Cor. 6.2.10]{GP} so is $\Psi$. Thus  $C_{\alpha}^{\otimes n}$ is rigid analytically trivial.
\end{remark}
Now we finish this section with a fundamental theorem which will be useful to prove Theorem \ref{T:22}.
\begin{theorem}[{cf.~ \cite[Thm.~4.5.5]{GP}}]\label{T:sec40} Let $\phi$ be an Anderson $A[\underline{t}_{\Sigma}]$-module defined in \eqref{E:sec40}. If $\phi$ has a rigid analytic trivialization $(\iota,\Phi,\Psi)$, then $\phi$ is uniformizable.
\end{theorem}
\begin{proof}See Appendix.
\end{proof}

As an immediate corollary of Remark \ref{R:sec40} and Theorem \ref{T:sec40}, we deduce the following.
\begin{corollary}\label{C:sec41}
	The Anderson $A[\underline{t}_{\Sigma}]$-module $C_{\alpha}^{\otimes n}$ defined in \eqref{E:sec44} is uniformizable.
\end{corollary}

We now discuss the rigid analytic triviality of the Anderson $A[\underline{t}_{\Sigma}]$-module $G$ whose corresponding  Frobenius module $H(G)$ given as in \eqref{E:Frobmodule}. 
We prove the following proposition.
\begin{proposition}\label{P:sec60}
	Let $G$ be the Anderson $A[\underline{t}_{\Sigma}]$-module defined as in \eqref{E:moduleG} corresponding $\mathcal{C}$ and the tuple $u=(u_1,\dots,u_r)\in (\TT_{\Sigma}\setminus\{0\})^r$ such that $\dnorm{u_i}<q^{\frac{s_iq-|U_i|}{q-1}}$ for $1\leq i \leq r$. Then $G$ is rigid analytically trivial.
	In particular, the exponential function $\exp_{G}$ is surjective.
\end{proposition}
\begin{proof}
	Let $\Phi$ be given as in \eqref{E:matrixphi}.
	For any $1\leq l<j\leq r+1$ and $u_{l,j}=(u_l,\dots,u_{j-1})\in (\TT_{\Sigma}\setminus \{0\})^r$ such that $\dnorm{u_i}<q^{\frac{s_iq-|U_i|}{q-1}}$ for $l\leq i \leq j-1$, set
	\begin{equation}
	\begin{split}
	L_{u_{l,j}}(t):&=\sum_{i_l>\dots>i_{j-1}\geq 0}(\omega_{U_l}\Omega^{s_l}(t)u_{l}(t))^{(i_l)} \dots (\omega_{U_{j-1}}\Omega^{s_{j-1}}(t)u_{j-1}(t))^{(i_{j-1})}\\
	&=\Omega^{s_{l}+\dots +s_{j-1}}(t)\prod_{i=l}^{j-1}\omega_{U_i}\times\\ 
	& \sum_{i_l>\dots>i_{j-1}\geq 0}\frac{u^{(i_l)}_{l}(t)\dots u_{j-1}^{(i_{j-1})}(t) b_{i_l}(U_l)\dots b_{i_{j-1}}(U_{j-1}) }{((t-\theta^q)\dots (t-\theta^{q^{i_l}}))^{s_l}\dots ((t-\theta^q)\dots (t-\theta^{q^{i_{j-1}}}))^{s_{j-1}} }.
	\end{split}
	\end{equation}
	Consider the matrix $\Psi\in \Mat_{r}(\TT_{\Sigma,t})$ by
	\[
	\Psi=\begin{bmatrix}
	\Omega(t)^{d_1}\prod_{i=1}^r\omega_{U_i} & & & & \\
	L_{u_{1,2}}(t)\Omega(t)^{d_2}\prod_{i=2}^r\omega_{U_i} & \ddots & & & \\
	& \ddots & \ddots & & \\
	&  & \ddots & & \\
	&  &  & & \\
	L_{u_{1,r}}(t)\Omega(t)^{d_r}\omega_{U_r}  &  & \dots & L_{u_{r-1,r}}(t)\Omega(t)^{d_{r}}\omega_{U_r}& \Omega(t)^{d_r}\omega_{U_r}\\
	\end{bmatrix}.
	\]
	Since $\Omega(t)$ and $\omega_{U_i}$ are invertible in $\TT_{\Sigma,t}$, by Theorem \ref{T:entireness}, we see that $\Psi \in \GL_{r}(\TT_{\Sigma}\{t/\theta\})$. Thus the proposition follows from the same argument of the proof of Proposition \ref{P:ljl}. 
\end{proof}

\subsection{Analysis on the  Coefficients of The Logarithm Function}
We continue with the notation from \S 4.1 and furthermore for any $1\leq i \leq r$, we define $E_{i}\in \Mat_{k}(\TT_{\Sigma})$ so that its $(i,i)$-th block matrix is $E[ii]$ and the rest is zero matrix. 

We denote the logarithm function $\log_{G}$ by
$
\log_{G}=\sum_{i\geq 0} P_i\tau^i
$ 
where
\[
P_0=\Id_{k}, \ \ P_i=\begin{bmatrix}
P_i[11] & \cdots & P_i[1r]
\\
\vdots & &\vdots
\\
P_i[r1] & \cdots & P_i[rr]
\end{bmatrix}\in \Mat_{k}(\TT_{\Sigma}), \quad P_i[jk]\in \Mat_{d_j\times d_k}(\TT_{\Sigma}).
\]
\begin{proposition}[{cf.~\cite[Prop. 3.2.1]{ChangMishibaApr}}]\label{P:log0} We have $P_i[lm]=0$ for $l>m$. For $l\leq m$, we denote the lower most right corner of $P_i[lm]$ by $y_i[lm]$. Then  $y_i[lm]=\frac{\prod_{j=m}^{r}b_i(U_j)}{L_i^{d_m}}$ if  $l=m$ and when $l<m$, we have
	\begin{multline}
	y_i[lm]=(-1)^{m-l}\times\\
	\ \ \ \ \ \ \ \sum\limits_{0\leq i_l\leq \dots \leq i_{m-1}<i}\frac{u_l^{(i_l)}\dots u_{m-1}^{(i_{m-1})}b_{i_l}(U_l)\dots b_{i_{m-1}}(U_{m-1})\prod_{j=m}^{r}b_i(U_j)}{\ell^{s_l}_{i_l}\dots \ell_{i_{m-1}}^{s_{m-1}}\ell_i^{d_m}}.
	\end{multline}
	
\end{proposition}
\begin{proof}
	We follow the ideas of Chang and Mishiba in \cite[Prop. 3.2.1]{ChangMishibaApr}. By \eqref{E:sec47} we have
	\begin{equation}\label{E:log1}
	P_{i+1}=-\sum\limits_{j=0}^{2d_1-2}\frac{\ad(N)^j(P_iE^{(i)})}{(\theta^{q^{i+1}}-\theta)^{j+1}}
	\end{equation}
	similar to the identity (2.1.3) in \cite{AndThak90}. Note that the upper bound of $j$ in \eqref{E:log1} is determined by the fact that $N^{d_1}=0$. Moreover, we have by \eqref{E:log1} and the definition of the matrix $E$ and $N$ that
	\begin{equation}\label{E:log2}
	P_i[lm]=0 \ \ \text{for } l>m.
	\end{equation}
	Observe that for a block matrix $Y\in \Mat_{k}(\TT_{\Sigma})$ of the form \eqref{E:block}, if $y$ is the element in the lower most right corner of the $(l,m)$-th block matrix, then $E_{l}^{\intercal}YE_m$ has all entries zero except the upper most left corner of the $(l,m)$-th block matrix which is $y\prod_{j=l}^r \alpha_{j}\prod_{j=m}^r\alpha_{j}$.
	
	Note that $E_l^{\intercal}N=0$. On the other hand, for any $j$, we can write $\ad(N)^j(P_iE^{(i)})=NM+(-1)^JP_iE^{(i)}N^j$ for some $M\in \Mat_{k}(\TT_{\Sigma})$. Thus, using  \cite[Eq. (3.2.6)]{ChangMishibaApr} we see that
	\begin{equation}\label{E:log3}
	E_{l}^{\intercal}P_{i+1}E_m=\frac{E_{l}^{\intercal}P_iE^{(i)}N^{d_m-1}E_m}{(\theta-\theta^{q^{i+1}})^{d_m}}.
	\end{equation}
	Observe that  $E^{(i)}N^{j-1}E_m=0$ if $j\neq d_m-1$ and $E^{(i)}N^{d_m-1}E_m$ has all zero columns except the $(d_1+\dots+d_{m-1}+1)$-st to $(d_1+\dots+d_{m})$-th columns which are of the form 
	\[
	\prod_{j=m}^r\alpha_j[ E[1m]^{(i)},
	\dots,
	E[mm]^{(i)}
	,
	\dots,0
	]^{\intercal}.
	\]
	Moreover $E_l^{\intercal} P_i$ has all zero rows except the $(d_1+\dots +d_{l-1}+1)$-th row which is of the form $\prod_{j=l}^r\alpha_{j} [*,\dots,y_i[l1],\dots,y_i[l2],\dots,y_i[lr]]$ where $y_i[lw]$ appears in the $(d_1+\dots +d_w)$-th place for $1\leq w \leq r$. 
	
	Now comparing the upper most left corner of the $(l,m)$-th block matrix of both sides of \eqref{E:log3} by using \eqref{E:log2} we observe that if $l=m$ we have 
	\begin{equation}\label{E:coef}
	\ell_{i+1}y_{i+1}[lm]=\prod_{j=l}^r\alpha_{j}^{(i)}y_i[lm]\ell_i^{d_m}
	\end{equation}
	and if $m>l$, then we obtain 
	\begin{equation}\label{E:log4}
	\ell_{i+1}^{d_m}y_{i+1}[lm]=\prod_{j=m}^r\alpha_j^{(i)}\bigg(\ell_i^{d_m}\sum_{n=l}^{m-1}y_i[ln](-1)^{m-n}\prod_{e=n}^{m-1}u_e^{(i)}+\ell_i^{d_m}y_i[lm]\bigg).
	\end{equation}
	We apply induction on $i$. Note that if $m=l$ then the first part of the proposition can be easily shown by using \eqref{E:coef} as $y_0[lm]=1$ in this case.
	
	When $m>l$ we assume the proposition holds for $y_i[ln]$ where $l\leq n <m$ and $i\geq 0$. Then using \eqref{E:log4} we have that
	\begin{equation}\label{E:log5} 
	\begin{split}
	&\ell_{i+1}^{d_m}y_{i+1}[lm]\\
	&=\ell_i^{d_m}\sum\limits_{n=l+1}^{m-1}(-1)^{n-l}\times\\
	&\ \ \ \  \ \ \ \ \ \sum\limits_{0\leq i_l\leq \dots \leq i_{n-1}<i}\frac{\prod_{j=l}^{n-1}u_{j}^{(i_j)}b_{i_j}(U_{j})\prod_{k=n}^{r}b_{i}(U_{k})\alpha_{k}^{(i)}}{\ell_{i_l}^{s_l}\dots \ell_{i_{n-1}}^{s_{n-1}}\ell_i^{d_n}}(-1)^{m-n}\prod_{e=n}^{m-1}u_e^{(i)}\\
	&\ \ \ +\ell_i^{d_m}\frac{1}{\ell_i^{d_l}}(-1)^{m-l}u_{l}^{(i)}\dots u_{m-1}^{(i)}\prod_{k=l}^{m-1}b_{i}(U_{l})b_i(U_m)\dots b_i(U_r)\prod_{j=m}^{r}\alpha_j^{(i)}\\
	&\ \ \ \  \ \ \ \ \ +\ell_i^{d_m}y_i[lm]\prod_{j=m}^r\alpha_j^{(i)}\\
	&=(-1)^{m-l}\times\\
	&\  \ \ \ \ \  \sum_{n=l+1}^{m-1}\sum\limits_{0\leq i_l\leq \dots \leq i_{n-1}<i}\frac{\prod_{j=l}^{n-1}u_{j}^{(i_j)}b_{i_j}(U_{j})u_n^{(i)}\dots u_{m-1}^{(i)}b_{i+1}(U_n)\dots b_{i+1}(U_r)}{\ell_{i_l}^{s_l}\dots \ell_{i_{n-1}}^{s_{n-1}}\ell_i^{s_n}\dots \ell_i^{s_{m-1}}}\\
	&\ \ \ +(-1)^{m-l}\frac{\prod_{j=l}^{m-1}u_{j}^{(i_j)}b_{i_j}(U_{j})b_{i+1}(U_m)\dots b_{i+1}(U_r)}{\ell_i^{s_l}\dots \ell_i^{s_{m-1}}}+\ell_i^{d_m}y_i[lm]\prod_{j=m}^r\alpha_j^{(i)}\\
	&=(-1)^{m-l}\sum_{\substack{0\leq i_l \leq \dots \leq i_{m-1}\\ i_{m-1}=i}}\frac{\prod_{j=l}^{m-1}u_{j}^{(i_j)}b_{i_j}(U_{j})b_{i+1}(U_m)\dots b_{i+1}(U_r)}{\ell_{i_l}^{s_l}\dots \ell_{i_{m-1}}^{s_{m-1}}}\\
	&\ \ \ \ \  \ \ \  \ +\ell_i^{d_m}y_i[lm]\prod_{j=m}^r\alpha_j^{(i)}.
	\end{split}
	\end{equation}
	Since $y_0[lm]=0$ we obtain
	\begin{align*}
	&\prod_{j=m}^r\alpha_j^{(i)}\ell_i^{d_m}y_i[lm]\\
	&=(-1)^{m-l}\sum_{h=0}^{i-1}\sum_{\substack{0\leq i_l \leq \dots \leq i_{m-1}\\ i_{m-1}=h}}\frac{\prod_{j=l}^{m-1}u_{j}^{(i_j)}b_{i_j}(U_{j})b_i(U_m)\dots b_i(U_r)\prod_{j=m}^r\alpha_j^{(i)}}{\ell_{i_l}^{s_l}\dots \ell_{i_{m-1}}^{s_{m-1}}}
	\end{align*}
	This also implies that 
	\begin{multline}\label{E:proofi}
	\prod_{j=m}^r\alpha_j^{(i)}\ell_i^{d_m}y_i[lm]=(-1)^{m-l}\times \\\sum\limits_{0\leq i_l\leq \dots \leq i_{m-1}<i}\frac{u_l^{(i_l)}\dots u_{m-1}^{(i_{m-1})}b_{i_l}(U_l)\dots b_{i_{m-1}}(U_{m-1})b_{i+1}(U_m)\dots b_{i+1}(U_r)}{\ell_{i_l}^{s_l}\dots \ell_{i_{m-1}}^{s_{m-1}}}.
	\end{multline}
	Thus the proposition follows from combining \eqref{E:log5} and \eqref{E:proofi}.
\end{proof}

\begin{lemma}[{cf.~\cite[Lem.~4.2.1]{ChangMishibaOct}}]\label{L:log0}
	Let $\mathcal{C}$ be a composition array as in \eqref{E:star00} and $u=(u_1,\dots,u_r)\in (\TT_{\Sigma} \setminus \{0\})^r$. Let also $n_l=|U_l|$ be the nonnegative integer for $1\leq l \leq r$. If $\dnorm{u_l}\leq q^{\frac{s_lq-n_l}{q-1}}$ for each $1\leq l <r$, then
	\begin{equation}\label{E:inequal}
	\dnorm{P_iN^{d_l-j}E_l}\leq q^{(d_l-j)q^i-(d_lq^i-d_1)\frac{q}{q-1}+(n_l+\dots +n_r)\frac{q^i}{q-1}}
	\end{equation}
	for each $i,j,k$ where $i\geq 0$, $1\leq l \leq r$, and $1\leq j \leq d_l$.
\end{lemma}
\begin{proof}
	We note that the matrix $P_iN^{d_l-j}E_l$ has all zero columns except the $(d_1+\dots +d_{l-1}+1)$-th column which is $\prod_{j=l}^{r}\alpha_j$ multiple of the $(d_1+\dots +d_{l-1}+j)$-th column of $P_i$. If $i=0$ then the lemma holds. Assume by induction that the inequality \eqref{E:inequal} holds for $i$ and we show that it also holds for $i+1$. By \eqref{E:log1} we have that 
	\begin{multline}\label{E:log7}
	P_{i+1}N^{d_l-j}E_l\\
	=-\sum_{m=0}^{2d_1-2}\frac{1}{(\theta^{q^{i+1}}-\theta)^{m+1}}\sum_{n=0}^m(-1)^n\binom{m}{n}N^{m-n}P_iE^{(i)}N^{n+d_l-j}E_l.
	\end{multline}
	We observe that $E^{(i)}N^{n+d_l-j}E_l=0$ for $n\neq j-1$. Moreover by the definition of $N$ we have that $N^{m-n}=0$ for $m-n\geq d_1$. Therefore using the definition of the matrices $E_l$ and $E$ we have
	\begin{align*}
	P_{i+1}N^{d_l-j}E_l&=\sum_{m=j-1}^{d_1+j-2}\frac{(-1)^j}{(\theta^{q^{i+1}}-\theta)^{m+1}}\binom{m}{j-1}N^{m-j+1}P_iE^{(i)}N^{d_l-1}E_l\\
	&=\sum_{m=j-1}^{d_1+j-2}\frac{(-1)^j}{(\theta^{q^{i+1}}-\theta)^{m+1}}\binom{m}{j-1}N^{m-j+1}\times\\ 
	&\ \ \ \ \ \ \ \ \ \ \ \  \ \ \ \  \sum_{n=1}^{l}(-1)^{l-n}P^{\prime}_{i,l,n}\prod_{n\leq e\leq l-1}u_e^{(i)}\prod_{h=l}^{r}\alpha_h\alpha_h^{(i)}
	\end{align*}
	where we define the matrix $P^{\prime}_{i,l,n}$ as the matrix whose $(d_1+\dots+d_{l-1}+1)$-th column is the $(d_1+\dots+d_{n-1}+d_n)$-th column of $P_i$ and all the other columns are zero. Thus taking $l=n$ and $j=d_n$ in the inequality \eqref{E:inequal} and using induction hypothesis we see that
	\begin{equation}\label{E:log8}
	\begin{split}
	&\dnorm{P^{\prime}_{i,l,n}\prod_{n\leq e\leq l-1}u_e^{(i)}\prod_{h=l}^{r}\alpha_h\alpha_h^{(i)}}\\
	&\leq q^{-(d_nq^i-d_1)\frac{q}{q-1}+\frac{q^i(n_n+\dots+n_r)}{q-1}}\prod_{n\leq e\leq l-1}q^{\frac{q^i(s_eq-n_e)}{q-1}}\prod_{h=l}^{r}\dnorm{\alpha_h^{(i)}}\\
	&\leq q^{-(d_nq^i-d_1)\frac{q}{q-1}+\frac{q^i(n_n+\dots+n_r)}{q-1}}q^{(d_n-d_l)\frac{q^{i+1}}{q-1}}q^{-\frac{q^i(n_n+\dots +n_{l-1})}{q-1}}q^{(n_l+\dots+n_r)q^i}\\
	&\leq q^{-(d_lq^i-d_1)\frac{q}{q-1}+\frac{q^i(n_l+\dots+n_r)}{q-1}}q^{(n_l+\dots+n_r)q^i}\\
	&\leq q^{-(d_lq^i-d_1)\frac{q}{q-1}+\frac{q^{i+1}(n_l+\dots+n_r)}{q-1}}.
	\end{split}
	\end{equation}
	Thus using \eqref{E:log7} and \eqref{E:log8} we see that 
	\begin{align*}
	\dnorm{P_{i+1}N^{d_l-j}E_l}&\leq \max_{j-1\leq k\leq d_1+j-2}q^{-(k+1)q^{i+1}-(d_lq^i-d_1)\frac{q}{q-1}}q^{\frac{(n_l+\dots+n_r)q^{i+1}}{q-1}}\\
	&=q^{-jq^{i+1}-(d_lq^i-d_1)\frac{q}{q-1}}q^{\frac{(n_l+\dots+n_r)q^{i+1}}{q-1}}\\
	&=q^{(d_l-j)q^{i+1}-(d_lq^{i+1}-d_1)\frac{q}{q-1}+(n_l+\dots +n_r)\frac{q^{i+1}}{q-1}}
	\end{align*}
	which concludes the proof.
\end{proof}
We continue with the notation in the statement of Lemma \ref{L:log0}.
\begin{proposition}[{cf.~  \cite[Prop.~4.2.2]{ChangMishibaOct}}]\label{P:log1}
	Let $\dnorm{u_l}\leq q^{\frac{s_lq-n_l}{q-1}}$ for $1\leq l <r$ and $x=(x_i)\in \TT_{\Sigma}^{k}$ be a point such that $\dnorm{x_{d_1+\dots + d_{l-1}+j}}<q^{-(d_l-j)+\frac{d_lq}{q-1}-\frac{(n_l+\dots +n_r)}{q-1}}$. Then $\log_G$ converges at $x$ in $\TT_{\Sigma}^k$.
\end{proposition}
\begin{proof}The proof follows from the standard estimation in non-archimedean analysis. In particular since $\dnorm{E_l}=q^{(n_l+\dots+n_r)}$, by Lemma \ref{L:log0} we have
	\begin{equation}\label{E:log6}
	\begin{split}
	&\dnorm{P_ix^{(i)}}\\
	&\leq \max_{j,l}\{q^{-(n_l+\dots+n_r)+(d_l-j)q^i-(d_lq^i-d_1)\frac{q}{q-1}+(n_l+\dots+n_r)\frac{q^i}{q-1}}\dnorm{x_{d_1+\dots + d_{l-1}+j}}^{q^i} \}\\
	&=\max_{j,l}\Big\{q^{-(n_l+\dots+n_r)+\frac{d_1q}{q-1}}\bigg(\dnorm{x_{d_1+\dots + d_{l-1}+j}}/q^{-(d_l-j)+\frac{d_lq}{q-1}-\frac{(n_l+\dots+n_r)}{q-1}}\bigg)^{q^i} \Big\}.
	\end{split}
	\end{equation}
	But by the assumption on the element $x$ when $i$ goes to infinity, the last term in \eqref{E:log6} approaches to 0 and thus this proves the statement in the proposition.
\end{proof}
The special point $v_{\mathcal{C},u}\in \TT_{\Sigma}^k$ corresponding to a composition array $\mathcal{C}$ and $u=(u_1,\dots,u_r)\in (\TT_{\Sigma}\setminus\{0\})^r$ is defined by 
\begin{equation}\label{E:spoint}
v=v_{\mathcal{C},u}:=[
0,
\dots,
0,
(-1)^{r-1}u_1\dots u_r,
0,
\dots,
0,
(-1)^{r-2}u_2\dots u_r,
0,
\dots,
0,
u_r
]^{\intercal}
\end{equation}
where the entry $(-1)^{r-j}u_j\dots u_r$ for $1\leq j\leq r$ appears in $(d_1+\dots +d_j)$-th place.

We continue with some notation. For any composition array 
\[
\mathcal{C}=\binom{U_j,\dots,U_i}{s_j,\dots,s_i}
\]
where $1\leq j \leq i$, we define the composition array $\tilde{\mathcal{C}}$ by
\[
\tilde{\mathcal{C}}:=\binom{U_i,U_{i-1},\dots,U_j}{ s_i,s_{i-1},\dots,s_j}.
\]
Furthermore, for any $u=(u_j,\dots,u_i) \in (\TT_{\Sigma}\setminus \{0\})^{j-i+1}$, we define $\tilde{u}:=(u_i,u_{i-1},\dots,u_j)\in (\TT_{\Sigma}\setminus \{0\})^{j-i+1}$.
\begin{theorem}[{cf.~\cite[Thm.~3.3.3]{ChangMishibaApr}}]\label{T:log0} Let $\mathcal{C}$ be a composition array of depth $r$ defined as in \eqref{E:star00} and $u=(u_1,\dots,u_r)\in (\TT_{\Sigma}\setminus \{0\})^r$ be such that $\tilde{u} \in  D_{\tilde{\mathcal{C}}}^{\prime\prime}$. Let $G$ and $v$ be defined as in \eqref{E:moduleG} and \eqref{E:spoint} respectively corresponding to $\mathcal{C}$ and $u$. Then $\log_{G}$ converges at $v$. Moreover for any $1\leq l \leq r$, the element in $(d_1+\dots +d_l)$-th place of $\log_G(v)$ is equal to $(-1)^{r-l}\Li^{*}_{\tilde{\mathcal{C}}_l}(\tilde{u}_l)$ where 
	$\mathcal{C}_l=\binom{U_l,\dots,U_r}{s_l,\dots,s_r}$ and $u_l=(u_l,\dots,u_r)$.
\end{theorem}
\begin{proof} We follow the technique in \cite[Thm. 3.3.3]{ChangMishibaApr}.
	By the assumption on the element $u$, the norm of the $(d_1+\dots +d_{l-1}+d_l)$-th coordinate of $v$ is
	\[
	\dnorm{(-1)^{r-l}u_l\dots u_r}<q^{\frac{s_lq-n_l}{q-1}}\dots q^{\frac{s_rq-n_r}{q-1}}=q^{\frac{d_lq}{q-1}-\frac{(n_l+\dots +n_r)}{q-1}}.
	\]	
	Then by Proposition \ref{P:log1} we see that $\log_{G}$ converges at $v$.
	Using Proposition \ref{P:log0} and the definition of $v$ we see that the $(d_1+\dots +d_l)$-th coordinate of $\log_G(v)$ is
	\begin{align*}
	&\sum_{i\geq 0}\sum_{m=l}^r y_i[lm](-1)^{r-m}u_m^{(i)}\dots u_r^{(i)}\\
	&=\sum_{i\geq 0}\frac{(-1)^{r-l}\prod_{j=l}^{r}b_i(U_j)u_l^{(i)}\dots u_r^{(i)}}{\ell_i^{d_l}}\\
	&\ \ \ +\sum_{i\geq 0}\sum_{m=l+1}^{r}(-1)^{m-l}\times \\ &\ \ \  \ \ \ \ \sum\limits_{0\leq i_l\leq \dots \leq i_{m-1}<i}\frac{\prod_{j=l}^{m-1}u_{j}^{(i_j)}b_{i_j}(U_{j})}{\ell_{i_l}^{s_l}\dots \ell_{i_{m-1}}^{s_{m-1}}\ell_i^{d_m}}(-1)^{r-m}u_m^{(i)}\dots u_r^{(i)}\prod_{j=m}^rb_i(U_j)\\
	&=(-1)^{r-l}\bigg(\sum_{i\geq 0}\frac{\prod_{j=l}^ru_j^{(i)}b_i(U_j)}{\ell_i^{s_l}\dots \ell_i^{s_r}}\\
	&\ \ \ \  \ \ \ \ \ \ +\sum_{i\geq 0}\sum_{m=l+1}^{r}\sum\limits_{0\leq i_l\leq \dots \leq i_{m-1}<i}\frac{\prod_{j=l}^{m-1}u_j^{(i_j)}b_i(U_j)u_m^{(i)}\dots u_r^{(i)}\prod_{j=m}^rb_i(U_j)}{\ell_{i_l}^{s_l}\dots \ell_{i_{m-1}}^{s_{m-1}}\ell_i^{s_m}\dots \ell_i^{s_r}}\bigg)\\
	&=(-1)^{r-l}\sum\limits_{0\leq i_l\leq \dots \leq i_{r}}\frac{u_l^{(i_l)}\dots u_r^{(i_r)}b_{i_l}(U_l)\dots b_{i_r}(U_r)}{\ell_{i_l}^{s_l}\dots \ell_{i_r}^{s_r}}\\
	&=(-1)^{r-l}\Li_{\tilde{\mathcal{C}}_l}^{*}(\tilde{u_l}).
	\end{align*}
\end{proof}

\subsection{The Construction of the Anderon $A[\underline{t}_{\Sigma}]$-module $G_{\mathcal{C}}$}
Let $\mathcal{C}$ be a composition array as in \eqref{E:star00} such that $\wght(\mathcal{C})=w$ and $\dep(\mathcal{C})=r$, and let $u=(u_1,\dots, u_r)\in (\TT_{\Sigma}\setminus \{0\})^{r}$. We recall the set of tuples $\{ (a_l,\mathcal{C}_l,u_l)| \ \ 1\leq l \leq n\}$ for some $n\in \ZZ_{\geq 1}$ from Theorem \ref{T:star0} and without loss of generality, for $1\leq l \leq s$, let $\mathcal{C}_l$ be a composition array such that $\dep(\mathcal{C}_l)=1$ and for $s+1\leq l \leq n $, let $\mathcal{C}_l$ be a composition array whose depth is bigger than 1. Assume that $\dep(\mathcal{C}_l)=m_l$. Let $G_l$ be the Anderson $A[\underline{t}_{\Sigma}]$-module corresponding to the tuple $(\tilde{\mathcal{C}_l},\tilde{u_l})$ defined as in \eqref{E:moduleG}.
We also recall the notation from \S 4.1 and set $k_l^{\prime}=\sum_{j=2}^{m_l}d_{lj}$ for any $1\leq l \leq n$. It is easy to see from the definition that $w=d_{l1}$ for $1\leq l \leq n$. Now define
\[
G_l^{\prime}:=\begin{bmatrix}
C^{\otimes d_{l2}}&E_{l}[23]&\dots &\dots&E_l[2m_l]\\
&C^{\otimes d_{l3}}& E_{l}[34]&\dots&E_l[3m_l] \\
& & \ddots & & \vdots\\
& & &\ddots &\vdots \\
& & & & C^{\otimes d_{lm_l}}
\end{bmatrix}\in \Mat_{k_l^{\prime}}(\TT_{\Sigma})[\tau]
\]
and $G_l^{\prime\prime}:=\begin{bmatrix}
E_{l}[12]&E_{l}[13]&\dots& E_l[1m_l]\end{bmatrix}\in \Mat_{w\times k_l^{\prime}}(\TT_{\Sigma})$.
Observe that 
\[
G_{l}(\theta)=\begin{bmatrix}
C^{\otimes w}_{\mathfrak{a}_{1}}(\theta) & G_{l}^{\prime \prime}\\
&G_{l}^{\prime}
\end{bmatrix}.
\]
Let us set $k_{\mathcal{C}}:=w+\sum_{l=1}^nk_{l}^{\prime}$.  Then we define the Anderson $A[\underline{t}_{\Sigma}]$-module $G_{\mathcal{C}}:A[\underline{t}_{\Sigma}]\to \Mat_{k_{\mathcal{C}}}(\TT_{\Sigma})[\tau]$ by
\begin{equation}\label{E:sec6m1}
G_{\mathcal{C}}(\theta)=\begin{bmatrix}
C^{\otimes d_1}_{\mathfrak{a}_1}&G_{s+1}^{\prime\prime}&G_{s+2}^{\prime\prime}&\dots &G_{n}^{\prime\prime}\\
&G_{s+1}^{\prime}& & & \\
& & G_{s+2}^{\prime} & & \\
& & & \ddots & \\
& & & &G_n^{\prime}
\end{bmatrix}\in \Mat_{k_{\mathcal{C}}}(\TT_{\Sigma})[\tau].
\end{equation}
Using the definition of matrices $G_l^{\prime}$ and $G_l^{\prime\prime}$ we see that  $G_{\mathcal{C}}$ can be rewritten as in \eqref{E:sec40} and therefore it has an exponential function $\exp_{G_{\mathcal{C}}}:\Mat_{k_{\mathcal{C}}\times 1}(\TT_{\Sigma})\to \Mat_{k_{\mathcal{C}}\times 1}(\TT_{\Sigma})$ which is everywhere convergent by Proposition \ref{P:sec40}.

For the rest of this section we aim to prove that $\exp_{G_{\mathcal{C}}}$ is a surjective function. Now let $k_l:=\sum_{j=1}^{m_l}d_{lj}=w+k_l^{\prime}$ for $1\leq l \leq n$. First we give the definition of the following map $\lambda:\TT_{\Sigma}^{\sum_{l=1}^nk_l}\to \TT_{\Sigma}^{k_{\mathcal{C}}}$ by
\begin{equation}\label{E:lambdamap}
\lambda:\begin{pmatrix}
z_{11} \\
\vdots \\
z_{1k_1}\\
\vdots\\
z_{n1}\\
\vdots\\
z_{nk_n}\\
\end{pmatrix}\to \Lambda\begin{pmatrix}
z_{11} \\
\vdots \\
z_{1k_1}\\
\vdots\\
z_{n1}\\
\vdots\\
z_{nk_n}\\
\end{pmatrix}=\begin{bmatrix}
z_{11}+\dots+z_{s1}+z_{(s+1)1}+\dots+z_{n1}\\
\vdots \\
z_{1w}+\dots+z_{sw}+z_{(s+1)w}+\dots+z_{nw}\\
z_{(s+1)(w+1)}\\
\vdots\\
z_{(s+1)k_{s+1}}\\
\vdots \\
z_{n(w+1)}\\
\vdots\\
z_{nk_{n}}
\end{bmatrix},
\end{equation}
where the matrix $\Lambda\in \Mat_{k_{\mathcal{C}}\times \sum_{l=1}^{n}k_{l}}(\TT_{\Sigma})$ defined by the block matrix
\begin{equation}\label{E:sec6m2}
\begin{bmatrix}
I_{w\times sw}& \Id_{w} & O_{w\times k_{s+1}^{\prime}} &\Id_{w} & O_{w\times k_{s+2}^{\prime}}&\dots & \Id_{w}& O_{w\times k_n^{\prime}}\\
&O_{k_{s+1}^{\prime}\times w}&\Id_{k_{s+1}^{\prime}}& & & &\\
& & &O_{k_{s+2}^{\prime}\times w}&\Id_{k_{s+2}^{\prime}}& & & \\
& & & &\ddots&\ddots &\\
& & & & &\ddots &\ddots \\
& & & & & &O_{k_n^{\prime}\times w}&\Id_{k_n^{\prime}}
\end{bmatrix}
\end{equation}
so that $I_{w\times sw}$ is the block matrix $[\Id_{w},\dots,\Id_{w}]\in \Mat_{w\times sw}(\TT_{\Sigma})$ and $O_{i\times j}\in \Mat_{i\times  j}(\TT_{\Sigma})$  is the $i\times j$ zero matrix.
Before we prove our next lemma, it should be noted that we define the Anderson $A[\underline{t}_{\Sigma}]$-module $\oplus_{l=1}^n G_l$ of dimension $\sum_{l=1}^nk_l$ by
\begin{align}\label{E:sec6m3}
\oplus_{l=1}^nG_l(\theta)&:=\begin{bmatrix}G_{1}(\theta)& & & \\
&G_{2}(\theta) & & \\
& & \ddots & \\
& & & G_{n}(\theta)
\end{bmatrix}\\
&=\begin{bmatrix} G(\theta) & & & & &  \\
& C^{\otimes w}_{\mathfrak{a}_1}(\theta) & G_{s+1}^{\prime\prime} & & &  \\
& & G_{s+1}^{\prime} & &  \\
& & & \ddots & & \\
& & & & C^{\otimes w}_{\mathfrak{a}_1}(\theta)& G_{n}^{\prime\prime} \\
& & & & & G_{n}^{\prime}
\end{bmatrix}
\end{align}
where $G(\theta)\in \Mat_{sw}(\TT_{\Sigma})[\tau]$ defined as 
\[
G(\theta)=\begin{bmatrix}
C^{\otimes w}_{\mathfrak{a}_1}(\theta) & & \\
& \ddots & \\
& & C_{\mathfrak{a}_1}^{\otimes w}(\theta)
\end{bmatrix}.
\]
Moreover its exponential function $\exp_{\oplus_{l=1}^nG_l}:\Mat_{(sw+\sum_{l=s+1}^nk_l)\times 1}(\TT_{\Sigma})\to \Mat_{(sw+\sum_{l=s+1}^nk_l)\times 1}(\TT_{\Sigma})$ is given by
\[
\exp_{\oplus_{l=1}^nG_l}:\begin{pmatrix}
f_1\\
\vdots\\
f_n
\end{pmatrix}\to \begin{pmatrix}
\exp_{G_1}(f_1)\\
\vdots\\
\exp_{G_n}(f_n)
\end{pmatrix}
\]
where $f_j\in \Mat_{k_{l} \times 1}(\TT_{\Sigma})$ for $1\leq j \leq n$. Using the matrices given in \eqref{E:sec6m1}, \eqref{E:sec6m2} and \eqref{E:sec6m3}, we immediately prove the following lemma.
\begin{lemma}\label{L:sec6morp}
	We have	 
	\[
	G_{\mathcal{C}}(\theta) \Lambda=\Lambda G_{\oplus_{l=1}^nG_l}(\theta).
	\]
	In other words, $\Lambda$ is an Anderson $A[\underline{
		t}_{\Sigma}]$-module homomorphism.
\end{lemma}

Our next lemma introduces the relation between the matrix $\Lambda$, the infinite series $\exp_{G_{\mathcal{C}}}$ and $\exp_{\oplus_{l=1}^nG_l}$.
\begin{lemma}\label{L:sec6exp}We have the following equality over $\Mat_{\sum_{l=1}^nk_l\times k_{\mathcal{C}}}(\TT_{\Sigma})[[\tau]]$:
	\[
	\exp_{G_{\mathcal{C}}}\Lambda=\Lambda \exp_{\oplus_{l=1}^nG_l}.
	\]
	In particular, for any $f\in \Mat_{\sum_{l=1}^nk_l\times 1}(\TT_{\Sigma})$, we have $\lambda(\exp_{\oplus_{l=1}^nG_l}(f))=\exp_{G_{\mathcal{C}}}(\lambda(f))$. 
\end{lemma}
\begin{proof}
	Let us set $G_{\mathcal{C}}(\theta)=\theta \Id_{k_{\mathcal{C}}} +N_1+E_1\tau$ for the nilpotent matrix $N_1$ such that $N_1^{w}=0$ and $E_1\in \Mat_{k_{\mathcal{C}}}(\TT_{\Sigma})$. Similarly, let $\oplus_{l=1}^{n}G_l(\theta)=\theta \Id_{\sum_{l=1}^n k_l} +N_2+E_2\tau$ such that $E_2\in \Mat_{\sum_{l=1}^n k_l}(\TT_{\Sigma})$. By the definition of $\oplus_{l=1}^nG_l$ we know that $N_2^{w}=0$. Since $\Lambda$ is invariant under the automorphism $\tau$, by Lemma \ref{L:sec6morp} we have that 
	\begin{align*}
	(\theta \Id_{k_{\mathcal{C}}} +N_1+E_1\tau)\Lambda=\theta \Id_{k_{\mathcal{C}}}\Lambda + N_1\Lambda +E_1\Lambda\tau=\Lambda \theta\Id_{\sum k_l} + \Lambda N_2 + \Lambda E_2\tau.
	\end{align*}
	Since $\Lambda\theta \Id_{\sum k_l}=\theta  \Id_{k_{\mathcal{C}}}\Lambda$, comparing coefficients of $\tau^{0}$ and $\tau$ above, we see that 
	$
	N_1\Lambda=\Lambda N_2
	$ and $
	E_1\Lambda=\Lambda E_2
	$.
	Now let $\exp_{G_{\mathcal{C}}}=\sum_{i\geq 0}\beta_{1,i}\tau^i$ and $\exp_{\oplus_{l=1}^nG_l}=\sum_{i\geq 0}\beta_{2,i}\tau^i$. We claim that $\beta_{1,i}\Lambda=\Lambda\beta_{2,i}$ for all $i\geq 0$. We do induction on $i$. For $i=0$, the claim holds. Assume that it is true for $i$. By \eqref{E:sec43} we have that 
	\begin{equation}\label{E:sec6beta}
	\beta_{m,i+1}=\sum_{j=0}^{2w-2}\frac{\ad(N_m)^j(E_m\beta_{m,i}^{(1)})}{[i+1]^{j+1}},\quad m=1,2.
	\end{equation}
	Moreover by the induction argument, Lemma \ref{L:Sec40} and commuting of $N_i$ and $E_i$ with $\Lambda$ for $i=1,2$ and for any $0\leq j \leq 2w-2$ we have
	\begin{align*}
	\ad(N_1)^j(E_1\beta_{1,i}^{(1)})\Lambda=\Lambda \ad(N_2)^j(E_2\beta_{2,i}^{(1)}).
	\end{align*}
	Thus the claim follows from \eqref{E:sec6beta}.
\end{proof}
\begin{proposition}\label{P:sec6unif}
	The exponential function $\exp_{G_{\mathcal{C}}}:\Mat_{k_{\mathcal{C}\times 1}}(\TT_{\Sigma})\to \Mat_{k_{\mathcal{C}\times 1}}(\TT_{\Sigma})$ is surjective. In other words, the Anderson $A[\underline{t}_{\Sigma}]$-module $G_{\mathcal{C}}$ is uniformizable.
\end{proposition}
\begin{proof}
	Let $G_{1},\dots,G_{s},G_{s+1},\dots, G_n$ be the Anderson $A[\underline{t}_{\Sigma}]$-modules that are used to construct $G_{\mathcal{C}}$ such that for $1\leq j \leq s$, $G_{j}=C^{\otimes w}$ and $G_j\neq C^{\otimes w}$ when $j\geq s+1$. Let $y=[y_{1},\dots,y_{w},y_{s+1,w+1},\dots,y_{s+1,k_{s+1}},\dots,y_{n,w+1},\dots,y_{n,k_n}]^{\intercal}$ be  in $\Mat_{k_{\mathcal{C}} \times 1}(\TT_{\Sigma})$. Let $E_{ij}\in \Mat_{k_i\times 1}(\mathbb{F}_q)$ be the column matrix such that $j$-th entry is 1 and the other entries are zero. We now define elements $Y_j\in \Mat_{k_j\times 1}(\TT_{\Sigma})$ for different cases. If $n\geq w \geq s$, then we set 
	\[
	Y_j := \left\{\begin{array}{lr}
	y_jE_{jj}, & \text{if } 1\leq j\leq s\\
	y_jE_{jj}+\sum_{l=w+1}^{k_j}y_{jl}E_{jl}, & \text{if } s< j\leq w\\
	\sum_{l=w+1}^{k_j}y_{jl}E_{jl}, & \text{if } w<j\leq  n
	\end{array}\right\}.
	\]
	If $n\geq s\geq w$, we set
	\[
	Y_j := \left\{\begin{array}{lr}
	y_jE_{jj}, & \text{if } 1\leq j\leq w\\
	O_{jj}, & \text{if } w< j\leq s\\
	\sum_{l=w+1}^{k_j}y_{jl}E_{jl}, & \text{if } s<j\leq n
	\end{array}\right\}.
	\]
	where $O_{jj}$ is the $k_j\times 1$-zero matrix. Finally, if $w\geq n \geq s$, then we define
	\[
	Y_j:= \left\{\begin{array}{lr}
	y_jE_{jj}, & \text{if } 1\leq j\leq s\\
	y_jE_{jj}+\sum_{l=w+1}^{k_j}y_{jl}E_{jl}, & \text{if } s< j\leq n-1\\
	\sum_{l=n}^{w} y_lE_{jl}+\sum_{l=w+1}^{k_j}y_{jl}E_{jl}, & \text{if } j=n
	\end{array}\right\}.
	\]	
	By Corollary \ref{C:sec41} and Proposition \ref{P:sec60}, $\exp_{G_{j}}$ is surjective for all $j$. So there exist elements $X_{j}\in \Mat_{k_j\times 1}(\TT_{\Sigma})$ such that $\exp_{G_{j}}(X_{j})=Y_{j}$.	
	Now let 
	$
	x:=[X_1,\dots,X_n]^{\intercal}\in \Mat_{\sum k_l\times 1}(\TT_{\Sigma})$ and 
	$
	Y:=[Y_1,\dots,Y_n]^{\intercal}\in \Mat_{\sum k_l\times 1}(\TT_{\Sigma}).
	$
	Thus by the definition of the map $\lambda$ and Lemma \ref{L:sec6exp} we see that
	\[
	\lambda(\exp_{\oplus_{l=1}^nG_l}(x))=\lambda(Y)=y=\exp_{G_{\mathcal{C}}}(\lambda(x))
	\]
	which gives the surjectivity of $\exp_{G_{\mathcal{C}}}$.
\end{proof}

\subsection{Proof of Theorem \ref{T:22}}
In this subsection we give the proof of our following result and introduce an example.

\begin{theorem}\label{T:result2} Let $\mathcal{C}$ be a composition array as in \eqref{E:star00} of weight $w$. Let also $\mathfrak{I}_1$ be the set of indices $i$ such that $U_i\neq \emptyset$ and $\mathfrak{I}_2$ be the set of $i$'s such that $U_i=\emptyset$. Then there exist a uniformizable Anderson $A[\underline{t}_{\Sigma}]$-module $G_{\mathcal{C}}$ of dimension $k_{\mathcal{C}}$ defined over $\TT_{\Sigma}$, a special point $v_{\mathcal{C}}\in \Mat_{k_{\mathcal{C}} \times 1}(K^{\text{perf}}(\underline{t}_{\Sigma}))$ and an element $Z_{\mathcal{C}}\in \Mat_{k_{\mathcal{C}} \times 1}(\TT_{\Sigma})$ such that 
	\begin{enumerate}
		\item[(i)]$\prod_{i\in \mathfrak{I}_1} \ell_{r_{s_i}-1}^{q^{r_{s_i}}-s_i}b_{r_{s_i}}(U_i)\prod_{i\in \mathfrak{I}_2}\Gamma_{s_i} \zeta_{C}(\mathcal{C})$ occurs as the $w$-th coordinate of $Z_{\mathcal{C}}$ where $r_{s_i}\geq 1$ is an integer such that $s_i\leq q^{r_{s_i}}$ for $i\in \mathfrak{I}_1$.
		\item[(ii)] $\exp_{G_{\mathcal{C}}}(Z_{\mathcal{C}})=v_{\mathcal{C}}.$ 
	\end{enumerate}
\end{theorem}
\begin{proof}
	We recall the construction of the Anderson $A[\underline{t}_{\Sigma}]$-module $G_{\mathcal{C}}$ from \S4.4 and elements $a_l\in A$ coming from the tuples $(a_l,\mathcal{C}_l,u_l)$ in Theorem \ref{T:star0}. By Proposition \ref{P:sec6unif}, we know that $G_{\mathcal{C}}$ is uniformizable. We set $Z_l:=\log_{G_l}(v_{\tilde{\mathcal{C}}_l,\tilde{u}_l})$ and  $v_l:=\exp_{G_l}(Z_l)=v_{\tilde{\mathcal{C}}_l,\tilde{u}_l}\in \Mat_{k_{l} \times 1}(K^{\text{perf}}(\underline{t}_{\Sigma}))$ where the last equality comes from the functional equation \eqref{E:loginv} and the  definition of $Z_l$ makes sense by Theorem \ref{T:log0}. We define \[
	Z_{\mathcal{C}}:=\lambda((\partial_{G_1}(a_1)\cdot Z_1,\dots,\partial_{G_n}(a_n)\cdot Z_n)^{\intercal})\in \Mat_{k_{\mathcal{C}} \times 1}(\TT_{\Sigma})
	\]
	and 
	\[
	v_{\mathcal{C}}:=\lambda((G_1(a_1)\cdot v_1,\dots,G_n(a_n)\cdot v_n)^{\intercal})\in \Mat_{k_{\mathcal{C}} \times 1}(K^{\text{perf}}(\underline{t}_{\Sigma})).
	\]
	Note that by Theorem \ref{T:log0}, the $w$-th coordinate of $Z_l$ is equal to $(-1)^{\dep(\tilde{\mathcal{C}_l})-1}\Li_{\tilde{\tilde{{\mathcal{C}}}}_l}^{*}(\tilde{\tilde{u_l}})=(-1)^{\dep(\mathcal{C}_l)-1}\Li_{\mathcal{C}_l}^{*}(u_l)$. We observe that by the definition of Anderson $A[\underline{t}_{\Sigma}]$-modules $G_l$, for any $a\in A[\underline{t}_{\Sigma}]$ and $f=(f_1,\dots,f_{k_l})^{\intercal}\in \Mat_{k_l \times 1}(\TT_{\Sigma})$, the $w$-th coordinate of $\partial_{G_{l}}(a)\cdot f$ is equal to $af_w$.  Thus, using the definition of the map $\lambda$, we see that the $w$-th coordinate of $Z_{\mathcal{C}}$ is equal to $\sum_{l}a_l(-1)^{\dep(\mathcal{C}_l)-1}\Li_{\mathcal{C}_l}^{*}(u_l)$. But by Theorem \ref{T:star0}, we see that the sum is equal to $\prod_{i\in \mathfrak{I}_1} \ell_{r_{s_i}-1}^{q^{r_{s_i}}-s_i}b_{r_{s_i}}(U_i)\prod_{i\in \mathfrak{I}_2}\Gamma_{s_i}\zeta_{C}(\mathcal{C})$  which proves the first part. 
	
	To prove part (ii), we use the equality \eqref{E:sec41} and Lemma \ref{L:sec6exp} to see that 
	\begin{align*}
	\exp_{G_{\mathcal{C}}}(Z_{\mathcal{C}})&=\exp_{G_{\mathcal{C}}}(\lambda((\partial_{G_1}(a_1)\cdot Z_1,\dots,\partial_{G_n}(a_n)\cdot Z_n)^{\intercal}))\\
	&=\lambda \exp_{\oplus_{l=1}^nG_l}((\partial_{G_1}(a_1)\cdot Z_1,\dots,\partial_{G_n}(a_n)\cdot Z_n)^{\intercal}))\\
	&=\lambda ((G_1(a_1)\cdot \exp_{G_1}(Z_{1}),\dots, G_n(a_n)\cdot \exp_{G_n}(Z_n))^{\intercal})\\
	&=\lambda((G_1(a_1)\cdot v_1,\dots, G_n(a_n)\cdot v_n)^{\intercal})\\
	&=v_{\mathcal{C}}.
	\end{align*}.
\end{proof}
\begin{remark}
	One can observe that we can capture Chang and Mishiba's result \cite[Thm. 1.4.1]{ChangMishibaOct} by defining the composition array $\mathcal{C}$ as in \eqref{E:array1}.
\end{remark}
\begin{example}\label{Ex:Pellarin} Let $\Sigma=\{1,\dots,n\}$ and let $L\binom{\Sigma}{s}$ be the Pellarin $L$-series defined as in \eqref{E:Pellarin}. By Theorem \ref{T:Demeslay}, for any $d\geq0$, there exists a polynomial 
	\[
	Q_{\Sigma,s}(t)=\sum_{l\geq 0} u_lt^l\in K^{\text{perf}}(\underline{t}_{\Sigma})[t]
	\]
	such that 
	\[
	\sum_{a\in A_{+,d}} \frac{a(t_1)\dots a(t_n)}{a^s}=\frac{b_{d}(\Sigma)}{\ell_{r-1}^{q^r-s}\ell_d^sb_r(\Sigma)}\tau^{d}(Q_{\Sigma,s}(t))_{|t=\theta}
	\]
	where $r\geq 1$ is an integer satisfying $q^r\geq s$. Choose $\beta=(t_1-\theta)\dots (t_n-\theta)$ and set $G:=C^{\otimes s}_{\beta}$. For any $l$, we define $v_l:=(0,\dots,0,u_l)$ and $Z_l:=\log_{G}(v_l)$ which is a well-defined element in $\TT_{\Sigma}^{s}$ by Theorem \ref{T:Demeslay} and Proposition \ref{P:log1}. Finally we set $v_{\mathcal{C}}:=\sum_{l\geq 0} G(\theta^l)\cdot v_l\in K^{\text{perf}}(\underline{t}_{\Sigma})^s$ and $Z_{\mathcal{C}}:=\sum_{l\geq 0} \partial_{G}(\theta^l)\cdot Z_l\in \TT_{\Sigma}^s$. Thus by the proof of Theorem \ref{T:result2}, we see that
	\[
	\exp_{G}(Z_{\mathcal{C}})=\exp_{G}((
	*,
	\dots,
	*,
	\ell_{r-1}^{q^r-s}b_r(\Sigma)L(\chi_{t_1}\dots \chi_{t_n},s))^{\intercal})=v_{\mathcal{C}}.
	\]
	
	For the case $n=s=1$, using  \cite[Thm. 4.16]{Per}, we immediately see that $Q_{\Sigma,1}(t)=t_1-t$. Thus, $v_{\mathcal{C}}=t_1-\theta -t_1+\theta =0$ and by \cite[Rem. 5.13, Lem. 6.8, Lem. 7.1]{APTR} (see also \cite[Thm. 1]{Pellarin0}) we have  
	\[
	Z_{\mathcal{C}}=(t_1-\theta)L\binom{\Sigma}{s}=(t_1-\theta)\log_{G}(1)=-\frac{\tilde{\pi}}{\omega_{1}}\in \TT_{\Sigma}(\mathbb{K}_{\infty}).
	\]
	
\end{example}
\begin{remark} We continue with the notation of Example \ref{Ex:Pellarin} and recall the definition of $\TT_{\Sigma}(\mathbb{K}_{\infty})$ from \S2.1. We set 
	\[
	U_{G_{{\mathcal{C}}}}:=\{x\in \Mat_{k_{\mathcal{C}}\times 1}(\TT_{\Sigma}(\mathbb{K}_{\infty})) \ \ | \ \  \exp_{G_{\mathcal{C}}}(x)\in \Mat_{k_{\mathcal{C}} \times 1}(A[\underline{t}_{\Sigma}])   \}.
	\]
	Using the action of $A[\underline{t}_{\Sigma}]$ to $\Mat_{k_{\mathcal{C}} \times 1}(\TT_{\Sigma}(\mathbb{K}_{\infty}))$ by left multiplication, one can see that $U_{G_{{\mathcal{C}}}}$ is an $A[\underline{t}_{\Sigma}]$-module. We call  $U_{G_{{\mathcal{C}}}}$ the unit module (see \cite{AnglesTavaresRibeiro} and \cite{ADTR} for more details). By Example \ref{Ex:Pellarin}, we see that $Z_{\mathcal{C}}\in U_{G_{\mathcal{C}}}$ when $n=s=1$. The situation is more interesting when $n$ is larger. Set $n=q$ and $s=1$. 
	By \cite[Ex. 3.3.7]{DemeslayTh}, we have 
	\[
	Q_{\Sigma,1}(t)=(t_1-t)\dots (t_q-t)\bigg(1-\frac{(t-\theta)}{(t_1-\theta^{1/q})\dots (t_q-\theta^{1/q})}\bigg)\in K^{\text{perf}}(\underline{t}_{\Sigma})[t].
	\]
	A small calculation shows that $v_{\mathcal{C}}\in A[\underline{t}_{\Sigma}]$ and therefore $Z_{\mathcal{C}}\in U_{G_{\mathcal{C}}}$ for $n=q$ and $s=1$. In other words, although the elements $v_l$ in the proof of Theorem  \ref{T:result2} constructing the special point $v_{\mathcal{C}}$ for this case are not in $A[\underline{t}_{\Sigma}]$, $v_{\mathcal{C}}$ is itself in $A[\underline{t}_{\Sigma}]$. It would be interesting to analyze under what conditions  $Z_{\mathcal{C}}$ lies in $U_{G_{\mathcal{C}}}$.
\end{remark}

\begin{appendices}
	\section{The Proof of Theorem \ref{T:sec40}}
	
	Throughout this section we let $G$ be the Anderson $A[\underline{t}_{\Sigma}]$-module of dimension $k$ defined as in \eqref{E:moduleG}. We should also mention that unlike the rest of the paper we use the notation $G_{\theta}$ for the matrix $G(\theta)$ in \eqref{E:moduleG} and $G_{\theta}(f)$ for $G(\theta)\cdot f$ for any $f\in \Mat_{k\times 1}(\TT_{\Sigma})$ in this section.
	\subsection{Operators}Let $\delta_0,\delta_1:\Mat_{1\times d}(\TT_{\Sigma}[\sigma])\to \Mat_{d\times 1}(\TT_{\Sigma})$ be the maps given by
	
	$
	\delta_0\Big(\sum_{i=0}^{n}a_i\sigma^i\Big)=a_0^{\intercal}
	$
	and $
	\delta_1\Big(\sum_{i=0}^{n}a_i\sigma^i\Big)=\sum_{i=0}^na_i^{\intercal (i)}.
	$
	Furthermore for any $f=\sum a_i\tau^i\in \Mat_{k\times d}(\TT_{\Sigma}[\tau])$ we recall the definition of $f^{*}$ in \S 3.3 and define the map $f^{*}:\Mat_{1\times d}(\TT_{\Sigma}[\sigma])\to \Mat_{1\times k}(\TT_{\Sigma}[\sigma])$ by
	\[
	f^{*}(g)=gf^{*}.
	\]
	Then we state the following lemma whose proof can be given similar to the proof of \cite[Lem. 4.2.2]{GP} and \cite[Lem. 1.1.21- 1.1.22]{Juschka10}.
	\begin{lemma}\label{L:Juschka}
		Let $f=\sum_{j=0}^n f_j\tau^j \in \Mat_{k\times d}(\TT_{\Sigma})[\tau]$. 
		\begin{enumerate}
			\item[(a)] Let us define $\partial_0f: \Mat_{d\times 1}(\TT_{\Sigma}) \to \Mat_{k\times 1}(\TT_{\Sigma})$ by $\partial_0f(g) = f_0g$.  The following diagram commutes with exact rows:
			\begin{displaymath}
			\begin{tikzcd}[column sep=small]
			0 \arrow{r} &\Mat_{1\times d}(\TTs[\sigma]) \arrow{r}{\sigma(\cdot)}\arrow{d}{f^{*}}
			&\Mat_{1\times d}(\TTs[\sigma])\arrow{r}{\delta_0}\arrow{d}{f^{*}}
			&\Mat_{d\times 1}(\TTs[\sigma])\arrow{d}{\partial_0 f}\arrow{r}&0\\
			0 \arrow{r} &\Mat_{1\times k}(\TTs[\sigma])\arrow{r}{\sigma(\cdot)}&\ \Mat_{1\times k}(\TTs[\sigma])\arrow{r}{\delta_0}&\Mat_{k\times 1}(\TTs) \arrow{r}&0
			\end{tikzcd}
			\end{displaymath}
			\item[(b)]
			\begin{displaymath}
			\begin{tikzcd}[column sep=small]
			0 \arrow{r} &\Mat_{1\times d}(\TTs[\sigma]) \arrow{r}{(\sigma-1)(\cdot)}\arrow{d}{f^{*}}
			&\Mat_{1\times d}(\TTs[\sigma])\arrow{r}{\delta_1}\arrow{d}{f^{*}}
			&\Mat_{d\times 1}(\TTs)\arrow{d}{f}\arrow{r}&0\\
			0 \arrow{r} &\Mat_{1\times k}(\TTs[\sigma])\arrow{r}{(\sigma-1)(\cdot)}&\Mat_{1\times k}(\TTs[\sigma])\arrow{r}{\delta_1}&\Mat_{k\times 1}(\TTs) \arrow{r}&0
			\end{tikzcd}
			\end{displaymath}
			In particular, we have $G_\theta\delta_1=\delta_1G_\theta^*$.
		\end{enumerate}
	\end{lemma}
	\subsection{Division Towers}
	We start with a definition.
	\begin{definition}For any $x\in \Mat_{k\times 1}(\TT_{\Sigma})$, we call a sequence $\{f_n\}_{n=0}^{\infty}$ in $\Mat_{k\times 1}(\TT_{\Sigma})$ a convergent $\theta$-division tower above $x$ if
		\begin{itemize}
			\item $\lim_{n\to \infty}\dnorm{f_n}=0$.
			\item $G_{\theta}(f_{n+1})=f_n$ for all $n\geq 0$.
			\item $G_{\theta}(f_0)=x$.
		\end{itemize}
	\end{definition}
	We now give the following theorem whose proof uses similar ideas as in the proof of \cite[Thm. 4.3.2]{GP}.
	\begin{theorem}[{cf. \cite[Thm. 4.3.2]{GP}}]\label{T:div} Let $x\in \Mat_{k\times 1}(\TT_{\Sigma})$ Then there exists a canonical bijection
		\[
		F:\{\zeta\in \Mat_{k\times  1}(\TT_{\Sigma}) |  \exp_{G}(\zeta)=x\}\to \{\text{convergent $\theta$-division towers above $x$}\}
		\]
		defined by $F(\zeta)=\{\exp_{G}(\partial_{G}(\theta)^{-(n+1)}\zeta)\}_{n=0}^{\infty}$.
		Moreover, if $\{f_n\}_{n=0}^{\infty}$ is a convergent $\theta$-division tower above $x$, then with respect to $\dnorm{\cdot}$, we have
		$
		\lim_{n\to \infty}\partial_{G}(\theta)^{n+1}f_n=\zeta.
		$
	\end{theorem}
	\begin{proof}
		Note that by the functional equation \eqref{E:sec41} we have
		\[
		G_{\theta}(\exp_{G}(\partial_G(\theta)^{-(n+1)}\zeta))=\exp_{G}(\partial_G(\theta)^{-n}\zeta)
		\]	
		and $G_{\theta}(\exp_{G}(\partial_{G}(\theta)^{-1}\zeta))=\exp_{G}(\zeta)=x$. We also see by Lemma \ref{L:sec4iso} that for arbitrarily large $n$, we have $\dnorm{\exp_{G}(\partial_G(\theta)^{-n}\zeta)}=\dnorm{\partial_{G}(\theta)^{-n}\zeta}$. So the sequence given as $F(\zeta)$ converges to 0 and is actually a $\theta$-division sequence above $x$. Thus the map $F$ is well-defined. For the injectivitiy, let us assume that $\exp_{G}(\partial_{G}(\theta)^{-(n+1)}\zeta_1)=\exp_{G}(\partial_{G}(\theta)^{-(n+1)}\zeta_2)$ for some $\zeta_1,\zeta_2\in \Mat_{k\times  1}(\TT_{\Sigma})$ and any $n\in \ZZ_{\geq 0}$. Then we have $\exp_{G}(\partial_{G}(\theta)^{-(n+1)}(\zeta_1-\zeta_2))=0$. But by Lemma \ref{L:sec4iso} we see that $\partial_{G}(\theta)^{-(n+1)}(\zeta_1-\zeta_2)$ should be equal to zero matrix for sufficiently large $n$. Thus, one can deduce by the invertibility of $\partial_{G}(\theta)$ that $\zeta_1=\zeta_2$. 
		
		We prove the surjectivity as follows. Let $\{f_n\}_{n=0}^{\infty}$ be a convergent $\theta$-division tower above $x$. By the convergence of the sequence, there exists $N\in \mathbb{Z}_{\geq 0}$ such that for any $n\geq N$, $f_n$ is in the radius of convergence of $\log_{G}$. Now we set $\zeta=\partial_{G}(\theta^{n+1})\log_{G}(f_n)$ for any $n\geq N$. Then by \eqref{E:sec45} we have
		\[
		\partial_{G}(\theta^{n+2})\log_{G}(f_{n+1})=\partial_{G}(\theta^{n+1})\log_{G}(G_{\theta}(f_{n+1}))=\partial_{G}(\theta^{n+1})\log_{G}(f_n).
		\]
		Thus, our choice for $\zeta$ is independent of $n$. Therefore we have $f_{n}=\exp_{G}(\partial_{G}(\theta^{n+1})^{-1}\zeta)$. Then for any $n<N$, we obtain by using \eqref{E:sec43} that
		\[
		f_n=G_{\theta^{N-n}}(f_N)=G_{\theta^{N-n}}(\exp_{G}(\partial_{G}(\theta^{N+1})^{-1}\zeta))=\exp_{G}(\partial_{G}(\theta^{n+1})^{-1}\zeta).
		\]
		Thus we see that $F(\zeta)=\{f_n\}_{n=0}^{\infty}$.	For the last assertion, we observe that for any $n\geq 0$, $\zeta-\partial_{G}(\theta^{n+1})\exp_{G}(\partial_{G}(\theta^{n+1})^{-1}\zeta)=\partial_{G}(\theta)^{n+1}\sum_{j\geqslant 1}\beta_j\partial_{G}(\theta)^{-q^j(n+1)}\zeta^{(j)}$ where $\exp_{G}=\sum_{j\geq 0}\beta_j \tau^j$. Notice that for each $j$, $\dnorm{\partial_{G}(\theta^{n+1})\beta_j\partial_{G}(\theta)^{-q^j(n+1)}\zeta^{(j)}}\to 0$ as $n\to \infty$ because by Proposition \ref{P:sec40} we see that $\lim_{j\to \infty}\dnorm{\beta_j} R^{q^j}=0$ for any $R\in \mathbb{R}_{>0}$. Thus $\lim_{n \to \infty} \dnorm{\zeta - \partial_{G}(\theta^{n+1})f_n}=0$.	
	\end{proof}

	\begin{remark}\label{R:map}
		We recall the definition of the row matrix $\bold{m}$ from \S4.1. By \eqref{E:claim222}, for any $1\leq j \leq r $, we have 
		\begin{equation}\label{E:sigmaaction2}
		\frac{(t-\theta)^{d_j}}{\prod_{i=j}^r\alpha_{i}^{(-1)}}\cdot m_j=\sigma m_j+\sum_{w=1}^{j-1}(-1)^w\prod_{n=1}^{w}u_{j-n}^{(-1)}\sigma m_{j-w}.
		\end{equation}
		We now consider the map $\delta_0\circ \iota:(\Mat_{1\times r}(\TT_{\Sigma}[t]),\lVert \cdot \rVert_{\theta})\to (\Mat_{k \times 1}(\TT_{\Sigma}),\dnorm{\cdot})$. Let $h=[h_1,\dots,h_r]\in \Mat_{1 \times r}(\TT_{\Sigma}[t])$ be such that $h_i=\sum_{j=0}^{\infty}h_{ij}(t-\theta)^j$ and $h_{ij}=0$ for $j\gg 0$. By the identification of $\TT_{\Sigma}[\sigma]$-basis of $H(G)$ as vectors $e_{ij}$ as in Remark \ref{R:sec4} and \eqref{E:sigmaaction2} we see that 
		\begin{equation}\label{E:map2}
		\begin{split}
		\delta_0\circ \iota(h)&=\delta_0(h_1\cdot m_1+\dots +h_r\cdot m_r)\\
		&=(h_{1(d_1-1)},\dots,h_{10},\dots,h_{r(d_r-1)},\dots,h_{r0}).
		\end{split}
		\end{equation}
	\end{remark}
	\begin{lemma}\label{L:map2} Let $h\in \TT_{\Sigma}[t]$ be such that $h=\sum_{j=0}^{l}h_{j}(t-\theta)^j$ for some $l\in \mathbb{Z}_{\geq 0}$. Then
		\[
		\lVert h\rVert_{\theta}=\sup\{\dnorm{h_i}\dnorm{\theta}^i \ \ | \ \ i\in \ZZ_{\geq 0}\}.
		\]
	\end{lemma}
	\begin{proof}
		Assume that $h=\sum_{j=0}^{l}h_{j}(t-\theta)^j=\sum_{j=0}^{l}g_jt^j$ for some $g_j\in \TT_{\Sigma}$. By the assumption on $g_j$ and $h_j$, we see that
		$
		g_j=h_j-\binom{j+1}{j}h_{j+1} \theta^{j+1-j}+\dots - \binom{l}{j}h_{l} \theta^{l-j}.
		$
		Thus we obtain 
		\[\lVert  h\rVert_{\theta}=\sup_{j} \{\dnorm{g_j}\dnorm{\theta}^j\}\leq \sup_{i} \{\dnorm{h_i}\dnorm{\theta}^i\}.
		\]
		On the other hand, again by the assumption, we have that $h_j=\sum_{i=j}^l\binom{i}{j}g_i\theta^{i-j}.$ Thus similarly we have
		\[\sup \{\dnorm{h_j}\dnorm{\theta}^j\}\leq \sup_{j}(\sup_{j\leq i\leq l}\{\dnorm{g_i}\dnorm{\theta}^i\})=\lVert  h\rVert_{\theta}
		\]
		which completes the proof.
	\end{proof}
	Thus, using \eqref{E:map2} and Lemma \ref{L:map2}, we obtain $\dnorm{\delta_0\circ \iota(h)}\leq \lVert h \rVert_{\theta}$. Therefore the map $\delta_0\circ \iota$ is bounded. By the fact that $\Mat_{1\times r}(\TT_{\Sigma}[t])$ is $\lVert \cdot \rVert_{\theta}$-dense in $\Mat_{1\times r}(\TT_{\Sigma}\{t/\theta \})$, we extend the map $\delta_0\circ \iota$ to a map $D:(\Mat_{1\times r}(\TT_{\Sigma}\{t/\theta\}), \lVert \cdot \rVert_{\theta})\to (\Mat_{k\times 1}(\TT_{\Sigma}),\dnorm{\cdot})$ of complete normed modules.
	\begin{theorem}[{cf. \cite[Thm. 4.4.6]{GP}}]\label{T:Unif} Let $(\iota,\Phi)$ be a $t$-frame for $G$. Moreover let $h\in \Mat_{1\times r}(\TT_{\Sigma})[t]$ and assume that there exists a matrix $g\in \Mat_{1\times r}(\TT_{\Sigma}\{t/\theta\})$ such that $g^{(-1)}\Phi-g=h$. If $v=\delta_1(\iota(h))\in \Mat_{k\times 1}(\TT_{\Sigma})$ and $\zeta=D(g+h)$, then we have $\exp_{G}(\zeta)=v$.
	\end{theorem}
	\begin{proof}
		The proof follows the ideas of the proof of Theorem 4.4.6 of \cite{GP}. We first let $g=\sum_{i=0}^{\infty}$ where $g_i\in \Mat_{1\times r}(\TT_{\Sigma})$ and define $g_{\leq n}=\sum_{i\leq n}g_it^i$ and $g_{>n}=\sum_{i>n}g_it^i$ for $n\geq 0$. Furthermore we set
		\begin{equation}\label{E:proof1}
		h_n:=\frac{h+g_{\leq n}-g_{\leq n}^{(-1)}\Phi}{t^{n+1}}=\frac{g_{> n}^{(-1)}\Phi-g_{>n}}{t^{n+1}}\in \Mat_{1 \times r}(\TT_{\Sigma}[[t]]).
		\end{equation}
		Observe that since $g^{(-1)}\Phi-g=h$, the second expression in \eqref{E:proof1} is a power series in $t$ and divisible by $t^{n+1}$. But since $\deg_{t}(g_{\leq n})\leq n$, $h_n$ should be a polynomial in $t$ and therefore $h_n\in \Mat_{1 \times r}(\TT_{\Sigma}[t])$. Moreover, $\deg_{t}(h_n)\leq \max\{\deg_{t}(h)-n-1,0\}$. Thus the degree of $h_n$ in $t$ does not depend on $n$ and therefore $h_n$ can be seen as an element of a free and finitely generated $\TT_{\Sigma}$-module $M$ of $\Mat_{1\times r}(\TT_{\Sigma}[t])$. We now prove several claims.
		
		\textit{Claim 1}: The sequence $\{\delta_1(\iota(h_n))\}_{n=0}^{\infty}$ is a convergent $\theta$-division tower above $\delta_1(\iota(h))$.
		
		\begin{proof} Using Lemma \ref{L:sec4unif} and Lemma \ref{L:Juschka}, we see that
			\begin{equation}\label{E:proof2}
			\begin{split}
			\delta_1(\iota(h_n))-G_{\theta}(\delta_1(\iota(h_{n+1})))&=\delta_1(\iota(h_n))-\delta_1(G_{\theta}^{*}(\iota(h_{n+1})))\\
			&=\delta_1(\iota(h_n))-\delta_1(\iota(h_{n+1})G_{\theta}^{*})\\
			&=\delta_1(\iota(h_n)-th_{n+1}).
			\end{split}
			\end{equation}
			From the definition of $h_n$ and using Lemma \ref{L:sec4unif} we obtain
			\begin{equation}\label{E:proof3}
			\begin{split}
			\delta_1(\iota(h_n-th_{n+1}))=\delta_1\iota\bigg(\Big(\frac{g_{n+1}}{t^{n+1}}\Big)^{(-1)}-\frac{g_{n+1}}{t^{n+1}}\bigg)=\delta_1\bigg((\sigma-1)\iota\Big(\frac{g_{n+1}}{t^{n+1}}\Big)\bigg)=0
			\end{split}
			\end{equation}
			where the last equality follows from Lemma \ref{L:Juschka}(b). Thus \eqref{E:proof2} and \eqref{E:proof3} imply that $\delta_1(\iota(h_n))=G_{\theta}(\delta_1(\iota(h_{n+1})))$ for $n\geq 0$. Similar calculation as above also shows that $G_{\theta}(\iota(h_0))=v=\delta_1(\iota(h))$.
			
			Recall the definition of the norm $\lVert \cdot \rVert_{\sigma}$ from \S3.3 and the norm $\lVert \cdot \rVert_{1}$ from \S4.2 to observe that since $\dnorm{g_{n}}\to 0$ as $n\to \infty$, we obtain $\lVert g_{>n} \rVert_{1}\to 0$ as $n\to \infty$. Moreover we have
			\[
			\lVert h_n \rVert_{1}\leq \max\{\lVert g_{>n}^{(-1)}\Phi \rVert_{1}, \lVert g_{>n} \rVert_{1} \}=\max\{\lVert g_{>n} \rVert^{1/q}_{1}\lVert \Phi \rVert_{1} ,\lVert g_{>n} \rVert_{1}\}.
			\]
			Thus $\lVert h_n \rVert_{1}\to 0$ as $n\to \infty$. By \cite[Lem. 2.2.2]{GP}, the norms $\lVert \cdot \rVert_{1}$ and $\lVert \iota(\cdot) \rVert_{\sigma}$ are equivalent on $M$ and therefore $\lVert \iota(h_n) \rVert_{\sigma}\to 0$ when $n\to \infty$. Since the $t$-degree of $h_n$ is bounded indepedent of $n$, we can also see that the $\sigma$-degree of $\iota(h_n)$ is also bounded and say for arbitrarily large $n$, $\iota(h_n)=\sum_{j=0}^{N}a_j\sigma^{j}$ such that $\lVert a_j \rVert_{\sigma}<1$. Thus we have
			\begin{equation}\label{E:proof4}
			\dnorm{\delta_1(\iota(h_n))}=\dnorm{\sum_{j=0}^Na_j^{\intercal (j)}}\leq \sup\{\dnorm{a_j}\}=\lVert \iota(h_n) \rVert_{\sigma}.
			\end{equation} 
			Thus \eqref{E:proof4} implies that $\dnorm{\delta_1(\iota(h_n))}\to 0$ as $n\to \infty$ and therefore the sequence $\{\delta_1(\iota(h_n))\}_{n=0}^{\infty}$ is a convergent $\theta$-division tower above $v=\delta_1(\iota(h))$.
		\end{proof}
		\textit{Claim 2}: We have $\lim_{n\to \infty}\dnorm{\partial_{G}(\theta^{n+1})}\lVert h_n \rVert_{\theta}=0$.
		\begin{proof}
			Let us set $h_n=\sum_{i=0}^{N_0}c_it^i$ where $c_i\in \Mat_{1\times r}(\TT_{\Sigma})$ and by the discussion in the beginning of the proof we know the existence of some positive integer $N_0$ which is independent of $n$. We have from the definition of $h_n$ that
			\begin{equation}\label{E:proof5}
			\begin{split}
			\dnorm{\partial_{G}(\theta^{n+1})}\lVert h_n \rVert_{\theta}&=\dnorm{\partial_{G}(\theta)}^{n+1}\lVert \frac{g_{>n}^{(-1)}\Phi-g_{>n}}{t^{n+1}} \rVert_{\theta}\\
			&=\dnorm{\partial_{G}(\theta)}^{n+1}\sup \{\dnorm{\theta}^i\lVert c_i \rVert_{1}\}\\
			&=\sup \{\dnorm{\theta}^{n+1+i}\lVert c_i \rVert_{1}  \}\\
			&=\lVert c_0t^{n+1}+\dots+c_{N_0}t^{n+N_0+1} \rVert_{\theta}\\
			&=\lVert g_{>n}^{(-1)}\Phi-g_{>n} \rVert_{\theta}.
			\end{split}
			\end{equation} 	
			But observe that $g\in \Mat_{1\times r}(\TT_{\Sigma}\{t/\theta\}$ so $\lVert g_{>n} \rVert_{\theta}\to 0$ as $n\to \infty$ which implies the claim together with \eqref{E:proof5}.
		\end{proof}
		\textit{Claim 3}: We have $\lim_{n\to \infty}\partial_{G}(\theta^{n+1})\delta_1(\iota(h_n))=\zeta$.
		
		\begin{proof}
			Using the definition of $h_n$, Lemma \ref{L:sec4unif} and Lemma \ref{L:Juschka}, observe that 
			\begin{align*}
			\lim_{n\to \infty}\delta_0(\iota(t^{n+1}h_n+g_{\leq n}^{(-1)}\Phi))&=\lim_{n\to \infty}\delta_0(\iota(h_n)G_{\theta^{n+1}}^{*})\\
			&=\lim_{n \to \infty}\partial_{G}(\theta^{n+1})\delta_0(\iota(h_n)).
			\end{align*}
			Thus using the definition of $\zeta$, we see that $\zeta=\lim_{n \to \infty}\partial_{G}(\theta^{n+1})\delta_0(\iota(h_n))$. Therefore we need to show that $\lim_{n\to \infty}\partial_{G}(\theta^{n+1})(\delta_1(\iota(h_n))-\delta_0(\iota(h_n)))=0$. Observe that 
			\begin{equation}\label{E:proof6}
			\dnorm{\delta_1(\iota(h_n))-\delta_0(\iota(h_n))}\leq \lVert \sum_{j=1}^{N}a_j^{\intercal (j)} \rVert_{\infty}\leq \sup \{\dnorm{a_j}^{q^j}\}\leq \lVert \iota(h_n) \rVert^q_{\sigma}.
			\end{equation}
			Thus \eqref{E:proof6} implies that the claim is equivalent to showing that $\lim_{n\to \infty}\partial_{G}(\theta^{n+1})\lVert \iota(h_n) \rVert^q_{\sigma}=0$. Since for sufficiently large $n$, $\lVert \iota(h_n) \rVert_{\sigma}\leq 1$, we have to show that $\lim_{n\to \infty}\partial_{G}(\theta^{n+1})\lVert \iota(h_n) \rVert_{\sigma}$. By the equivalence of the norms $\lVert \cdot \rVert_{\theta}$ and $\lVert \iota(\cdot) \rVert_{\sigma}$ on $M$, the claim follows from Claim 2.	
		\end{proof}
		Now by Claim 1, Claim 3 and Theorem \ref{T:div} we see that $\exp_{G}(\zeta)=\delta_1(\iota(h))=v$.
	\end{proof}

	\begin{proof}[Proof of Theorem \ref{T:sec40}]
		Let us choose an arbitrary element
		\[
		\tilde{h}=[h_{1(d_1-1)},\dots,h_{10},\dots,h_{r(d_r-1)},\dots,h_{r0}]^{\intercal}\in \Mat_{k\times 1}(\TT_{\Sigma})
		\]
		and let
		$
		h=[\sum_{j=0}^{d_1-1}h_{1j}(t-\theta)^j,\dots,\sum_{j=0}^{d_r-1}h_{rj}(t-\theta)^j] \in \Mat_{1\times r}(\TT_{\Sigma}[t]).
		$
		Since $\Mat_{1 \times r}(\TT_{\Sigma}[t])$ is $\lVert \cdot \rVert_{\theta}$-dense in $\Mat_{1 \times r}(\TT_{\Sigma}\{t/\theta\})$, we can write $h\Psi=u+h$ so that $u \in \Mat_{1 \times r}(\TT_{\Sigma}[t])$ and $\lVert h \rVert_{\theta} <1$. We also have that $\lVert h^{(n)}\rVert_{\theta} \leq \lVert h\rVert_{\theta}^{q^n}$ holds for all $n \geq 0$. Then the series $H:=\sum\limits_{n=1}^{\infty}h^{(n)}$ converges to an element of $\Mat_{1 \times r}(\TT_{\Sigma}\{t/\theta\})$ because $\lVert h\rVert_{\theta} < 1$. Moreover, $H^{(-1)}-H=h$. By Proposition \cite[Prop. 4.5.2]{GP}, there exists $U \in \Mat_{1 \times r}(\TT_{\Sigma}[t])$ such that $U^{(-1)}-U=u$. Set $x:=(U+H)\Psi^{-1}$. Then one can see that
		\begin{align*}
		x^{(-1)}\Phi -x=(u+h)\Psi^{-1}=h.	
		\end{align*}
		Using Remark \ref{R:sec4} and Remark \ref{R:map} one can observe that $\delta_1(\iota(h))=\tilde{h}$. Therefore by Theorem \ref{T:Unif} we have that $\exp_{G}(D(x+h))=\delta_1(\iota(h))=\tilde{h}$. So $\exp_{G}$ is surjective.
	\end{proof}

\end{appendices}

\end{document}